\documentclass{article}
\usepackage{amssymb,amsmath,amscd}
\usepackage[amsmath,thmmarks]{ntheorem}
\theoremseparator{.}
\theoremsymbol{\rule{1ex}{1ex}}
\theorembodyfont{\upshape}
\newtheorem{definition}{Definition}[section]
\newtheorem{remark}[definition]{Remark}

\newtheorem{example}[definition]{Example}

\newtheorem{problem}[definition]{Problem}
\newtheorem*{proof}{Proof}

\theorembodyfont{\itshape}
\newtheorem{Lemma}[definition]{Lemma}

\newtheorem{Theorem}[definition]{Theorem}

\theoremnumbering{Alph}
\newtheorem{inttheorem}{Theorem}

\theoremsymbol{}
\newtheorem{lemma}[definition]{Lemma}
\newtheorem{proposition}[definition]{Proposition}
\newtheorem{theorem}[definition]{Theorem}
\newtheorem{corollary}[definition]{Corollary}
\usepackage[a4paper,hscale=0.8,vscale=0.8,hmarginratio=1:1,includehead,headheight=14pt]{geometry}
\usepackage{fancyhdr}
\usepackage{amssymb,amsmath,amscd,dsfont}
\usepackage[all,cmtip,line]{xy}
\usepackage{slashed}
\usepackage{colonequals}

\newenvironment{tabsection}{}{}
\newcommand{\argument}{\ensuremath{\mathinner{\;\cdot\;}}}
\newcommand{\op}[1]{\ensuremath{\operatorname{#1}}}
\newcommand{\wt}[1]{\ensuremath{\widetilde{#1}}}
\newcommand{\wh}[1]{\ensuremath{\widehat{#1}}}
\newcommand{\ol}[1]{\ensuremath{\overline{#1}}}

\newcommand{\cB}{\ensuremath{\mathcal{B}}}
\newcommand{\cC}{\ensuremath{\mathcal{C}}}

\newcommand{\cG}{\ensuremath{\mathcal{G}}}
\newcommand{\cH}{\ensuremath{\mathcal{H}}}

\newcommand{\cL}{\ensuremath{\mathcal{L}}}

\newcommand{\cP}{\ensuremath{\mathcal{P}}}

\newcommand{\cR}{\ensuremath{\mathcal{R}}}

\newcommand{\cU}{\ensuremath{\mathcal{U}}}
\newcommand{\cV}{\ensuremath{\mathcal{V}}}
\newcommand{\cW}{\ensuremath{\mathcal{W}}}

\newcommand{\fk}{\ensuremath{\mathfrak{k}}}

\newcommand{\fa}{\ensuremath{\mathfrak{a}}}
\newcommand{\fg}{\ensuremath{\mathfrak{g}}}

\newcommand{\fh}{\ensuremath{\mathfrak{h}}}

\newcommand{\mf}[1]{\ensuremath{\mathfrak{#1}}}

 \newcommand{\Q}{\ensuremath{\mathbb{Q}}}

 \newcommand{\R}{\ensuremath{\mathbb{R}}}
 \newcommand{\F}{\ensuremath{\mathbb{F}}}
 \newcommand{\N}{\ensuremath{\mathbb{N}}}
 \newcommand{\Z}{\ensuremath{\mathbb{Z}}}

 \newcommand{\T}{\ensuremath{\mathbb{T}}}

\newcommand{\bS}{\ensuremath{\mathbb{S}}}

\newcommand{\id}{\ensuremath{\operatorname{id}}}
\newcommand{\pr}{\ensuremath{\operatorname{pr}}}
\newcommand{\ev}{\ensuremath{\operatorname{ev}}}

\newcommand{\Ad}{\ensuremath{\operatorname{Ad}}}

\newcommand{\cat}[1]{\ensuremath{\mathsf{\mathop{#1}}}}

\newcommand{\Aut}{\ensuremath{\operatorname{Aut}}}

\newcommand{\Diff}{\ensuremath{\operatorname{Diff}}}

\newcommand{\im}{\ensuremath{\operatorname{im}}}

\newcommand{\SO}{\ensuremath{\operatorname{SO}}}

\newcommand{\res}{\ensuremath{\operatorname{res}}}

\newcommand{\per}{\ensuremath{\operatorname{per}}}

\newcommand{\se}{\ensuremath{\nobreak\subseteq\nobreak}}
\newcommand{\from}{\ensuremath{\nobreak\colon\nobreak}}
\renewcommand{\to}{\ensuremath{\nobreak\rightarrow\nobreak}}
\newcommand{\toto}{\ensuremath{\nobreak\rightrightarrows\nobreak}}

\newcommand{\coloneq}{\colonequals}
\DeclareMathOperator{\A}{\Sigma}

\DeclareMathOperator{\evol}{evol}
\DeclareMathOperator{\Evol}{Evol}
\DeclareMathOperator{\one}{{\bf 1}}

\newcommand\opn{\ensuremath{\mathrel{\mathpalette\opncls\circ}}}

\newcommand{\opncls}[2]{
  \ooalign{$#1\subseteq$\cr
  \hidewidth\raisefix{#1}\hbox{$#1{\stylefix{#1}#2}\mkern2mu$}\cr}}
\def\raisefix#1{
  \ifx#1\displaystyle
    \raise.39ex
  \else
    \ifx#1\textstyle
      \raise.39ex
    \else
      \ifx#1\scriptstyle
        \raise.275ex
      \else
        \raise.150ex
      \fi
    \fi
  \fi
}
\def\stylefix#1{
  \ifx#1\displaystyle
    \scriptstyle
  \else
    \ifx#1\textstyle
      \scriptstyle
    \else
      \ifx#1\scriptstyle
        \scriptscriptstyle
      \else
        \scriptscriptstyle
      \fi
    \fi
  \fi
}
\DeclareFontFamily{U}{mathx}{\hyphenchar\font45}
\DeclareFontShape{U}{mathx}{m}{n}{
      <5> <6> <7> <8> <9> <10>
      <10.95> <12> <14.4> <17.28> <20.74> <24.88>
      mathx10
      }{}
\DeclareSymbolFont{mathx}{U}{mathx}{m}{n}
\DeclareFontSubstitution{U}{mathx}{m}{n}
\DeclareMathAccent{\widecheck}{0}{mathx}{"71}
\DeclareMathAccent{\wideparen}{0}{mathx}{"75}
\usepackage{graphicx}%
\usepackage{mathtools}
\usepackage[all,cmtip,line]{xy}%
\CompileMatrices%
\newcommand{\xyhookrightarrow}[1]{\ar@{}[r]|-*[@]{\xhookrightarrow{\hphantom{\hspace{#1}}}}}
\newcommand{\xyhookrrightarrow}[1]{\ar@{}[rr]|-*[@]{\xhookrightarrow{\hphantom{\hspace{#1}}}}}
\newcommand{\Symp}{\ensuremath{\op{Symp}}}
\newcommand{\symp}{\ensuremath{\op{\mf{symp}}}}
\newcommand{\ham}{\ensuremath{\op{\mf{ham}}}}
\newcommand{\TGamma}[1]{\ensuremath{\mathbb{T}_{#1}}}
\usepackage[normalem]{ulem}
\usepackage{enumitem}%
\setlist[enumerate]{label={\alph*})}%
\usepackage{hyperref}%
\hypersetup{
 urlcolor = red,
 colorlinks = true,
 linkcolor = blue,
 citecolor = blue,
 linktocpage = true,
 pdftitle = {(Re)constructing Lie groupoids from their bisections and applications to prequantisation},
 pdfauthor = {Alexander Schmeding, Christoph Wockel},
 bookmarksopen = true,
 bookmarksopenlevel = 1,
 unicode = true,
 hypertexnames =false 
}%
\newcommand{\Bis}{\ensuremath{\op{Bis}}}

\newcommand{\Bisf}[1]{\Bis_{#1}}
\newcommand{\Lf}{\ensuremath{\mathbf{L}}}
\newcommand{\Stab}[1]{\ensuremath{K_{#1}}}
\newcommand{\Vtx}[1]{\ensuremath{\op{Vert}_{#1}}}
\newcommand{\Loop}[1]{\ensuremath{\op{Loop}_{#1}}}
\DeclareMathOperator*{\Ker}{ker}

\begin{document}

\begin{flushright}
   {\sf ZMP-HH/15-15}\\
   {\sf Hamburger$\;$Beitr\"age$\;$zur$\;$Mathematik$\;$Nr.$\;$552}\\[2mm]
\end{flushright}

\title{(Re)constructing Lie groupoids from their bisections and applications to prequantisation}
\author{Alexander Schmeding\footnote{NTNU Trondheim, Norway
\href{mailto:alexander.schmeding@math.ntnu.no}{alexander.schmeding@math.ntnu.no}}~
and Christoph
Wockel\footnote{University of Hamburg, Germany
\href{mailto:christoph@wockel.eu}{christoph@wockel.eu}}} {\let\newpage\relax\maketitle}

\begin{abstract}
  This paper is about the relation of the geometry of Lie groupoids over a fixed
  compact manifold $M$ and the geometry of their (infinite-dimensional)
  bisection Lie groups. In the first part of the paper we investigate the
  relation of the bisections to a given Lie groupoid, where the second part is
  about the construction of Lie groupoids from candidates for their bisection
  Lie groups. The procedure of this second part becomes feasible due to some
  recent progress in the infinite-dimensional Frobenius theorem, which we heavily
  exploit. The main application to the prequantisation of (pre)symplectic
  manifolds comes from an integrability constraint of closed Lie subalgebras to
  closed Lie subgroups. We characterise this constraint in terms of a modified
  discreteness conditions on the periods of that manifold.
\end{abstract}

\medskip


\medskip

(secondary)

\tableofcontents

\section*{Introduction}

\begin{tabsection}
 The Lie group structure on bisection Lie groups was established in
 \cite{Rybicki02A-Lie-group-structure-on-strict-groups,SchmedingWockel14},
 along with a smooth action of the bisections on the arrow manifold of a Lie
 groupoid. Is this paper, we develop a tight relation between Lie groupoids and
 their bisection Lie groups by making use of this action. As a first step, we
 show how this action can be used to reconstruct a Lie groupoid from its
 bisections. In general, this will not be possible, since there may not be
 enough bisections for a reconstruction to work. However, under mild
 assumptions on the Lie groupoid, e.g., the groupoid being source connected,
 the reconstruction works quite well. It is worthwhile to note that in this
 chapter the analytical tools we need are of moderate complexity, all results
 follow from a thorough usage of the concept of a submersion.
 
 In an intermediate step, we analyse in some more detail the structure that
 bisections have in addition to merely being Lie groups. For instance, they act
 on $M$ and the stabilisers are closely related to the vertex groups of the Lie
 groupoid. To this end, we use the results of Gl\"ockner on submersion
 properties in infinite-dimensions \cite{Glokcner06Implicit,hg2015}.
 
  The next step is then to take the insights from the previous sections in order
  to formulate structures and conditions on a Lie group that turn it into a Lie
  group of bisections of a Lie groupoid. This becomes analytically more
  challenging, and the procedure is only possible due to a heavy usage of the
  recent results from \cite{hg2015}, building on a generalised Frobenius Theorem
  of integrable co-Banach distributions
  \cite{Eyni14,Hiltunen00}.
  The result is the concept of a transitive pair, which describes in an
  efficient way the necessary structure that is needed on a Lie group in order
  to relate it in a natural way to the bisection group of a Lie groupoid.
  However, we restrict here to transitive (respectively locally trivial) Lie
  groupoids, the general theory will be part of future research.
 
 In the final section we then apply the integration theory for abelian
 extensions from
 \cite{Neeb04Abelian-extensions-of-infinite-dimensional-Lie-groups} in order to
 derive a transitive pair from an integration of an extension of $\cV(M)$,
 given by a closed 2-form $\omega\in \Omega^{2}(M)$ to an extension of
 $\Diff(M)_{0}$. The crucial point here is the integration of a certain Lie
 subalgebra to a closed Lie subgroup of the integrated extension, for which we
 derive a new discreteness condition in terms of the associated period
 groups.\medskip
 
 We now go into some more detail and explain the main results. Suppose
 $\cG = (G \toto M)$ is a Lie groupoid. This means that $G,M$ are smooth
 manifolds, equipped with submersions $\alpha,\beta\from G\to M$ and an
 associative and smooth multiplication $G\times _{\alpha,\beta}G\to G$ that
 admits a smooth identity map $1\from M\to G$ and a smooth inversion
 $\iota\from G\to G$. Then the bisections $\Bis(\cG)$ of $\cG$ are the sections
 $\sigma\from M\to G$ of $\alpha$ such that $\beta \circ \sigma$ is a
 diffeomorphism of $M$. This becomes a group with respect to
 \begin{equation*}
  (\sigma \star \tau ) (x) \coloneq \sigma ((\beta \circ \tau)(x))\tau(x)\text{ for }  x \in M.
 \end{equation*}
 This group is an (infinite-dimensional) Lie group (cf.\
 \cite{SchmedingWockel14}) if $M$ is compact, $G$ is modelled on a metrisable
 space and the groupoid $\cG$ admits a local addition adapted to the source
 projection $\alpha$, i.e.\ it restricts to a local addition on each fibre
 $\alpha^{-1} (x)$ for $ x \in M$ (cf.\ \cite{michor1980,SchmedingWockel14}).
 
 By construction of the Lie group structure, we obtain a natural action
 $  \Bis (\cG) \times M \rightarrow M$, $(\sigma,m) \mapsto \beta (\sigma   (m)) $
 of $\Bis (\cG)$ on $M$. This action gives rise to the bisection action groupoid
 $\cB (\cG) \coloneq (\Bis (\cG) \ltimes M \toto M)$. Observe that the action
 is constructed from the the joint evaluation map
 \begin{displaymath}
  \ev \colon \Bis (\cG) \times M \rightarrow G,\quad (\sigma,m)\mapsto \sigma
  (m)
 \end{displaymath}
 and the target projection of the groupoid $\cG$. While the target projection
 is a feature of the groupoid $\cG$, the evaluation map yields a groupoid
 morphism over $M$ from $\cB (\cG)$ to $\cG$. Thus information on the groupoid
 can be recovered from the group of bisections and the base manifold via the
 joint evaluation map. The idea is to recover the smooth structure of the arrow
 manifold from the evaluation map based on the following result: \medskip
 
 \begin{inttheorem}\label{thm:A}
  {Let $\cG = (G \toto M)$ be a Lie groupoid, where $M$ is compact, $G$ is
  modelled on a metrisable space and $\cG$ admits an $\alpha$-adapted local
  addition. Then the joint evaluation $\ev $ is a submersion.}
  {If $\cG$ is in addition source-connected, i.e.\ the source fibres are
  connected manifolds, then $\ev$ is surjectiv.}
 \end{inttheorem}
 
 Note that as an interesting consequence of Theorem \ref{thm:A} we obtain also
 information on the evaluation of smooth maps from a compact manifold $M$ into
 a manifold $N$ modelled on a metrisable space. In this case, the evaluation
 map
 \begin{displaymath}
  \ev \colon C^\infty (M,N) \times M \rightarrow N,\quad (f,m) \mapsto f(m)
 \end{displaymath}
 is a surjective submersion.
 
 A crucial point of our approach will be that the joint evaluation map is a
 surjective submersion. Unfortunately, this will not be the case in general as
 there are Lie groupoids without enough bisections (see Remark \ref{rem:
 ev:sur} b) for an example). In this case there is no hope to recover the
 manifold of arrows, whence not all information on the groupoid is contained in
 its group of bisections. However, it turns out that at least the identity
 subgroupoid can always be reconstructed. Moreover, Theorem \ref{thm:A} still
 gives a sufficient criterion for the surjectivity of $\ev$. It is sufficient
 that the Lie groupoid is source connected, and in general this condition can
 not be dispensed with. As a byproduct, we obtain generalisations of some
 results about the existence of global bisections through each point (cf.\
 \cite{MR2511542}) to infinite dimensions. Hence we can reconstruct the
 groupoid from its bisections and the base manifold for a fairly broad class of
 Lie groupoids. Namely, we obtain the following reconstruction result.\medskip
 
 \begin{inttheorem}\label{thm:B}
  {Let $\cG = (G \toto M)$ be a Lie groupoid, where $M$ is compact, $G$ is
  modelled on a metrisable space and $\cG$ admits an $\alpha$-adapted local
  addition. If the joint evaluation map $\ev$ is surjective, e.g.\ $\cG$ is
  source-connected, then the Lie groupoid morphism
  \begin{displaymath}
   \ev\colon \cB (\cG) = (\Bis (\cG) \ltimes M \toto M) \rightarrow \cG
  \end{displaymath}
  is the groupoid quotient of $\cB (\cG)$ by a normal Lie subgroupoid. In this
  case, the Lie group of bisections and the manifold $M$ completely determine
  the Lie groupoid $\cG$.}
 \end{inttheorem}
 
 Note that Theorem \ref{thm:B} really is a \emph{re}construction theorem, i.e.,
 we need a Lie groupoid to begin with as a candidate for the quotient. Hence,
 the original groupoid is needed and Theorem \ref{thm:B} does not provide a way
 to construct $\cG$ without knowing it beforehand. The problem here is twofold.
 At first, we need to know the kernel of $\ev$, or some equivalent information,
 that allows us to determine the groupoid $\cG$ as a quotient of
 $(\Bis (\cG) \ltimes M \toto M)$. The other problem is an analytical problem:
 quotients of (infinite-dimensional) Lie groupoids and Lie groups usually do
 not admit a suitable smooth structure. In particular, it is not known if the
 familiar tools, e.g.\ Godement's criterion, carry over to the
 infinite-dimensional setting beyond the Banach setting.
 
 Nevertheless, one can extract some information on the quotient from Theorem
 \ref{thm:B}. The groupoid quotient is controlled by certain subgroups of the
 group of bisections which arise from the kernel of the joint evaluation.
 Namely, the source fibre over $m$ in the kernel corresponds to the Lie
 subgroup $\Bisf{m} (\cG) = \{\sigma \in \Bis (\cG) \mid \sigma (m) =1_m \}$ of
 all bisections which take $m$ to the corresponding unit. Observe that
 $\Bisf{m} (\cG)$ sits inside the Lie subgroup
 $\Loop{m} (\cG) \coloneq \{\sigma \in \Bis (\cG) \mid \beta (\sigma (m)) = m\}$
 of all elements whose image at $m$ is an element in the vertex group. Both
 subgroups will later turn out to be important in the construction of Lie
 groupoids from their bisections. In this context Lie theoretic properties,
 like regularity in the sense of Milnor, of these subgroups are crucial to our
 approach. Thus we first study them in a separate section. Regularity (in the
 sense of Milnor) of a Lie group roughly means that a certain class of
 differential equations can be solved on the Lie group. Many familiar results
 from finite-dimensional Lie theory carry over only to regular Lie groups (cf.\
 \cite{HGRegLie15} for a survey). Our results then subsume the following
 theorem.
 
 \begin{inttheorem}\label{thm:C}
  {Let $\cG = (G \toto M)$ be a Banach Lie groupoid, then for each $m \in M$
  the inclusions
  \begin{displaymath}
   \Bisf{m} (\cG) \subseteq \Loop{m} (\cG) \subseteq \Bis (\cG)
  \end{displaymath}
  turn $\Bisf{m} (\cG)$ and $\Loop{m} (\cG)$ into split Lie subgroups which are
  regular in the sense of Milnor. As submanifolds these subgroups are even
  co-Banach submanifolds.}
 \end{inttheorem}
 
 These Lie subgroups are of interest, as the quotient
 $\Bis (\cG) / \Bisf{m} (\cG)$ reconstructs the source fibre of $\cG$ over the
 point $m$ and $\Loop{m} (\cG) / \Bisf{m} (\cG)$ reconstructs the vertex group
 at $m$. Moreover, Theorem \ref{thm:C} enables us to construct a natural smooth
 structure on these quotients which coincides a posteriori with the manifold
 structure on the fibre and the vertex group, respectively.
 
 We now use the results obtained so far to turn the reconstruction result given
 Theorem \ref{thm:B} into a construction result, at least in the locally
 trivial case. This means that we start with Lie groups, some extra structure
 on them and then produce a Lie groupoid such that the groups are related to
 the Lie group of bisections. Recall that a locally trivial Lie groupoid is
 completely determined by its source fibre and the vertex group over a given
 point. Hence, if we fix a point $m\in M$, the problem to construct a Lie
 groupoid reduces in the locally trivial case to reconstructing a manifold
 (modelling the source fibre) and a Lie group (modelling the vertex group).
 This motivates the notion of a \emph{transitive pair} (cf.\ Definition
 \ref{defn: tgpair}). A transitive pair $(\theta , H)$ consists of a transitive
 Lie group action $\theta \colon K \times M \rightarrow M$ and a normal
 subgroup $H$ of the $m$-stabiliser $\Stab{m}$ of $\theta$, such that $H$ is a
 regular and co-Banach Lie subgroup of $\Stab{m}$. The guiding example is here
 the transitive pair
 $(\beta_\cG \circ \ev \colon \Bis (\cG) \times M \rightarrow M, \Bisf{m} (\cG))$
 induced by the bisections of a locally trivial Lie groupoid $\cG$ with
 connected base $M$. The notion of transitive pair can be thought of as a
 generalisation of a Klein geometry for fibre-bundles (see Remark
 \ref{rem:klein_geometries_vs_transitive_pairs} for further information). We
 then obtain the following construction principle.
 
 \begin{inttheorem}\label{thm:D}
  {Let $(\theta \colon K \times M \rightarrow M,H)$ be a transitive pair. Then
  there is a locally trivial Banach Lie groupoid $\cR (\theta, H)$ together
  with a Lie group morphism
  $a_{\theta , H } \colon K \rightarrow \Bis (\cR (\theta,H))$}. \emph{If
  $(\theta,H) = (\beta_\cG \circ \ev, \Bisf{m} (\cG))$ for some locally trivial
  Banach Lie groupoid $\cG$, then $a_{\theta,H}$ is an isomorphism.}
 \end{inttheorem}
 
 One should think of the construction principle from Theorem \ref{thm:D} as an
 analogue of the reconstruction result in Theorem \ref{thm:B} for locally
 trivial Lie groupoids. The crucial difference here is that Theorem \ref{thm:D}
 makes no reference to the Lie groupoid, but constructs it purely from the
 given transitive pair. Note that the Lie group morphism $a_{\theta,H}$ in
 Theorem \ref{thm:D} will in general not be an isomorphism. However, the
 morphisms are interesting in their own right due to the fact that the
 definition of a transitive pair is quite flexible. It allows us to construct
 for a wide range of Lie groups with transitive actions on $M$ Lie group
 morphisms into the Lie group of bisections $\Bis (\cR (\theta, H))$. Moreover,
 these Lie group morphisms carry geometric information and thus connect the
 group actions of both Lie groups.
 
 In the last section we then invoke the integration theory of abelian
 extensions of infinite-dimensional Lie algebras from
 \cite{Neeb04Abelian-extensions-of-infinite-dimensional-Lie-groups} to
 construct transitive pairs from a closed 2-form $\omega$ on a 1-connected and
 compact manifold $M$. We formulate the results here for $\Diff(M)_{0}$,
 whereas in the text we allow for more general $K\leq \Diff(M)$. If $\omega$ is
 prequantisable, then the prequantisation provides a gauge groupoid and thus an
 integration of the Lie algebroid extension. By the results from
 \cite{SchmedingWockel14} and
 \cite{Neeb04Abelian-extensions-of-infinite-dimensional-Lie-groups}, the
 associated Lie algebra extension also integrates to an extension of Lie
 groups. In the other direction, we show that the integration of the extensions
 of Lie algebras to transitive pairs is in fact a two-step process. The first
 step is concerned with the integration of the extension
 \begin{equation}\label{eqn19}
  C^{\infty}(M)\to C^{\infty}(M)\oplus_{\ol{\omega}}\cV(M)\to \cV(M)
 \end{equation}
 of Lie algebras to an extension of Lie groups, where $\ol{\omega}$ is the
 abelian cocycle $(X,Y)\mapsto \omega(X,Y)$. By the results of
 \cite{Neeb04Abelian-extensions-of-infinite-dimensional-Lie-groups}, the
 integration of \eqref{eqn19} is governed by the discreteness of the
 \emph{primary} periods, i.e., the periods of the Lie group $\Diff(M)_{0}$ for
 the equivariant extension $\ol{\omega}^{\op{eq}}$ of $\ol{\omega}$. The second
 step is then the integration of the Lie subalgebra of
 $C^{\infty}(M)\oplus_{\ol{\omega}}\cV(M)$, that corresponds to the vector
 fields vanishing in the base-point, to a closed Lie subgroup. This is governed
 by the discreteness of the \emph{secondary} periods, i.e., the periods of
 $(M,\omega)$ modulo the periods of $(\Diff(M)_{0},\ol{\omega}^{\op{eq}})$ (see
 Remark \ref{rem:prequant} and Remark \ref{rem:secondary_periods} for a precise
 definition).

 \begin{inttheorem}
  Let $M$ be a compact and 1-connected manifold with base-point $m$ and
  $\omega\in \Omega^{2}(M)$ be closed. If the extension \eqref{eqn19}
  integrates to an extension of Lie groups, then $(M,\omega)$ is prequantisable
  if and only if the secondary periods are discrete. The latter is equivalent
  to the integrability of the subalgebra
  $C^{\infty}_{m}(M)\oplus _{\ol{\omega}}\cV_{m}(M)$ to a closed Lie subgroup,
  where $C^{\infty}_{m}(M)$ and $\cV_{m}(M)$ denote the functions (respectively
  vector fields) that vanish in $m$.
 \end{inttheorem}
 
 Finally, we would like to remark that the constructions of Lie groupoids given
 in the present paper yield functors on suitable categories of Lie groups and
 Lie groupoids. These functors are closely connected to the bisection functor.
 However, there is no need for these results in the present paper as we are
 only concerned with the (re-)construction of Lie groupoids. Thus we will
 largely avoid categorical language and postpone a detailed investigation of
 these functors to \cite{SW15Func}.
\end{tabsection}

\section{Locally convex Lie groupoids and the Lie group of bisections}
\label{sec:locally_convex_lie_groupoids_and_lie_groups}

\begin{tabsection}
 In this section we recall basic facts and conventions on Lie groupoids and bisections used in this paper. 
 We refer to \cite{Mackenzie05General-theory-of-Lie-groupoids-and-Lie-algebroids} for an
 introduction to (finite-dimensional) Lie groupoids and the associated group of
 bisections. The notation for Lie groupoids and their structural maps also
 follows \cite{Mackenzie05General-theory-of-Lie-groupoids-and-Lie-algebroids}.
 However, we do not restrict our attention to finite dimensional Lie groupoids.
 Hence, we have to augment the usual definitions with several comments. Note
 that we will work all the time over a fixed base manifold $M$.
\end{tabsection}

\begin{definition}
 Let $\cG = (G \toto M)$ be a groupoid over $M$ with source projection
 $\alpha \colon G \rightarrow M$ and target projection
 $\beta \colon G \rightarrow M$. Then $\cG$ is a \emph{(locally convex and
 locally metrisable) Lie groupoid over $M$}\footnote{See Appendix
 \ref{Appendix: MFD} for references on differential calculus in locally convex
 spaces.} if
 \begin{itemize}
  \item the objects $M$ and the arrows $G$ are locally convex and locally
        metrisable manifolds,
  \item the smooth structure turns $\alpha$ and $\beta$ into surjective
        submersions, i.e., they are locally projections\footnote{This implies
        that the occurring fibre-products are submanifolds of the
        direct products, see \cite[Appendix
        C]{Wockel13Infinite-dimensional-and-higher-structures-in-differential-geometry}.}
  \item multiplication
        $m \colon G \times_{\alpha,\beta} G \rightarrow G$, object
        inclusion $1 \colon M \rightarrow G$ and inversion
        $\iota \colon G \rightarrow G$ are smooth.
 \end{itemize}
 \end{definition}
 
 \begin{definition}
 The \emph{group of bisections} $\Bis (\cG)$ of $\cG$ is given as the set of
 sections $\sigma \colon M \rightarrow G$ of $\alpha$ such that
 $\beta \circ \sigma \colon M \rightarrow M$ is a diffeomorphism. This is a
 group with respect to
 \begin{equation}\label{eq: BISGP1}
  (\sigma \star \tau ) (x) \coloneq \sigma ((\beta \circ \tau)(x))\tau(x)\text{ for }  x \in M.
 \end{equation}
 The object inclusion $1 \colon M \rightarrow G$ is then the neutral element
 and the inverse element of $\sigma$ is
 \begin{equation}\label{eq: BISGP2}
  \sigma^{-1} (x) \coloneq \iota( \sigma ((\beta \circ\sigma)^{-1} (x)))\text{ for } x \in M.
 \end{equation}
\end{definition}

\begin{tabsection}
 In \cite{SchmedingWockel14} we have established a Lie group structure on the group of bisections of a (locally convex) Lie groupoid which admits a certain type of local addition.
 To understand the Lie group structure on $\Bis(\cG)$ one uses local additions (cf.\ Definition \ref{def:local_addition}) which respect the fibres of a submersion. 
 This is an adaptation of the construction of
 manifold structures on mapping spaces
 \cite{Wockel13Infinite-dimensional-and-higher-structures-in-differential-geometry,conv1997,michor1980}
 (see also Appendix \ref{Appendix: MFD}).
\end{tabsection}

\begin{definition}
 We say that that a Lie groupoid $\cG=(G\toto M)$ admits an \emph{adapted local
 addition} if $G$ admits a local addition which is adapted to the source
 projection $\alpha$ (or, equivalently, to the target projection $\beta$). 
\end{definition}

Recall from \cite[Section 3]{SchmedingWockel14} the following facts on the Lie group structure of the group of bisections:

\begin{Theorem} \label{theorem: A} 
Suppose $M$ is compact and $\cG=(G  \toto M)$ is a locally convex and
 locally metrisable Lie groupoid over $M$ which admits an adapted local
 addition. Then $\Bis (\cG)$ is a submanifold of $C^{\infty}(M,G)$ (with the
 manifold structure from Theorem \ref{thm: MFDMAP}). Moreover, the induced
 manifold structure and the group multiplication
 \begin{equation*}
  (\sigma \star \tau ) (x) \coloneq \sigma ((\beta \circ \tau)(x))\tau(x)\text{ for }  x \in M
 \end{equation*}
 turn $\Bis(\cG)$ into a Lie group modelled on $E_1 \coloneq   \{\gamma\in C^{\infty}(M,TG) \mid \forall x \in M, \  \gamma(x)\in T_{1_{x}}\alpha^{-1}(x) \}$.
\end{Theorem}

As the Lie group $\Bis (\cG)$ is a submanifold of $C^\infty(M.G)$ the exponential law for smooth maps \ref{thm:  MFDMAP} \ref{thm:manifold_structure_on_smooth_mapping_d} applies to maps defined on $\Bis (\cG)$.
In particular, as in \cite[Proposition 3.11]{SchmedingWockel14} one easily concludes that the natural action of the bisections on the arrows and the evaluation of bisections are smooth:  

\begin{proposition}\label{prop: LGP:Bis}
 Let $\cG = (G \toto M)$ be a Lie groupoid such that $M$ is a compact manifold and $\cG$ admits an adapted local addition.
 Then
 \begin{enumerate} 
  \item the natural action $\gamma \colon \Bis (\cG) \times G \rightarrow G, (\psi , g) \mapsto \psi (\beta (g)) \cdot g$ is smooth and for $g \in G$ the restricted action 
    \begin{displaymath}
     \gamma_g \colon \Bis (\cG) \rightarrow \alpha^{-1}(\alpha (g)) , \psi \mapsto \psi (\beta (g))\cdot g
    \end{displaymath}
    is smooth.
  \item the evaluation map $\ev \colon \Bis (\cG) \times M \rightarrow G , (\sigma , m) \mapsto \sigma (m)$ is smooth and satisfies 
    \begin{equation}\label{eq: ev:act}
     \ev (\sigma , m) = \sigma (m) = \gamma (\sigma , 1_m).
    \end{equation}
    Furthermore, the map $\ev_m \colon \Bis (\cG) \rightarrow \alpha^{-1} (m), \sigma \mapsto \sigma(m)$ is smooth for each $m \in M$.
 \end{enumerate}
\end{proposition}

\section{Reconstruction of the Lie groupoid from its group of bisections}
\label{sec:reconstruction_theorem}

\begin{tabsection}
 In this section, the close link between Lie groupoids and their
 Lie groups of bisections is established. More precisely, we will show that each Lie groupoid
 is the quotient of its Lie group of bisections. In the end, we discuss some
 examples and applications of this link.
 
 Throughout this section assume that $\cG=(G\toto M)$ is a locally metrisable Lie groupoid
 that admits an adapted local addition and that has a compact space of objects $M$.
\end{tabsection}

\begin{definition}
 The Lie group $\Bis(\cG)$ acts on $M$, via the action induced by $(\beta_{\cG})_{*}\from \Bis(\cG)\to \Diff(M)$ and the natural action of
 $\Diff(M)$ on $M$. This gives rise to the action Lie groupoid
 $\cB(\cG)\coloneq \Bis(\cG)\ltimes M$, with source and target projections defined by
 $\alpha_{\cB}(\sigma,m)=m$ and $\beta_{\cB}(\sigma,m)=\beta_{\cG}(\sigma(m))$.
 The multiplication on $\cB(\cG)$ is defined by
 \begin{equation*}
  (\sigma,\beta_{\cG}(\tau(m)))\cdot (\tau,m)\coloneq(\sigma \star \tau,m).
 \end{equation*}
 Clearly, any morphism $f\from \cG\to \cH$ of Lie groupoids over $M$ induces a
 morphism $f_{*}\times \id_{M}\from \cB(\cG)\to \cB(\cH)$ of Lie groupoids.
\end{definition}

\begin{remark}
\label{setup: bis:act}
 The Lie groupoid $\cB(\cG)$  admits an adapted local addition (for $\alpha_{\cB}$
 and thus also for $\beta_{\cB}$). In fact, this is the case for the Lie group
 $\Bis(\cG)$ and the finite-dimensional manifold $M$ separately (cf.\ \cite[p. 441]{conv1997}), and on
 $\Bis(\cG)\times M$ one can simply take the product of these local additions.
  In addition, the evaluation map $\ev \from \Bis(\cG)\times M\to G$ is a
 morphism of Lie groupoids over $M$
 \begin{equation*}
  \left(\vcenter{  \xymatrix{
  \Bis(\cG)\times M \ar@<1.2ex>[d]^{\beta _{\cB}}\ar@<-1ex>[d]_{\alpha_{\cB}} \\
  M
  }}\right)\xrightarrow{\ev}
  \left(\vcenter{\xymatrix{
  G \ar@<1.2ex>[d]^{\beta _{\cG}}\ar@<-1ex>[d]_{\alpha _{\cG}}\\
  M       
  }}\right).
 \end{equation*}
 \end{remark}

We will now study the analytic properties of the morphism $\ev\from \Bis(\cG)\times M\to G$. 
It will turn out that under mild assumptions on $\cG$ the groupoid is a groupoid quotient of the action groupoid $\cB (\cG)$ via the evaluation map.
The key point to establish this result is to prove that $\ev$ actually induces a quotient map, i.e.\ we need $\ev$ to be a surjective submersion.
Let us first deal with some preparations:

\begin{remark}\label{setup: preparation}
 Let $\cG = (G \toto M)$ be a Lie groupoid and fix $m \in M$ and $\tau \in \Bis (\cG)$. 
 Our goal is to split the space of sections $\Gamma (\tau^* T^\alpha G)$ into a product of (closed) subspaces.
 To this end define	
  \begin{displaymath} 
   \Gamma (\tau^* T^\alpha G)_{m} \coloneq \{X \in \Gamma (\tau^* T^\alpha G \mid X (m) = 0_{\tau (m)}\}.
  \end{displaymath}
 Choose a bundle trivialisation $\lambda \colon M_\lambda \rightarrow \tau^*\pi^\alpha (M_\lambda) \times E$ of the bundle $\tau^*\pi^\alpha \colon \tau^* T^\alpha G \rightarrow M$ such that $m \in \tau^*\pi^\alpha (M_\lambda)$.
 Since $M$ is compact, we can choose a smooth cut-off function $\rho \colon M \rightarrow [0,1]$ with $\rho_\lambda (m) = 1$ and $\rho|_{M \setminus \pi^\alpha (M_\lambda)} \equiv 0$. 
 Then we obtain a (non-canonical) isomorphism of topological vector spaces
  \begin{equation} \label{eq: split:iso}
   I_\lambda \colon \Gamma (\tau^* T^\alpha G) \rightarrow \Gamma (\tau^* T^\alpha G)_{m} \times T^\alpha_{\tau (m)} G ,\quad X \mapsto (X- (\rho \circ \tau^*\pi^\alpha) \cdot \lambda^{-1} (\tau^*\pi^\alpha,  \pr_2 \circ \lambda (X(m))) , X(m)).
  \end{equation}
 its inverse is given by 
 \begin{equation}\label{eq: nci:inv}
  I_\lambda^{-1} (X_0 , y) \coloneq X_0 + (\rho \circ \tau^*\pi^\alpha ) \cdot \lambda^{-1} (\tau^*\pi^\alpha , \pr_2 \circ \lambda (y)) 
 \end{equation}
 This turns $\Gamma (\tau^*T^\alpha G)_m$ into a complemented subspace of $\Gamma (\tau^*T^\alpha G)$. 
 Moreover, if $\cG$ is a Banach-Lie groupoid then $T^\alpha_{\tau (m)} G$ is a Banach space and thus $\Gamma (\tau^*T^\alpha G)_m$ turns into a co-Banach subspace of $\Gamma (\tau^*T^\alpha G)$.
\end{remark}

Let us establish the submersion property for the restricted action of the group of bisections on the manifold of arrows.

\begin{proposition} \label{prop: act:subm}
 Let $\cG = (G \toto M)$ be a Lie groupoid and $g \in G$. 
 Then the restricted action 
  \begin{displaymath}
   \gamma_g \colon \Bis (\cG) \rightarrow \alpha^{-1} (\alpha (g)) , \sigma \mapsto \gamma (\sigma,g) = \sigma (\beta (g)) .g
  \end{displaymath}
 is a submersion.
\end{proposition}

\begin{proof}
 From Proposition \ref{prop: LGP:Bis} a) we infer that $\gamma_g$ is smooth. 
 Hence, we only have to prove that $\gamma_g$ is locally a projection.  
 To see this fix $\tau \in \Bis (\cG)$ and a chart $\kappa \colon U_\kappa \rightarrow V_\kappa \subseteq E$ of the manifold $\alpha^{-1} (\alpha (g))$ such that $\gamma_g (\tau) \in U_\kappa$.
 Furthermore, choose a bundle trivialisation $\lambda$ of $\tau^* T^\alpha G$ and construct the vector space isomorphism \eqref{eq: split:iso} for $\lambda$.
 
 Now consider the canonical chart $(O_\tau , \varphi_\tau)$ of $\Bis (\cG)$.
 The set $I_\lambda \circ \varphi_\tau (O_\tau)$ is an open zero-neighbourhood in $\Gamma (\tau^* T^\alpha G)_{\beta (g)} \times T_{\tau (\beta (g))}^\alpha G$.
 Shrinking $O_\tau$, we can assume that $O_\tau \subseteq \gamma_g^{-1} (U_\kappa)$ and that there are open zero-neighbourhoods $U \subseteq \Gamma (\tau^* T^\alpha G)_{\beta (g)}$ and $W \subseteq T_{\tau (\beta (g))}^\alpha G$ such that $U\times W = I_\lambda \circ \varphi_\tau (O_\tau)$. 
 In conclusion, we obtain a commutative diagram 
  \begin{displaymath}
   \begin{xy}
  \xymatrix{
    \Bis (\cG) \ar[r]^-{\gamma_g} \supseteq O_\tau \ar[d]_{I_\lambda \circ \varphi_\tau}   &  \alpha^{-1} (\alpha (g)) \ar[r]^-{\kappa} & E  \ar@2{-}[d]\\
    U \times W  \ar[rr]^{\tilde{\gamma}_g}  & & E
  }
\end{xy}
  \end{displaymath}
 with $\tilde{\gamma}_g \coloneq \kappa \circ \gamma_g \circ (I_\lambda\circ \varphi_\tau)^{-1}|_{U \times W}$.
 Denote by $\A$ the adapted local addition of $\cG$ and consider the right translation $R_g \colon \alpha^{-1} (\beta (g)) \rightarrow \alpha^{-1} (\alpha (g)), h \mapsto hg$.
 Then \eqref{eq: nci:inv} and the definition of $\varphi_\tau^{-1}$ yield
 \begin{equation}\label{eq: form:ind} \begin{aligned}
  \tilde{\gamma}_g (X,y) &=  \kappa (\gamma_g (\A (X + (\rho \circ \pi^\alpha ) \cdot \lambda^{-1} (\pi^\alpha , \pr_2 \lambda (y)))) = \kappa (\A (\underbrace{X(\beta (g))}_{=0}  + \underbrace{(\rho (\beta (g)))}_{=1} \cdot y).g) \\
			 &= \kappa (\A (y).g) = \kappa \circ R_g \circ \A (y).
  \end{aligned}
 \end{equation}
 Note that  by \eqref{eq: form:ind} the map $\tilde{\gamma}_g$ does neither depend on the choice of the trivialisation $\lambda$ nor on the the cut-off function $\rho_\lambda$.
 
  By definition of the adapted local addition, $\A$ restricts to a diffeomorphism $W \rightarrow \A (W) \subseteq \alpha^{-1} (\beta (g))$.
 Moreover, $\kappa$ is a chart and the right translation $R_g$ is a diffeomorphism.
 Hence $\psi \coloneq \kappa \circ R_g \circ \A|_{W} \colon W \rightarrow E$ is a diffeomorphism onto its (open) image, mapping $W \subseteq T_{\tau (\beta (g))}^\alpha G$ to an open subset in $E$.
 Now \eqref{eq: form:ind} yields a commutative diagram 
 \begin{displaymath}
   \begin{xy}
  \xymatrix{
    \Bis (\cG) \ar[r]^-{\gamma_g} \supseteq O_\tau \ar[d]_{(\id_U , \psi) \circ I_\lambda \circ \varphi_\tau}  &  \alpha^{-1} (\alpha (g)) \ar[r]^-{\kappa} & E  \ar@2{-}[d]\\
    U \times \psi (W)  \ar[rr]^{\pr_2} & & E
  }
\end{xy}
  \end{displaymath}
from which we conclude that $\gamma_g$ is on ${O_\tau}$ a projection. 
As $\tau \in \Bis (\cG)$ was chosen arbitrarily, $\gamma_g$ is a submersion.
 \end{proof}

 The following corollary is now an immediate consequence of \eqref{eq: ev:act} and Proposition \ref{prop: act:subm}: 
 
\begin{corollary}\label{cor: evx:subm}
 Let $\cG = (G \toto M)$ be a Lie groupoid, then for each $m \in M$ the evaluation map 
  \begin{displaymath}
   \ev_m \colon \Bis (\cG) \rightarrow \alpha^{-1} (m) ,\quad \sigma \mapsto \sigma(m)
  \end{displaymath}
 is a submersion.
\end{corollary}

  Furthermore, we can establish the existence of bisections through certain arrows which coincide with the object inclusion outside of pre-chosen open sets.  

\begin{lemma}\label{lem: loc:bis}
  Let $\cG = (G \toto M)$ be a Lie groupoid and $m\in M$ fixed. 
  For any $m$-neighbourhood $U \opn M$ there exists a $1_m$-neighbourhood $W \opn \alpha^{-1} (m) \cap \beta^{-1} (U)$ such that for each $g \in W$ there is a bisection $\sigma_g \in \Bis (\cG)$ with $\sigma_g (m) = g$ and $\sigma_g (y) = 1_y$ for all $y \in M \setminus U$.
 \end{lemma}

 \begin{proof}
  Choose a $C^\infty$-function $\lambda \colon M \rightarrow [0,1]$ such that $\lambda (m) = 1$ and $\lambda|_{M \setminus U} \equiv 0$.
  We denote by 
    \begin{displaymath}
     \varphi_1 \colon \Gamma (1^*T^\alpha G) \supseteq \Omega \rightarrow \Bis (\cG), X \mapsto \A \circ X
    \end{displaymath}
  the canonical manifold chart of $\Bis (\cG)$ (see Theorem \ref{theorem: A}).
  Computing with local representatives, it is easy to see that the map $m_\lambda \colon \Gamma (1^*T^\alpha G) \rightarrow \Gamma (1^*T^\alpha G) , X \mapsto (x \mapsto \lambda (x ) \cdot X(x))$ is continuous linear.
  Hence, there is an open zero-neighbourhood $P \opn \Omega \subseteq \Gamma (1^*T^\alpha G)$ with $m_\lambda (P) \subseteq \Omega$.
  
  Now define $W \coloneq \ev_m (\varphi_1 (P)) \cap \beta^{-1} (U)$. 
  Observe that $W$ is an open $1_m$-neighbourhood in $\alpha^{-1} (m) \cap \beta^{-1} (U) \opn \alpha^{-1} (m)$ since $P$ is open and $\ev_m$ is a submersion by Corollary \ref{cor: evx:subm}.
  Moreover, for $g \in W$ we have $g = s_g (m)$ for some $s_g \in \varphi_1 (P)$.
  Define $X_g \coloneq m_\lambda \circ \varphi_1^{-1} (s_g)$ for each $g \in \cW$. 
  By construction $X_g$ is contained in $\Omega$ since $m_\lambda$ takes the section $\varphi_1^{-1} (s_g) \in P$ to $\Omega$. 
  Hence $\sigma_g \coloneq \varphi_1 (X_g)$ makes sense and is a bisection of $\cG$.
  From $\lambda (m) =1$ we derive that $\sigma_g (m) = s_g (m) = g$.
  Moreover, for $x \in M \setminus U$ we obtain by definition of a local addition
    \begin{displaymath}
     \sigma_g (x) = \A (\underbrace{\lambda (x)}_{=0} \varphi_1^{-1} (s_g) (x)) = \A (0_{T^\alpha_{1_x} G}) =  1_x.
    \end{displaymath}
  Hence $\sigma_g (x) = 1_x$ for all $x \in M \setminus U$. 
 \end{proof}
 
 We can now prove a variant of \cite[Theorem 3.2]{MR2511542} for infinite-dimensional Lie groupoids over compact base.  
 The proof of loc.cit. carries over verbatim if one uses Lemma \ref{lem: loc:bis}, whence we omit it.
 
 \begin{proposition}\label{prop: nuff:bis}
  Let $\cG = (G \toto M)$ be a Lie groupoid and $g \in \alpha^{-1} (m)$ for $m\in M$. Suppose that $W \opn \alpha^{-1} (m)$ is connected and contains $g$ and $1_m$ and there is $U \subseteq M$ with $\beta (W) \subseteq U$. Then there exists $\sigma_g \in \Bis (\cG)$ with $\sigma_g (m) = g$ and $\sigma_g (x) = 1_x$ for all $x \in M \setminus U$.
 \end{proposition}

 \begin{tabsection}
 Having dealt with the pointwise evaluation, we will now use the results
 obtained so far to also show that the \emph{joint} evaluation is a submersion.
\end{tabsection}

\begin{proposition}\label{cor:ev_on_section_space_is_a_submersion}
 Let $M$ be a compact manifold, $Q$ be a locally metrisable manifold and
 $s\from Q\to M$ be a submersion such that $Q$ admits a local addition that
 is adapted to $s$. Then the joint evaluation map
 \begin{equation}\label{eqn5}
  \ev\from \Gamma(M \xleftarrow{s} Q)\times M\to Q,\quad (\sigma,m)\mapsto \sigma(m)
 \end{equation}
 is a submersion. Here,
 $\Gamma(M \xleftarrow{s} Q) \coloneq \{ \sigma \in C^\infty (M,Q) \mid s \circ   \sigma = \id_M \}$
 is the space of sections of $s$, which is a submanifold of $C^{\infty}(M,Q)$
 by \cite[Proposition 3.6]{SchmedingWockel14}.
\end{proposition}

\begin{proof}
 Let $(\sigma_{0},m_{0})\in \Gamma(M \xleftarrow{s} Q)\times M$ be arbitrary,
 but fixed from now on.
 Set $q_{0}:=\sigma_{0}(m_{0})$. Then there exist open neighbourhoods $O\se M$
 of $m_{0}$ and $P\se Q$ of $q_{0}$ such that $\sigma_{0}(M)\se P$ and
 $\left.s\right|_{s^{-1}(O)}\cong \pr_{1}$. Indeed, there exist open
 neighbourhoods $P'$ of $q_{0}$ and $O'$ of $m_{0}$ such that
 $\left.s\right|_{P'}\cong \pr_{1}$ and $s(P')=O'$. After shrinking $O'$ if
 necessary, we may assume that $\sigma(O')\se P'$. Then we choose $O$ to be an
 open neighbourhood $O$ with $\ol{O}\se O'$ and set
 $P:= s^{-1}(M\setminus \ol{O})\cup P'$.
 
 From this it follows that $C^{\infty}(M,P)\cap \Gamma(M \xleftarrow{s} Q)$ is
 an open neighbourhood of $M$. Observe that $\left.s\right|_{P}\from P\to M$ is
 also a submersion and that the adapted local addition on $Q$ restricts to an
 adapted local addition on $P$. Consequently, the manifold structure on
 \begin{equation*}
  C^{\infty}(M,P)\cap \Gamma(M \xleftarrow{s} Q)=\Gamma(M\xleftarrow{\left.s\right|_{P}}P)
 \end{equation*}
 that is induced from $ \Gamma(M \xleftarrow{s} Q)$ on the left hand side
 coincides with the manifold structure on the right hand side that is induced
 from applying \cite[Proposition 3.6]{SchmedingWockel14} to the submersion
 $\left.s\right|_{P}$. Thus it suffices to consider the case where $Q=P$ and
 $s=\left.s\right|_{P}$.
 
 We now reduce the claimed submersion property of \eqref{eqn5} to the case of
 the evaluation in $m_{0}$. 
 Since the evaluation map
 $\ev_{m_{0}}\from \Diff(M) \to M$ is also a submersion, there exists a local
 smooth section of it, i.e., an open neighbourhood $U_{0}\se M$ of $m_{0}$ and
 a smooth map $\beta\from U_0\to \Diff(M)$ such that $\beta(m_{0})=\id_{M}$ and
 $\beta(m)(m_{0})=m$ for all $m\in U_{0}$. Moreover, we may assume that
 $U_{0}\se O$ and $\beta$ takes values in the identity
 neighbourhood $\varphi_{\id} (\Omega)$ where $\varphi_{\id}$ is the chart of $\Diff (M)$ from Theorem \ref{thm: MFDMAP}
 \ref{thm:manifold_structure_on_smooth_mapping_a}.
 
 Consider now $\varphi_{\id} \circ \beta \colon U_0 \rightarrow \Gamma (TM)$. 
 We choose a and a compact $m_0$-neighbourhood $K$ and a neighbourhood $\Omega \opn TO$ of the zero-section which is mapped by the local addition on $M$ (which was used to define $\varphi_{\id})$ to $O$.
 For later use, we shrink $\Omega$ to achieve that $\Omega \cap T_x M$ is convex for each $x \in O$.
 Then the open set $\lfloor K , \Omega\rfloor \coloneq \{X \in \Gamma (TM) \mid X (K) \subseteq \Omega\} \opn \Gamma (TM)$ is a zero-neighbourhood.
 Further, $\varphi_{\id}^{-1}$ maps $\lfloor K, \Omega\rfloor \cap (\varphi_{\id}^{-1})^{-1} (\Diff (M))$ to the set of diffeomorphisms which map $K$ into $O$.
 Shrinking $U_0$ we can achieve that $\varphi_{\id} \circ \beta$ takes its image in $\lfloor K , \Omega\rfloor$, i.e.\ the diffeomorphisms in the image of $\beta$ map $K$ into $O$. 
 
 Apply now the exponential law \cite[Corollary 7.5]{Wockel13Infinite-dimensional-and-higher-structures-in-differential-geometry} (cf.\ Theorem \ref{thm:   MFDMAP} \ref{thm:manifold_structure_on_smooth_mapping_d}) to obtain a smooth map 
 \begin{displaymath}
  \gamma \coloneq (\varphi_{\id} \circ \beta)^\vee \colon U_0 \times M \rightarrow TM, (u,m) \mapsto \varphi_{\id} \circ \beta (u) (m) \text{ with }\gamma (\cdot , m) \in T_m M \quad \forall m \in M.
 \end{displaymath}
 Note that $\gamma (m_0, \cdot)$ coincides with the zero-section as $\beta (m_0) = \id_M$.
 Choose a smooth cutoff function $\rho \colon M \rightarrow [0,1]$ which maps a $m_0$-neighbourhood $N \subseteq K$ to $1$ and vanishes near the boundary $\partial K$. 
 Multiplying fibre-wise we obtain a smooth map $\tilde{\gamma} \colon U_0 \times M \rightarrow TM, (u,m) \mapsto \rho(m) \cdot \gamma(u,m)$ which vanishes near the boundary of $K$.
 Apply the exponential law in reverse to obtain a smooth map $\tilde{\gamma}^\wedge \colon U_0 \rightarrow \Gamma (TM)$.
 
 Now recall that $(\varphi_{\id}^{-1})^{-1} (\Diff (M))$ is an open set in the compact open $C^1$-topology on $\Gamma (TM)$ (see \cite[4.3]{michor1980} and \cite[proof of Theorem 43.1]{conv1997}).
 Hence we can choose suitable convex zero-neighbourhoods which control only the values of $X \in \Gamma(TM)$ and $TX$ on suitable compact sets, such that their intersection is contained in $(\varphi_{\id}^{-1})^{-1} (\Diff (M))$. 
 Since $\gamma (m_0,\cdot)$ is the zero-section, we can thus shrink $U_0$ to achieve that $\tilde{\gamma}^\wedge$ is still contained in $\lfloor K , \Omega \rfloor \cap (\varphi_{\id}^{-1})^{-1} (\Diff (M))$.
 As $\tilde{\gamma}^\wedge$ vanishes near the boundary of $K$ (and outside of $K$), we can replace $\beta$ with a map which satisfies $\beta (m) (O) \subseteq O$ for all $m \in U_0$.
  
 Choose another smooth cutoff function $\delta$ which vanishes near the boundary $\partial U_0$ and takes an open $m_0$-neighbourhood $U \subseteq N$ to $1$. 
 Multiplying (in local charts) $\varphi_{\id} \circ \beta$ with $\delta$, we can assume that $\beta (m) = \id_M$ for all $m$ near the boundary $\partial U_0$.
 Thus we can extend $\beta$ to a smooth function $\beta\from M\to \Diff(M)$ by setting
 $\beta(m)=\id_{M}$ if $m\notin U_{0}$.
 By construction this map satisfies $\beta (m) (O) \subseteq O$ for all $m \in M$.
 Moreover, since $\rho$ and $\delta$ take the $m_0$-neighbourhood $U \cap N$ to $1$, $\beta$ satisfies $\beta(m)(m_{0})=m$ for all $m\in U \cap N$. 
  
 Choose a
 diffeomorphism $\xi \from s^{-1}(O)\to O \times W$ that makes
 \begin{equation*}
  \xymatrix@=1em{
  s^{-1}(O)\ar[dr]_{s} \ar[rr]^{\xi} && O\times W\ar[dl]^{\pr_{1}}\\
  & O &
  }
 \end{equation*}
 commute, we may lift each $\beta(m)$ to a diffeomorphism
 \begin{equation*}
  \wt{\beta}_{m}\from Q\to Q,\quad q\mapsto \begin{cases}
  \xi^{-1}(\beta(m)(o),w) & \text{ if }q=\xi^{-1}(o,w)\in s^{-1}(O)\\
  q &\text{ if }q\notin s^{-1}(O)
  \end{cases}.
 \end{equation*}
 From this explicit construction it follows in particular, that the map
 \begin{equation*}
  Q\times M\to Q,\quad (q,m)\mapsto \wt{\beta}(m)(q)
 \end{equation*}
 is smooth. This then gives rise to the diffeomorphism
 \begin{equation*}
  B\from  \Gamma(M \xleftarrow{s} Q)\times M\to  \Gamma(M \xleftarrow{s} Q)\times M,\quad (\sigma,m)\mapsto ((\wt{\beta}(m))^{-1}\circ \sigma \circ \beta(m),m).
 \end{equation*}
 with inverse
 $(\sigma,m)\mapsto (\wt{\beta}(m)\circ \sigma \circ (\beta(m))^{-1},m)$.
 Moreover, $\beta(m_{0})=\id_{M}$ implies
 $B(\sigma _{0},m_{0})=(\sigma _{0},m_{0})$.
 
 For the next step, let $\Sigma\from TQ \supseteq \Omega \to Q$ be an $s$-adapted
 local addition with corresponding open $\sigma_0$-neighbourhood $O_{\sigma_{0}}\se C^{\infty}(M,Q)$ and let
 $\varphi_{\sigma _{0}}\from O_{\sigma _{0}}\to \Gamma( \sigma _{0}^{*}(TQ))$
 be the chart from Theorem \ref{thm: MFDMAP}
 \ref{thm:manifold_structure_on_smooth_mapping_a}. Recall from the proof of
 \cite[Proposition 3.6]{SchmedingWockel14} that $\varphi_{\sigma_{0}}$ also is
 a submanifold chart for $\Gamma(M\xleftarrow {s} Q)$ that maps the open
 neighbourhood $ O_{\sigma_{0}}\cap \Gamma(M \xleftarrow{s} Q)$ onto an open
 neighbourhood of the zero section in the closed subspace
 $\{ \tau\in \Gamma(\sigma_{0}^{*}(TQ))\mid \tau(m)\in T_{\sigma_{0}(m)}s^{-1}(m)\text{ for all }m\in M \}$
 of $\Gamma(\sigma_{0}^{*}(TQ))$.

 Denote
 $ R_{0}:= \varphi _ { \sigma _{0}}(O_{\sigma_{0}}\cap \Gamma(M \xleftarrow{s} Q))$.
 Then we define the diffeomorphism
 \begin{equation*}
  \Xi \from   B^{-1}((O_{\sigma_{0}}\cap \Gamma(M \xleftarrow{s} Q))\times O)\to R_{0}\times O,\quad (f,m)\mapsto (\varphi_{\sigma _{0}}( (\wt{\beta}(m))^{-1}\circ f \circ \beta(m)),m).
 \end{equation*}
 By shrinking $B^{-1}(O_{\sigma _{0}}\times O)$ if necessary, we may assume
 that $\ev$ maps the open neighbourhood $(B^{-1}(O_{\sigma _{0}}\times O))$ of
 $(\sigma _{0},m_{0})$ into the open neighbourhood $s^{-1}(O)$ of $q_{0}$.
 Moreover, the following diagram commutes by the definitions of
 $\varphi_{\sigma _{0}}$ and of $\wt{\beta}(m)$
 \begin{equation*}
  \xymatrix{
  B^{-1}(O_{\sigma _{0}}\times O)\ar[d]_{\Xi} \ar[r]^-{\ev} & s^{-1}(O) \ar[r]^{\xi}& W\times O %
  \\
  R_{0}\times O\ar[rr]^-{\ev_{m_{0}}\times \id_{O}} &	& T_{q_{0}}s^{-1}(m_{0})\times O \ar[u]_{\Sigma \times \id_{O}}.
  }
 \end{equation*}
 Since $W$ is diffeomorphic to $s^{-1}(m_{0})$ and $\Sigma$ restricts to
 a local diffeomorphism of a zero-neighbourhood in $T_{q_{0}}s^{-1}(m_{0})$
 onto an open neighbourhood of $q_{0}$ in $s^{-1}(m_{0})$, it follows that the
 arrow on the right is a local diffeomorphism. Hence each morphism except $\ev$
 and $\ev_{m_{0}}\times \id_{O}$ in this diagram is a local diffeomorphism. It
 will thus follow from showing that $\ev_{m_{0}}$ is a submersion and \cite[Lemma 1.6]{hg2015} that $\ev$ is
 one.
 
 The map $\ev_{m_0}$ is defined on an open zero neighbourhood of the space
 \begin{equation*}
  E_{\sigma_0}\coloneq \{ \tau\in \Gamma(\sigma_{0}^{*}(TQ))\mid \tau(m)\in T_{\sigma_{0}(m)}s^{-1}(m)\text{ for all }m\in M \}
 \end{equation*}
 and takes values in $T_{q_{0}}s^{-1}(m_{0})$. 
 
 By using the diffeomorphism
 $\xi\from s^{-1}(O)\to W\times O$ we have the following identifications
 \begin{equation*}
  T_{q_{0}}s^{-1}(m_{0})\cong T_{(w_{0},m_{0})} \xi (s^{-1} (m_0))\cong F
 \end{equation*}
 where $F$ is the modelling space of $W$ and $(w_{0},m_{0}):=\xi(q_{0})$. Using
 these and a cutoff function, one can build as in Remark \ref{setup: preparation} a
 continuous inverse to $\ev_{m_0}$ that takes $T_{q_{0}}s^{-1}(m_{0})$ into $E_{\sigma_0}$.
 Thus $\ev_{m_{0}}$ is a submersion, finishing the proof.
\end{proof}

\begin{corollary}\label{cor:ev_on_mapping_space_is_a_submersion}
 Let $M$ be a compact manifold and $N$ be a locally metrisable manifold that
 admits a local addition. Then the joint evaluation map
 \begin{equation}\label{eqn6}
  \ev\from C^{\infty}(M,N)\times M\to N,\quad (f,m)\mapsto f(m)
 \end{equation}
 is a submersion.
\end{corollary}

\begin{proof}
 We have the natural identification
 $C^{\infty}(M,N)\cong \Gamma (M\xleftarrow{\pr_{2}} N\times M)$. Thus
 Proposition \ref{cor:ev_on_section_space_is_a_submersion} shows that
 \begin{displaymath}%
  C^{\infty}(M,N)\times M\to N\times M,\quad (f,m)\mapsto (f(m),m)
 \end{displaymath}
 is a submersion. Now \eqref{eqn6} is a submersion, since it is the
 composition to two submersions.
\end{proof}

\begin{corollary}\label{cor:ev_subm}
 Let $\cG = (G \toto M)$ be a Lie groupoid. Then the joint evaluation
 map
 \begin{displaymath}
  \ev \colon \Bis (\cG)\times M \rightarrow G ,\quad (\sigma,m) \mapsto
  \sigma(m)
 \end{displaymath}
 is a submersion.
\end{corollary}

\begin{proof}
 This is implied by Proposition \ref{cor:ev_on_section_space_is_a_submersion}
 since $\Bis (\cG)$ is an open submanifold of $\Gamma(M\xleftarrow{\alpha}G)$.
\end{proof}

\begin{corollary}
   Let $\cG= (G \toto M)$ be a Lie groupoid. 
   Then the \emph{division morphism} 
	\begin{displaymath}
	 \delta \colon \Bis (\cG) \times \Bis (\cG) \times M \rightarrow G, (\sigma , \tau , m) \mapsto (\sigma \star \tau^{-1}) (m)
	\end{displaymath}
       and for $m \in M$ the restricted division $\delta_m \colon \Bis (\cG ) \times \Bis (\cG) \rightarrow \alpha^{-1} (m) , \delta_m (\sigma ,\tau) \coloneq \delta (\sigma, \tau, m)$  are submersions.
  \end{corollary}
  
  \begin{proof}
   Note that we can write $\delta (\sigma, \tau , m) = \ev (\sigma \star \tau^{-1} , m)$ and $\delta_m (\sigma ,\tau) = \ev_m (\sigma \star \tau^{-1})$. 
   Since $\ev$ and $\ev_m$ are submersions by Corollary \ref{cor: evx:subm} and Corollary \ref{cor:ev_subm}, it suffices to prove that the map 
    \begin{displaymath}
     f \colon \Bis (\cG) \times \Bis (\cG) \rightarrow \Bis (\cG) , (\sigma, \tau) \mapsto \sigma \star \tau^{-1} 
    \end{displaymath}
   is a submersion. 
    However, as $(\Bis (\cG), \star)$ is a Lie group the map $f$ is a submersion 
  \end{proof}

 We have now established that the evaluation map from the bisections onto the manifold of arrows is a submersion. 
 However, to completely determine the manifold of arrows, we need $\ev$ to be surjective.
 Note that this means that there is a (global) bisection through each point in $G$. 
 Consider first an easy example 
 
\begin{example}\label{ex:
 ev:actgpd} Let $H$ be a Lie group modelled on a metrisable space which acts on
 the compact manifold $M$, i.e.\ the associated action groupoid $H\ltimes M$
 admits an adapted local addition by Remark \ref{setup: bis:act}. Then the evaluation
 $\ev \colon \Bis (H\ltimes M ) \times M \rightarrow H \times M$,
 $(\sigma ,m) \mapsto (\sigma (m),m)$ is a surjective submersion. We already
 know from Corollary \ref{cor:ev_subm} that $\ev$ is a submersion and thus have
 to establish only surjectivity. For each pair $(h,m) \in H \times M$ we can
 define the constant bisection
 $c_{h} \colon M \rightarrow H \times M,  n \mapsto (h,n)$ which is contained
 in $\Bis (H\ltimes M)$. Hence $\ev (c_h,m)=(h,m)$ and thus $\ev$ is
 surjective. In particular, for each arrow $g \in H \times M$ in the action
 groupoid $\cG$ there is a global bisection $\sigma_g$ with
 $\sigma_g (\alpha (g))= g$.
 
 The structure of $\Bis(H\ltimes M)$ is interesting in its own. Since the
 second component of a bisection $\sigma\from M\to H\times M$ has to be the
 identity, $\Bis(H\ltimes M)$ can be identified with the subset
 \begin{equation*}
  \{\gamma\in C^{\infty}(M,H)\mid m\mapsto \gamma(m).m\text{ is a diffeomorphism of }M\}
 \end{equation*}
 of $C^{\infty}(M,H)$. Since $C^{\infty}(M,H)\to C^{\infty}(M,M)$,
 $\gamma\mapsto (m\mapsto \gamma(m).m)$ is smooth and
 $\Diff(M)\se C^{\infty}(M,M)$ is open, it follows that $\Bis(M\ltimes H)$ is
 an open submanifold of $C^{\infty}(M,H)$ that contains the constant maps.
 However, the group structure on the functions from $M$ to $H$ is not given by
 the pointwise multiplication, but by
 \begin{equation*}%
  (\gamma\star \eta)(m):= \gamma(\eta(m).m)\cdot \eta(m).
 \end{equation*}
 One effect of this is that $\gamma\mapsto (m\mapsto \gamma(m).m)$ is a
 homomorphism $\Bis(H\ltimes M)\to \Diff(M)$.
\end{example}

In general there will not be a bisection through each arrow of a given groupoid (see Remark \ref{rem: ev:sur} b) below).
Nevertheless, for source connected finite-dimensional Lie groupoids it is known (see \cite{MR2511542}) that bisections through each arrow exist.
We will now generalise this result to our infinite-dimensional setting.  

\begin{definition}
 Let $\cG = (G \toto M)$ be a Lie groupoid. 
 \begin{enumerate}
  \item Denote for $m \in M$ by $C_m$ the connected component of $1_m$ in $\alpha^{-1} (m)$. 
  Then we define the subset $C(\cG) \coloneq \bigcup_{m \in M} C_m$. 
  By \cite[Proposition 1.5.1]{Mackenzie05General-theory-of-Lie-groupoids-and-Lie-algebroids} we obtain a wide subgroupoid $C(\cG) \toto M$ of $\cG$, called the \emph{identity-component subgroupoid} of $\cG$.\footnote{Note that at this stage we do not know that $C(\cG)$ is a Lie subgroupoid. In particular, the proof for finite dimensional Lie groupoids (see  \cite[Proposition 1.5.1]{Mackenzie05General-theory-of-Lie-groupoids-and-Lie-algebroids}) does not carry over to our setting. Compare however Theorem \ref{thm: bis:points} b).}
  \item The groupoid $\cG$ is called $\alpha$- or \emph{source connected} if for each $m \in M$ the fibre $\alpha^{-1} (m)$ is connected.  
 \end{enumerate}
Observe that for an $\alpha$-connected groupoid $\cG$ we have $C(\cG) = \cG$.
\end{definition}

 Note that Proposition \ref{prop: nuff:bis} implies that for $\alpha$-connected Lie groupoids there is for every arrow a bisection whose image contains the given arrow.
 However, we give an alternative proof in the following theorem, which also yields more information:

\begin{theorem}
 \label{thm: bis:points} Let $\cG = (G \toto M)$ be a locally metrisable Lie
 groupoid with compact $M$ that admits an adapted local addition.
 \begin{enumerate}
  \item The image of the evaluation $\ev$ is an open and wide Lie subgroupoid
        which contains the identity subgroupoid $C(\cG)$.
  \item The identity subgroupoid $C(\cG)$ coincides with
        $\ev(\Bis(\cG)_{0}\times M)$, where $\Bis(\cG)_{0}$ is the identity
        component of $\Bis (\cG)$. Hence, $C(\cG)$ is an open Lie subgroupoid of
        $\cG$.
 \end{enumerate}
 Assume in addition that $\cG$ is $\alpha$-connected, then
 \begin{enumerate}
  \item[c)] For each $g \in G$ there is a bisection
        $\sigma_g \in \Bis (\cG)_{0}$ with $\sigma_g (\alpha (g)) = g$. In
        particular, $\ev$ is surjective.
 \end{enumerate}
\end{theorem}

\begin{proof} 
  \begin{enumerate}
   \item The image $\cU \coloneq \ev (\Bis (\cG) \times M)$ contains the image of the object inclusion $1 \colon M \rightarrow G$, i.e.\ $1_m \in \cU$ for all $m \in M$.
   Define for $m \in M$ the set $\cU_m = \cU \cap \alpha^{-1} (m)$ and note that $\cU_m = \ev_m (\Bis (\cG))$.
   As $\ev_m \colon \Bis (\cG) \rightarrow \alpha^{-1} (m)$ is a submersion by Corollary \ref{cor: evx:subm} we infer that $\cU_m$ is an open subset of $\alpha^{-1} (m)$. 
     
   By construction of the group operations of $\Bis (\cG)$ the set $\cU$ yields a wide subgroupoid $\cU \toto M$ of $\cG$ such that $\cU_m$ is open in $\alpha^{-1} (m)$.
   Hence the image of $\ev$ is an open and wide Lie subgroupoid.
   Now \cite[Proposition 1.5.7]{Mackenzie05General-theory-of-Lie-groupoids-and-Lie-algebroids}\footnote{Loc.cit. considers only finite-dimensional Lie groupoids. 
   However, the proof of this result carries over verbatim to the infinite-dimensional setting.} shows that $\cU_m$ is also closed in $\alpha^{-1} (m)$.
   Thus the clopen set $\cU_m$ contains the connected component of $1_m \in \alpha^{-1} (m)$.
   Since this holds for each $m$, we see that $C(\cG) \subseteq \cU$.
   \item Set $B_0:=\Bis (\cG)_{0}$.
   As $B_0$ is an open subgroup, an argument as in a) shows that $\ev (B_0 \times M)$ is a subgroupoid of $\cG$  which contains $C(\cG)$.
    
   Furthermore, for each $m \in D_i$ the set $\ev_m (B_0) \subseteq \alpha^{-1} (m)$ is connected and contains $1_m$.
   Thus by definition of the connected set $C_m \subseteq \alpha^{-1} (m)$ we have $\ev_m (B_0) \subseteq C_m \subseteq C(\cG)$.
   Hence, $\ev (B_0 \times M) = C(\cG)$ and since $\ev$ is a submersion by Corollary \ref{cor:ev_subm}, $C(\cG)$ is open in $G$. 
   In particular, $C(\cG)$ is an open subgroupoid of $\cG$, i.e.\ it is an open Lie subgroupoid. 
   \item If $\cG$ is $\alpha$-connected then $C(\cG) = G$ whence the assertion follows from b).
  \end{enumerate}
\end{proof}

\begin{definition}
 We say that a for a Lie groupoid $\cG = (G \toto M)$ there exists a \emph{bisection through each arrow} if $\cG$ satisfies the condition of Theorem \ref{thm: bis:points} c), 
 i.e.\ for each $g \in G$ there exists $\sigma_g \in \Bis (\cG)$ with $\sigma_g (\alpha (g))= g$.   
\end{definition}

 Note that Part c) of Theorem \ref{thm: bis:points} yields \cite[Theorem 3.1]{MR2511542} as a corollary for (finite-dimensional) Lie groupoids over a compact base.

\begin{corollary}\label{cor: bis:pts}
 In a finite-dimensional source-connected Lie groupoid with compact space of objects there exist bisections through each arrow.
\end{corollary}

We also obtain the following well known result on the natural action of $\Diff(M)$ on $M$ (cf.\ \cite{Banyaga97The-structure-of-classical-diffeomorphism-groups,MichorVizman94n-transitivity-of-certain-diffeomorphism-groups}).

\begin{corollary}\label{cor: diff0:trans}
 If $M$ is a compact and connected manifold, then $\Diff(M)_{0}$ acts
 transitively on $M$.
\end{corollary}

\begin{remark}\label{rem: ev:sur}
\begin{enumerate}
 \item Note that \cite[Theorem 3.1]{MR2511542} holds for arbitrary finite-dimensional $\alpha$-connected Lie groupoids whereas Corollary \ref{cor: bis:pts} is limited to groupoids over compact base. 
 \item  The assumption on $\cG$ to be source connected cannot be dispensed
 with. For instance, if $N,N'$ are non-diffeomorphic compact manifolds (of the
 same dimension), then the pair groupoid $\cP(M):=(M\times M\toto M)$ of
 $M\coloneq N \sqcup N'$ has $\Bis(\cP(M))\cong \Diff(M)$ and the action of
 $\Bis(\cP(M))$ on the source fibre naturally identifies with the natural
 action of $\Diff(M)$ on $M$. But since $N,N'$ are not diffeomorphis, there
 cannot exist a diffeomorphism of $M$ that interchanges the points $n$ and $n'$
 if $n\in N$ and $n'\in N'$. Consequently, there cannot exist a bisection
 through the morphism $((n,n'),(n',n))$ of $\cP(M)$.
 
 However, as we have seen in  Example \ref{ex: ev:actgpd}, there exist non-source connected Lie groupoids, for which there exist
 bisections through each point. 
 For another example consider the gauge groupoid $\cG=\op{Gauge}(M\times K)$ of the trivial principal bundle $M\times K\to M$, then
 $\Bis(\op{Gauge}(M\times K))\cong \Aut(M\times K)\cong C^{\infty}(M,K)\rtimes \Diff(M)$.
 If $M$ is connected, then $\Diff(M)$ acts transitively on $M$ and
 $C^{\infty}(M,K)$ always acts transitively on $K$. Thus there exist a
 bisection through each arrow of $\op{Gauge}(M\times K)$, even if $K$ is
 not connected.
\end{enumerate}
\end{remark}

 Before we continue with our investigation of the bisection action groupoid, note the following interesting consequences of Corollary \ref{cor: evx:subm}.
 
 \begin{Lemma} \label{lem: BG:loctriv}
 Let $\cG$ be locally trivial and denote by $\theta \colon \Bis (\cG) \times M \rightarrow M, (\sigma,x) \mapsto \beta_\cG \circ \sigma (m)$ the canonical Lie group action. 
 Then 
  \begin{enumerate}
   \item $\theta$ restricts for each $m \in M$ to a submersion $\theta_m \colon \Bis (\cG) \times \{m\} \rightarrow M$.
  \end{enumerate}
 If in addition $\cG$ admits bisections through each arrow or $M$ is connected, then 
  \begin{enumerate}
   \item[b)] $\cB(\cG)$ is locally trivial and $\theta$ is transitive.
  \end{enumerate}
 \end{Lemma}
 
 \begin{proof}
 As $\cG$ is locally trivial $\left.\beta_{\cG}\right|_{\alpha^{-1}(m)}\from \alpha^{-1}(m)\to M$ is a
 surjective submersion. Since the $\alpha_{\cB}$-fibre is $\Bis(\cG)$ and $\ev_{m}\from  \Bis(\cG)\to \alpha^{-1}(m)$ is a submersion by Corollary \ref{cor: evx:subm} we see that $\beta_{\cB}|_{\alpha_\cB^{-1} (m)} = \beta_\cG \circ \ev_m$ is a submersion. 
 We claim that $\beta_\cB$ is surjective. If this is true then $\cB (\cG)$ is locally trivial by \cite[Proposition 1.3.3]{Mackenzie05General-theory-of-Lie-groupoids-and-Lie-algebroids} which carries over verbatim to our infinite-dimensional setting.
 Moreover, we derive that the canonical action of $\Bis (\cG)$ is transitive and restricts to a submersion $\Bis (\cG) \times \{m\} \rightarrow M$ for all $m \in M$ 
 
 To prove the claim we have to treat both cases separately.
 Assume first that $\cG$ admits bisections through each arrow, then $\ev_m$ is surjective and thus $\beta_\cB$ is surjective.
 On the other hand let now $M$ be connected.
 Then we note that $\beta_\cB = \widetilde{\ev_m} \circ (\beta_\cG)_*$, where $(\beta_\cG)_* \colon \Bis (\cG) \rightarrow \Diff (M), \sigma \mapsto \beta_\cG \circ \sigma$ and $\widetilde{\ev_m} \colon \Diff (M) \rightarrow M, \varphi \mapsto \varphi (m)$.
 Now the image of $(\beta_\cG)_*$ contains the identity component $\Diff (M)_0$ of $\Diff (M)$ (by \cite[Example 3.16]{SchmedingWockel14}) and $\Diff (M)_0$ acts transitively on the connected manifold $M$ by Corollary \ref{cor: diff0:trans}.
 Thus $\beta_\cB$ is surjective.
 
 We conclude that in both cases the assertion holds.
 \end{proof}

\begin{remark}
 Quotient constructions for Lie groupoids (and already for Lie groups) are
 quite tricky. In fact, Lie groupoids are a tool to circumvent badly behaved
 quotients (for instance for non-free group actions). However, each category
 carries a natural notion of quotient object for an internal equivalence
 relation. If $\cC$ is a category with finite products and $ R\se E\times E$ is
 an internal equivalence relation, then the quotient $E\to E/R$ in $\cC$
 (uniquely determined up to isomorphism) is, if it exists, the coequaliser of
 the diagram
 \begin{equation}\label{eqn1}
  \xymatrix{ R \ar@<.8ex>[r]^{\pr_{1}} \ar@<-.8ex>[r]_{\pr_{2}} & E}.
 \end{equation}
 If, in the case that the quotient exists, \eqref{eqn1} is also the pull-back
 of $E\to E/R$ along itself, then the quotient $E\to E/R$ is called
 \emph{effective} (see \cite[Appendix.1]{Mac-LaneMoerdijk94Sheaves-in-geometry-and-logic} for
 details). 
We want to apply this to the category $\cat{LieGroupoids}_{M}$, whose objects are locally
 convex and locally metrisable Lie groupoids over $M$ and whose morphisms are
 smooth functors that are the identity on $M$. Note that the
 product of two Lie groupoids $(G\toto M)$ and $(H\toto M)$ is given by restricting the product Lie groupoid $(G\times H\toto M\times M)$ to the diagonal $M\cong \Delta M\se M\times M$.
\end{remark}

\begin{theorem}\label{thm:
 gpd:quot} If $\cG=(G\toto M)$ is a Lie groupoid with a bisection through each
 arrow in $G$, e.g.\ $\cG$ is source connected, then the morphism
 $\ev\from \cB(\cG)\to \cG$ is the quotient of $\cB(\cG)$ in
 $\cat{LieGroupoids}_{M}$ by
 \begin{equation*}
  R=\{(\sigma,m),(\tau,m)\in \Bis(\cG)\times \Bis(\cG)\times M\mid \sigma(m)=\tau(m) \}.
 \end{equation*}
\end{theorem}

\begin{proof}
 We first note that $R$ is isomorphic to $  K\times \Bis(\cG)$, where
 \begin{equation}\label{eqn2}
  K\coloneq\{(\sigma,m)\in \Bis(\cG)\times M\mid \sigma(m)=1_{m}\}=\ev^{-1}(M)
 \end{equation}
 is the kernel of $\ev$. 
 To see this note that as $M\se G$ is a closed submanifold and $\ev$ is a
 submersion by Corollary \ref{cor:ev_subm}, it follows that $K$ is a closed
 submanifold of $\Bis(\cG)\times M$. 
 Now  
 \begin{equation*}
  \Bis(\cG)\times \Bis(\cG)\times M\to \Bis (\cG) \times M \times \Bis(\cG),\quad ((\sigma,m),(\tau,m))\mapsto (\sigma \star \tau^{-1},m,\tau)
 \end{equation*}
 is a diffeomorphism which maps $R$ onto
 $K\times \Bis(\cG)$. Consequently, $R$ is a closed submanifold
 of $\Bis(\cG)\times\Bis(\cG)\times M$.
 
 We now argue that $R$ is in fact an internal equivalence relation. It is clear
 that the relation is reflexive and symmetric. After applying the
 diffeomorphism \eqref{eqn2}, the second projection
 $\pr_{2}\from R\to \Bis(\cG)\times M$, $((\sigma,m),(\tau,m))\mapsto (\tau , m)$
 is an actual projection. So $\pr_{2}$ is a surjective submersion. Thus the
 pull-back
 \begin{equation*}
  R*R:=R\times _{(\Bis(\cG)\times M)}R=\{(((\sigma,m),(\tau,m)),((\sigma',m'),(\tau',m')))\mid m=m', \tau=\sigma'\}
 \end{equation*}
 is a submanifold of $R\times R$ and
 \begin{equation*}
  R*R\to \Bis(\cG)\times \Bis(\cG)\times M,\quad
  (((\sigma,m),(\tau,m)),((\tau,m),(\tau',m)))\mapsto ((\sigma,m),(\tau',m))
 \end{equation*}
 clearly factors through $R$. Consequently, $R$ is an internal equivalence
 relation.
 
 Finally, if $f\from \Bis(\cG)\times M\to H$ is a morphism of Lie groupoids
 that satisfies $f \circ \pr_{1}=f \circ \pr_{2}$ (for
 $\pr_{i}\from R\to \Bis(\cG)\times M$ the canonical maps), then we construct a
 smooth map $\varphi \from G\to H$ by taking a local smooth section of $\ev$
 and composing it with $f$. Since $f \circ \pr_{1}=f \circ \pr_{2}$, two
 possible pre-images of an element from $G$ in $\Bis(\cG)\times M$ are mapped
 to the same element in $H$, and thus $\varphi$ is well-defined and smooth by
 construction. One directly checks that it also defines a morphism of Lie
 groupoids (i.e., $\varphi$ is compatible with the structure maps, see also
 \cite[Proposition
 2.2.3]{Mackenzie05General-theory-of-Lie-groupoids-and-Lie-algebroids}).
\end{proof}
 
\begin{remark}\label{rem:
 Re:inconstruction} We have seen in Theorem \ref{thm: gpd:quot} that a source
 connected Lie groupoid $\cG$ with compact base is the quotient of its
 associated bisection action Lie groupoid $\cB (\cG)$. However, Theorem
 \ref{thm: gpd:quot} already uses that a candidate for the quotient, namely
 $\cG$ exists. Thus the theorem does not provide the existence of the quotient
 without using $\cG$.
\end{remark}

In the proof of Theorem \ref{thm: gpd:quot} it is visible that the quotient of $\cB(\cG)$ was taken with respect to the kernel\footnote{i.e.\ the set $\{g \in \cB (\cG) \mid \ev (g) = 1_x \text{ for some } x \in M\}$ which is a wide subgroupoid of the inner subgroupoid, called a normal Lie subgroupoid [see \cite[Definition 2.2.2]{Mackenzie05General-theory-of-Lie-groupoids-and-Lie-algebroids}).} of the base-preserving morphism $\ev$.
 We refer to \cite[2.2]{Mackenzie05General-theory-of-Lie-groupoids-and-Lie-algebroids} for details on the groupoid quotient by a normal subgroupoid.
 Since $\ev$ is a surjective submersion in the situation of Theorem \ref{thm: gpd:quot}, its kernel is a Lie subgroupoid of $\cB (\cG)$. 
 Finally, the $\alpha$-fibre of the kernel over $m\in M$ is given by 
  \begin{displaymath}
  \ev (\cdot ,m)^{-1} \{1_m\} = \{\sigma \in \Bis (\cG) \mid \sigma (m) = 1_m\}.
  \end{displaymath}
 In the next section we will study the subgroups of $\Bis (\cG)$ which arise from this construction.
 Later on these information will allow us to investigate the groupoid quotients in more detail.

\section{Subgroups and quotients associated to the bisection group}

In this section we study subgroups of the bisections which are associated to a fixed point in the base manifold. 
These subgroups are closely related to the reconstruction result outlined in Theorem \ref{thm: gpd:quot} and will become an important tool to study locally trivial Lie groupoids in Section \ref{sect: loc:triv} and \ref{sec:application_to_integration_of_extensions}.

As before, (unless stated explicitly otherwise) we shall assume that $\cG = (G \toto M)$ is a locally metrisable Lie groupoid over a compact base $M$ which admits an adapted local addition.

\begin{definition}\label{setup: pt:sbgp} Let
 $\cG = (G \toto M)$ be a locally convex Lie groupoid (we require neither that $M$ is be compact or finite dimensional nor that $G$ admits
 a local addition). Fix $m \in M$ and denote by $\Vtx{m}(\cG)$ the vertex subgroup
 of the groupoid $\cG$. There are now two subsets of $\Bis (\cG)$
 whose elements are characterised by their value at $m$
 \begin{align*}
  \Loop{m} (\cG) &\coloneq \{\sigma \in \Bis (\cG) \mid \sigma (m) \in \Vtx{m}(\cG) = \alpha^{-1} (m) \cap \beta^{-1} (m)\} ,\\
  \Bisf{m}(\cG) &\coloneq \{\sigma \in \Bis (\cG) \mid \sigma (m) = 1_m\}.
 \end{align*}
 Clearly $\Bisf{m}(\cG) \subseteq \Loop{m} (\cG)$ and both are subgroups of
 $\Bis (\cG)$ by definition of the group operation (see \eqref{eq: BISGP1} and \eqref{eq: BISGP2}).
\end{definition}

 Note that $\Bisf{m}(\cG)$
 is a normal subgroup of $\Loop{m} (\cG)$ as for $\sigma \in \Bisf{m}(\cG)$ and
 $\tau \in \Loop{m} (\cG)$ we have
 \begin{align*}
  \tau \star \sigma \star \tau^{-1} (m) &= \tau \star (\sigma \circ \beta \circ \tau \cdot \tau) (m) = \tau (\underbrace{\beta \circ \sigma \circ \beta \circ \tau (m)}_{=m} )\cdot (\sigma ( \underbrace{\beta \circ \tau (m)}_{=m}) \cdot \tau (m) \\
  &= \tau (m) \cdot 1_m \cdot \tau^{-1} (m) = 1_m.
 \end{align*}

 We will now investigate the subgroups from Definition \ref{setup: pt:sbgp} in the case that $\Bis (\cG)$ is a Lie
 group. 
 Thus $M$ will be assumed to be compact, whence Lemma \ref{lem: vertex:Lie} implies that the vertex group $\Vtx{m} (\cG)$ of $\cG = (G \toto M)$ is a submanifold of
 $G$ and in particular a Lie group.
 
 \begin{proposition}\label{prop: subgp:vertex}
 Fix some $m \in M$.  
  \begin{enumerate}
   \item The group $\Loop{m} (\cG)$ is a Lie subgroup of $\Bis (\cG)$ and as a submanifold in $\Bis (\cG)$ it is of finite codimension.
   \item The map $\ev_m \colon \Bis (\cG) \rightarrow \alpha^{-1} (m)$ restricts to a Lie group morphism $\psi_m \colon \Loop{m} (\cG) \rightarrow \Vtx{m}(\cG)$ whose kernel is $\Bisf{m}(\cG)$.
   Moreover, $\psi_m$ is a submersion.
   \item The group $\Bisf{m}(\cG)$ is a split Lie subgroup of $\Loop{m} (\cG)$ and a split Lie subgroup of $\Bis (\cG)$.
         It is modelled on $\Gamma (1^*T^\alpha G)_{m} = \{X \in \Gamma (1^* T^\alpha G) \mid X (1_m) = 0_{1_m}\}$.
   \item If $\cG$ is a Banach-Lie groupoid then $\Bisf{m}(\cG)$ is a co-Banach submanifold in $\Loop{m} (\cG)$ and also in $\Bis (\cG)$. 
  \end{enumerate}
\end{proposition}

\begin{proof}
 \begin{enumerate}
  \item Recall that by Corollary \ref{cor: evx:subm} the map $\ev_m \colon \Bis (\cG) \rightarrow \alpha^{-1} (m)$ is a submersion.
  Moreover, the vertex group $\Vtx{m}(\cG)$ is a submanifold of finite codimension of $\alpha^{-1} (m)$ by Lemma \ref{lem: vertex:Lie},
  Thus $\Loop{m} (\cG) = \ev_m^{-1} (\Vtx{m}(\cG))$ is a submanifold of finite codimension in $\Bis (\cG)$ by \cite[Theorem C]{hg2015}.
   Now $\Loop{m} (\cG)$ is a subgroup and a split submanifold of $\Bis (\cG)$ whence a split Lie subgroup.
  \item 
  To see that $\psi_m$ is also a group homomorphism we pick $\sigma, \tau \in \Loop{m} (\cG)$ and compute 
    \begin{displaymath}
     \psi_m (\sigma \star \tau) = \psi_m ((\sigma \circ \beta \circ \tau) \cdot \tau) = \sigma (\beta (\tau (m)) \cdot \tau (m) = \sigma (m) \cdot \tau (m) = \psi_m (\sigma) \cdot \psi_m (\tau).  
    \end{displaymath}
  As $\ev_m (\sigma) = 1_m$ if and only if $\sigma (m) = 1_m$, we see that $\Bisf{m}(\cG)$ is the kernel of $\psi_m$.  
  Having applied \cite[Theorem C]{hg2015} in part a), we observe that this also entails that $\psi_m$ is a submersion. 
  \item By part b) the subgroup $\Bisf{m}(\cG)$ is the preimage of a singleton under a submersion, whence a split submanifold of $\Loop{m} (\cG)$ by the regular value theorem \cite[Theorem D]{hg2015}.
  In particular, $\Bisf{m}(\cG)$ becomes a split Lie subgroup of $\Loop{m} (\cG)$.
  Since $\Loop{m} (\cG)$ is a split submanifold in $\Bis (\cG)$, \cite[Lemma 1.4]{hg2015} yields that $\Bisf{m}(\cG)$ is a split submanifold of $\Bis (\cG)$ and thus a split Lie subgroup of $\Bis (\cG)$.
  
  Recall that by the regular value theorem the tangent space of $\Bisf{m}(\cG)$ at $1$ is the kernel $\Ker T_1 \ev_m \subseteq T_1 \Bis (\cG) \cong \Gamma(1^*T^\alpha G)$.
  Taking identifications we compute $\Ker T_1 \ev_m$ as a subspace of $\Gamma (1^*T^\alpha G)$.
  On the level of isomorphism classes of curves the isomorphism $\varphi_\cG \colon T_1 \Bis (\cG) \rightarrow \Gamma (1^*T\alpha G)$ is given by $\varphi_\cG ([t \mapsto c(t)]) = (m \mapsto [t \mapsto c^\wedge (t,m)])$ (cf.\cite[Remark 4.1]{SchmedingWockel14}).
  Hence $\varphi_\cG$ identifies $T_1\ev_m$ with $\Gamma (1^*T^\alpha G) \rightarrow G, X \mapsto X(m)$ whose kernel is $\Gamma (1^*T^\alpha G)_{m}$.  
  \item If $\cG$ is a Banach-Lie groupoid then $\alpha^{-1} (m)$ and thus also $\Vtx{m}(\cG)$ are manifolds modelled on Banach-spaces.
  In this situation the regular value theorem implies that $\Bisf{m}(\cG)$ is a co-Banach submanifold of $\Loop{m} (\cG)$.
  Since $\Loop{m} (\cG)$ is of finite codimension in $\Bis (\cG)$, we deduce that $\Bisf{m}(\cG)$ is a co-Banach submanifold of $\Bis (\cG)$. 
 \end{enumerate}
\end{proof}

An important property of the Lie subgroups constructed in Proposition \ref{prop: subgp:vertex} is that they are regular as Lie groups 
(we recall the definition of regularity for Lie groups in Appendix \ref{Appendix:  MFD}).
Namely, the subgroups $\Bisf{m}(\cG)$ and $\Loop{m} (\cG)$ will be regular Lie groups if $\Bis (\cG)$ is a regular Lie group.
Let us first recall when bisection groups are regular.\footnote{At the moment, no example of a non-regular Lie group modelled on a space with suitable completeness properties (i.e.\ Mackey completeness) is known.}

\begin{remark}\label{rem: reg:sbgp}
 In \cite[Section 5]{SchmedingWockel14} we established $C^k$-regularity of $\Bis (\cG)$ if $\cG$ is either 
  \begin{enumerate}
   \item a Banach-Lie groupoid (then $\Bis (\cG)$ is $C^0$-regular).
   \item or a locally-trivial Lie groupoid whose vertex groups are locally exponential $C^k$-regular Lie groups (then $\Bis (\cG)$ is $C^k$-regular) 
  \end{enumerate}
  \end{remark}

 We will now prove that $\Bisf{m}(\cG)$ and $\Loop{m} (\cG)$ inherits the regularity properties from $\Bis (\cG)$.

\begin{proposition}\label{prop: stabsubgp:reg}
 Let $\cG = (G \toto M)$ be a Lie groupoid and fix $m \in M$. \
 Assume that $\Bis (\cG)$ is $C^k$-regular for some $k \in \N_0 \cup \{\infty\}$, e.g.\ in the situation of Remark \ref{rem: reg:sbgp}.
 Let $H$ be either $\Bisf{m}(\cG)$ and  $\Loop{m} (\cG)$ and consider $\eta \in C^k ([0,1],\Lf (H))$.
 \begin{enumerate}
  \item The solution $\gamma_\eta$ of the initial value problem
  \begin{equation}\label{eq: reg:eq}
      \begin{cases}
       \gamma' (t) &= \gamma (t).\eta (t) \quad \forall t \in [0,1] \\
       \gamma (0) &= 1 
       \end{cases}
     \end{equation} 
   in $G$ takes its image in $H$.   
  \item The Lie groups $\Bisf{m}(\cG)$ and  $\Loop{m} (\cG)$ are $C^k$-regular
  \end{enumerate} 
\end{proposition}

\begin{proof}
 \begin{enumerate}
  \item Fix $\eta \in C^{k} ([0,1] , \Lf (H))$ together with the evolution $\gamma_\eta \colon [0,1] \rightarrow \Bis (\cG)$ of $\eta$.
  The composition $\ev_m \circ \gamma_\eta$ yields a smooth curve in $\alpha^{-1} (m)$.
  As $\gamma_\eta$ solves \eqref{eq: reg:eq} we already know that $\ev_m \circ \gamma_\eta (0) = 1_m$. 
  Now we consider both subgroups separately:\medskip
  
  \textbf{Case $H= \Bisf{m}(\cG)$.} By definition, $\gamma_\eta$ will take its image in $\Bisf{m}(\cG)$ if $\ev_m \circ \gamma_\eta (t) = 1_m, \forall t \in [0,1]$, i.e.\ we have to prove that $(\ev_m \circ \gamma_\eta)'(t) = 0 \in T_{\gamma_\eta (t)}^\alpha G,\ \forall t \in [0,1]$. 
   To this end fix $t \in [0,1]$ and a curve $c_{t,\eta} \colon ]-\varepsilon , \varepsilon [ \rightarrow \Bisf{m}(\cG)$ such that $c_{t,\eta} (0) = 1$ and $\eta (t)$ (as an element in the tangent space $T_1 \Bisf{m}(\cG)$) coincides with the equivalence class $[s \mapsto c_{t,\eta} (s)]$.
  Now we compute 
    \begin{align}
     (\ev_m \gamma_\eta )' (t) &= T\ev_m (\gamma_\eta' (t)) \stackrel{\eqref{eq: reg:eq}}{=} T\ev_m (\gamma_\eta (t) . \eta (t)) = T (\ev_m \circ \lambda_{\gamma_\eta (t)}) (\eta (t)) \notag\\
			       &= [s\mapsto \ev_m (\gamma_\eta (t) \star c_{t,\eta} (s))] =[s\mapsto \gamma_\eta (t) (\beta \circ c_{t,\eta} (s)(m)) \cdot c_{t,\eta} (s)(m)] \label{eq: logderiv} \\
			       &=  [s\mapsto \gamma_\eta (t)(m)\cdot 1_m] = 0 \in T^\alpha_{\gamma_{\eta} (t)(m)}G \notag
    \end{align}
 In passing from the second to the last line, we have used $c_{t,\eta} (s) \in \Bisf{m}(\cG)$, whence $c_{t,\eta} (s)(m) = 1_m$. 
 We finally conclude that $\gamma_\eta (t) \in \Bisf{m}(\cG)$ for all $t \in [0,1]$ if $\eta \in C^k ([0,1], \Lf (\Bisf{m}(\cG))$. 
 \medskip
 
 \textbf{Case $H=\Loop{m} (\cG)$.} We need to show that for all $t \in [0,1]$ we have $\gamma_\eta (t) \in \Vtx{m}(\cG)$, i.e.\ that $\beta \circ \ev_m \circ \gamma_\eta (t) = m$. 
  Since $\ev_m (\gamma_\eta (0))=1_m$ this will follow from $(\beta \circ \ev_m \circ \gamma_\eta)'(t) =0$ for all $t \in [0,1]$.
  To this end fix again $t \in [0,1]$ and a curve $c_{t,\eta} \colon ]-\varepsilon , \varepsilon [ \rightarrow \Loop{m} (\cG)$ with $c_{t,\eta} (0) = 1$ and $\eta (t) = [s\mapsto c_{t,\eta} (s)]$.
  Computing as in \eqref{eq: logderiv} we obtain 
    \begin{align*}
     (\beta \circ \ev_m \circ \gamma_\eta)'(t) &\stackrel{\eqref{eq: logderiv}}{=} [s \mapsto \beta \left(\gamma_\eta (t) (m) \cdot c_{t,\eta} (s) (m)\right)] 
										    = [s \mapsto \beta \left(\gamma_\eta (t) (m)\right)] = 0 \in T_{\beta (\gamma_\eta (t)(m))} M							    . 
    \end{align*}
  We can thus conclude that $\gamma_\eta (t) \in \Loop{m} (\cG)$ for all $t \in [0,1]$ if $\eta \in C^k ([0,1], \Lf (\Loop{m} (\cG)))$.
   \item Proposition \ref{prop: subgp:vertex} c) asserts that the subgroups $\Bisf{m}(\cG)$ and  $\Loop{m} (\cG)$ are closed subgroups of the $C^k$-regular Lie group $\Bis (\cG)$. 
   By part (a), the Lie groups $\Bisf{m}(\cG)$ and  $\Loop{m} (\cG)$ are $C^k$-semiregular.
   Thus Lemma \ref{lem: semisub:reg} proves that $\Bisf{m}(\cG)$ and  $\Loop{m} (\cG)$ are $C^k$-regular.
 \end{enumerate}
\end{proof}

\begin{example}\label{exemp:Diff_m_is_Lie_subgroup}
 Let $M$ be a compact manifold and consider the pair groupoid $\cP (M) = (M \times M \toto M)$. 
 The vertex group $\Vtx{m}(\cP(M))$ for $m \in M$ is just $\{ (m,m)\}$, whence $\Loop{m} (\cP (M)) = \Bis (\cP (M))_m$. 
 Then the map $(\pr_2)_* \colon \Bis (\cP (M)) \rightarrow \Diff (M) , \sigma \mapsto \pr_2 \circ \sigma$ is an isomorphism of Lie groups.
 By construction, this restricts to an isomorphism 
  \begin{displaymath}
   \Loop{m} (\cP(M)) = \Bis (\cP(m))_m \cong \Diff_m (M) \coloneq \{\varphi \in \Diff (M) \mid \varphi (m) = m\}.
  \end{displaymath}
 In particular, we infer from Proposition \ref{prop: subgp:vertex} and Proposition \ref{prop: stabsubgp:reg} that $\Diff_m (M)$ is a regular and split Lie subgroup of $\Diff (M)$. 
\end{example}

\begin{proposition}\label{prop: quot:fibre} 
Let $\cG = (G \toto M)$ be a Lie
 groupoid. For $m \in M$ we endow the right coset space
 $\Bis (\cG) / \Bisf{m}(\cG)$ with the quotient topology induced by
 $\Bis (\cG)$. Then the map $\ev_m$ induces homeomorphisms 
\begin{enumerate}
  \item a homeomorphism $\widetilde{\ev_m}$ of $\Bis (\cG) / \Bisf{m}(\cG)$ onto $\ev_m (\Bis (\cG)) \opn \alpha^{-1} (m)$,
  \item an isomorphism of topological groups $e_m$ of $\Lambda_m \coloneq \Loop{m} (\cG) / \Bis (\cG)$ onto an open subgroup of $\Vtx{m}(\cG)$,
\end{enumerate} 
 Moreover, $\Bis (\cG) / \Bisf{m}(\cG) $ and $\Loop{m} (\cG) / \Bisf{m}(\cG)$ carry unique manifold structures turning the canonical quotient maps into submersions.
 
 If there is a bisection through every arrow in $G$, e.g.\ $\cG$ is $\alpha$-connected, then  $\Bis (\cG)/\Bisf{m}(\cG) \cong \alpha^{-1} (m)$ as manifolds and $\Loop{m} (\cG) / \Bisf{m}(\cG) \cong \Vtx{m}(\cG)$ as Lie groups.
\end{proposition}

\begin{proof}
 
  \begin{enumerate}
   \item By definition, the quotient topology turns $q_m \colon \Bis (\cG) \rightarrow \Bis (\cG) / \Bisf{m}(\cG) , \sigma \mapsto \sigma \Bisf{m}(\cG)$
 into a quotient map.
 Recall from Corollary \ref{cor: evx:subm}
 that $\ev_m \colon \Bis (\cG) \rightarrow \alpha^{-1} (m)$ is a submersion, whence its image in the $\alpha$-fibre is open. For $\tau \in \Bis (\cG)_{m}$ we observe $\ev_m (\sigma \star \tau) = \sigma \star \tau (m) = \sigma (\beta (\tau (m))) \tau (m) = \sigma (\beta (1_m)) 1_m = \sigma (m) = \ev_m (\sigma)$.
 Hence $\ev_m$ is constant on right cosets and $\ev_m$ factors
 through
 $h_m \colon \Bis (\cG) / \Bisf{m}(\cG) \rightarrow \im \ev_m \opn \alpha^{-1} (m),\ \sigma \Bis (\cG) \mapsto \ev_m (\sigma)$.
 Now $h_m$ is continuous since
 $\ev_m = h_m \circ q_m$ is continuous.
 
 Let us prove that for $\sigma, \tau \in \Bis (\cG)$ with
 $\ev_m (\sigma ) = \ev_m (\tau)$ we have
 $\sigma^{-1} \star \tau \in \Bisf{m}(\cG)$. Using the formulae \eqref{eq:
 BISGP1} and \eqref{eq: BISGP2} for the group operations of $\Bis (\cG)$ we
 obtain
 \begin{align*}
  \sigma^{-1} \star \tau (m) = \sigma^{-1} (\beta (\tau (m))) \tau (m) = \sigma^{-1}( \beta (\sigma (m))) \sigma (m) = \iota (\sigma (m))\sigma(m) = 1_m.
 \end{align*}
 Hence $\sigma^{-1} \star \tau \in \Bisf{m}(\cG)$ if $\sigma (m) = \tau (m)$ and
 in this case we see $\sigma \Bisf{m}(\cG) = \tau \Bisf{m}(\cG)$. This implies
 that $h_m$ is a bijection onto the open set $\im \ev_m$.
 
  We deduce from Corollary \ref{cor: evx:subm}
 that $\ev_m|^{\im \ev_m} \colon \Bis (\cG) \rightarrow \im \ev_m \opn \alpha^{-1} (m)$ is a surjective
 submersion. In particular, $\ev_m$ is open, whence a quotient map and thus
 $q_m = h_m^{-1} \circ \ev_m$ implies that $h_m^{-1}$ is continuous.
 \item By Proposition \ref{prop: subgp:vertex} b) we know that $\ev_m$ induces a Lie group morphism $\psi_m \colon \Loop{m} (\cG) \rightarrow \Vtx{m}(\cG)$ which is a submersion. 
 Its kernel is the normal Lie subgroup $\Bisf{m}(\cG)$.
 Thus $\Loop{m} (\cG) / \Bisf{m}(\cG)$ with the quotient topology becomes a topological group such that $\psi_m$ descents to an isomorphism of topological groups onto $\ev_m (\Loop{m} (\cG)) \opn \Vtx{m}(\cG)$.
 Endow the quotient with the manifold structure turning the isomorphism into an isomorphism of Lie groups. 
 Then the canonical quotient map becomes a submersion as a composition of a diffeomorphism and the submersion $\psi_m$.
 \end{enumerate}
 The manifold structure on the open submanifolds $\im \ev_m \opn \alpha^{-1}(m)$ and $\ev_m (\Loop{m} (\cG)) \opn \Vtx{m}(\cG)$ is uniquely determined up to
 diffeomorphism by the property that $\ev_{m}$ is a submersion. This is due to \cite[Lemma 1.9]{hg2015}.

The last assertion follows from part a) and b), since then $\im \ev_m = \alpha^{-1} (m)$ and $\im \psi_m = \Vtx{m}(\cG)$ hold.
  \end{proof}

\begin{example}\label{exmp:evaluation_for_Aut(P)}
 Suppose $\pi\from P\to M$ is a principal $K$-bundle for some locally
 exponential Lie group $K$ with connected $P$. Then the gauge groupoid
 $\op{Gauge}(P):=((P\times P)/K\toto M)$ admits an adapted local addition
 \cite[Proposition 3.14]{SchmedingWockel14} and $\Bis(\op{Gauge}(P))$ is
 naturally isomorphic to $\Aut(P)$. Assume that $K$ and $P$ are locally metrisable, i.e.\ $\op{Gauge}(P)$ is locally metrisable. 
 The source fibre
 $\alpha^{-1}(m)=(P_{m}\times P)/K$ of $\op{Gauge}(P)$ is diffeomorphic to $P$
 by choosing $o\in P_{m}$ and mapping $\langle p,q\rangle$ to
 $q. (p^{-1}\cdot o)$. Here we use $p^{-1}\cdot o$ as the suggestive notation
 for the element $k\in K$ that satisfies $p.k=o$. With respect to these
 identification the evaluation $\ev_{m}$ turns into the evaluation map
 \begin{equation*}
  \ev_{o}\from \Aut(P)\to P,\quad \varphi\mapsto \varphi(o).
 \end{equation*}
 Consequently,
 \begin{equation*}
 \Aut_{o}(P):=\ev^{-1}(o)=\{f\in \Aut(P)\mid f(o)=o\}
 \end{equation*} 
is a Lie subgroup of $\Aut(P)$ and by
 Proposition \ref{prop: quot:fibre}, $\Aut(P)/\Aut_{o}(P)$ carries a unique smooth
 structure turning the induced map $[\varphi]\mapsto \varphi(p)$ into a
 diffeomorphism. So we may view $P$ a a homogeneous space for its automorphism
 group. In particular, this applies to the trivial bundle, yielding a smooth
 structure on $\Diff(M)/ \Diff_{m}(M)$ and a diffeomorphism
 $\Diff(M)/ \Diff_{m}(M)\cong M$.
\end{example}

By now, the quotients $\Bis (\cG) / \Bisf{m}(\cG)$ and $\Loop{m} (\cG) / \Bisf{m}(\cG)$ carry a manifold structure which was derived from the manifold structure of the $\alpha$-fibre to the quotient.
However, if the Lie groupoid $\cG$ is a Banach-Lie groupoid then the homogeneous space $\Bis (\cG) / \Bisf{m}(\cG)$ already carries a natural manifold structure as a homogeneous space.
This is a consequence of Gl{\"o}ckners inverse function theorem (see the next Lemma for references and details).
Again this manifold structure turns the canonical quotient map into a submersion.

\begin{lemma}
Let $\cG = (G \toto M)$ be a Banach-Lie groupoid and $m \in M$.
Then the homogeneous spaces $\Bis (\cG) / \Bisf{m}(\cG)$ and $\Loop{m} (\cG) / \Bisf{m}(\cG)$ are manifolds and these manifold structures coincides with the ones obtained in Proposition \ref{prop: quot:fibre}.
\end{lemma}

\begin{proof}
 Combining Proposition \ref{prop: stabsubgp:reg} and Remark \ref{rem: reg:sbgp} we see that $\Bisf{m}(\cG)$ is a $C^0$-regular closed Lie subgroup of $\Bis (\cG)$ and of $\Loop{m} (\cG)$.
 Moreover, $\Bisf{m}(\cG)$ is a co-Banach submanifold in both $\Loop{m} (\cG)$ and $\Bis (\cG)$.
 Thus the homogeneous spaces $\Bis (\cG) /\Bisf{m}(\cG)$ and $\Loop{m} (\cG) / \Bis {\cG}_m$ carry manifold structures by \cite[Theorem G (a)]{hg2015}. 
 Furthermore, since $\Bisf{m}(\cG)$ is a normal Lie subgroup, the manifold $\Loop{m} (\cG)/\Bis {\cG}_m$ becomes a Lie group.
 In both cases this manifold structure turns the canonical quotient map into a submersion.
 By the uniqueness assertion in Proposition \ref{prop: quot:fibre} the manifold structures on the homogeneous spaces must coincide. 
\end{proof}

The interesting feature of the manifold structures obtained on the homogeneous spaces $\Bis (\cG) / \Bisf{m}(\cG)$ and $\Loop{m} (\cG)/\Bisf{m}(\cG)$ for Banach-Lie groupoids is exactly that it coincides with the structure induced by the fibre.
Hence, under some assumptions, we can endow the quotient $\Bis (\cG) / \Bis (\cG)_*$ with a manifold structure which does not a priori use the manifold structure on the $\alpha$-fibre.
We will apply these results in the next section after we compile some more facts on natural group actions on the quotient $\Bis (\cG) / \Bisf{m}(\cG)$.

\begin{lemma}\label{lem: eff:act}
 Let $\cG = (G \toto M)$ be a Lie groupoid, fix $m \in M$ and define $\Lambda_m \coloneq \Loop{m} (\cG) / \Bisf{m}(\cG)$.
 Then the quotient $\Bis (\cG) / \Bisf{m}(\cG)$ admits left- and right group actions
 \begin{align*}
  \lambda_{\Bis (\cG)} \colon \Bis (\cG) \times (\Bis (\cG) / \Bisf{m}(\cG)) &\rightarrow \Bis (\cG) / \Bisf{m}(\cG) , (\sigma , \tau \Bisf{m} (\cG)) \mapsto (\sigma \star \tau )\Bisf{m} (\cG) \\
  \rho_{\Lambda_m } \colon (\Bis (\cG) / \Bisf{m}(\cG)) \times \Lambda_m &\rightarrow \Bis (\cG) / \Bisf{m}(\cG) , (\sigma \Bisf{m} (\cG) , \gamma \Bisf{m}(\cG)) \mapsto (\sigma \star \gamma) \Bisf{m} (\cG)
 \end{align*}
 which commute, i.e.\ $\lambda_{\Bis (\cG)} (\sigma , \cdot) \circ \rho_{\Lambda_m} ([\tau], \cdot) = \rho_{\Lambda_m} ([\tau], \cdot) \circ \lambda_{\Bis (\cG)} (\sigma , \cdot)$ for all $\sigma \in \Bis (\cG)$ and $[\tau] \in \Lambda_m$.
\end{lemma}

\begin{proof}
 Observe that $\Bis (\cG)$ acts via left translation on itself and this action descents to a group action on the quotient $\Bis (\cG) / \Bisf{m}(\cG)$.
 Moreover, $\Loop{m} (\cG)$ acts via right translation on $\Bis (\cG)$ and this action descents to a $\Lambda_m = \Loop{m} (\cG) / \Bisf{m}(\cG)$-action on the quotient. 
 The left action by left translation on $\Bis (\cG)$ commutes with the right translation with elements in $\Loop{m} (\cG)$, whence the induced actions on the quotient commute.
\end{proof}
 
 \begin{lemma}\label{lem: eff:act2}
   Let $\cG$ be a transitive Lie groupoid which admits bisections through each arrow. 
   Then $\lambda_{\Bis (\cG)}$ is an effective group action, i.e.\ $\lambda_{\Bis (\cG)} (\sigma,\cdot) = \id_{\Bis (\cG)/\Bisf{m}(\cG)}$ implies $\sigma = 1$.
 \end{lemma}

 \begin{proof}
 The prerequisites imply that $\ev_m$ is a surjective map.
 Consider $\sigma \in \Bis (\cG) \setminus \{1\}$ and choose $n \in M$ such that $\sigma (n) \neq 1_n$. 
 Now $\cG$ is transitive, whence there is $g_n \in \alpha^{-1} (m)$ with $\beta (g) = n$. 
 As $\ev_m$ is surjective we can choose $\tau \in \Bis (\cG)$ with $\tau (m) = g_n$. 
 Arguing indirectly, we assume that $[\sigma \star \tau] = [\tau]$, i.e.\ there is $s \in \Bisf{m}(\cG)$ with $\sigma \star \tau = \tau \star s$.  
 Evaluating in $m$, we use $s \in \Bisf{m}(\cG)$ to obtain 
  \begin{displaymath}
   \sigma (n) \cdot g = \sigma (\beta (\tau (m)) \cdot \tau (m) = (\sigma \star \tau )(m) = (\tau \star s)(m) = \tau (m)\cdot 1_m = g
  \end{displaymath}
 Hence, $\sigma (n) = g\cdot g^{-1} = 1_n$ follows, contradicting our choice of $n$. 
 We conclude that $\lambda_{\Bis (\cG)}$ is effective.
 \end{proof}

\begin{remark}
 In general, the left action $\Lambda_{\Bis (\cG)}$ will not be effective. 
 To see this, we return to the example given in Remark \ref{rem: ev:sur} b): 
 
 Let $M = N \sqcup N'$ be the disjoint union of two non-isomorphic smooth manifolds. 
 Then the pair groupoid $\cP (M)$ is locally trivial, but there are arrows which are not contained in the image of any bisection.
 Fix $m \in N$ and recall that 
  \begin{displaymath}
   \Bis (\cG) \cong \Diff (M) \cong \Diff (N) \times \Diff (N') \text{ and } \Bisf{m}(\cG) \cong \Diff_m (N) \times \Diff (N').
  \end{displaymath}
 Thus $\Bis (\cG) / \Bisf{m}(\cG) \cong \Diff (N)/\Diff_m (N)$. 
 Hence, for $\varphi \in \Diff (N')$ the bisection $\id_N \times \varphi$ acts trivially on $\Bis (\cG) / \Bisf{m}(\cG)$.
 If $N'$ is not the singleton manifold, choose $\varphi \neq \id_{N'}$, to see that $\lambda_{\Bis (\cG)}$ can not be effective.
\end{remark}

\begin{lemma}\label{lem: ev:equiv}
 Let $\cG = (G \toto M)$ be a Lie groupoid, fix $m \in M$ and define $\Lambda_m \coloneq \Loop{m} (\cG) / \Bisf{m}(\cG)$.
 Then the map 
  \begin{displaymath}
   \widetilde{\ev_m} \colon \Bis (\cG) / \Bisf{m} (\cG) \rightarrow \widetilde{\ev_m} (\Bis (\cG) / \Bisf{m}(\cG)) \opn \alpha^{-1} (m), \sigma \Bisf{m} (\cG) \mapsto \sigma (m)
  \end{displaymath}
 (cf.\ Proposition \ref{prop: quot:fibre}) is equivariant with respect to the right $\Lambda_m$-action $\rho_{\Lambda_m}$ and the right $\Vtx{m}(\cG)$-action, i.e.\ for $\tau \in \Loop{m} (\cG)$ and $\sigma \in \Bis (\cG)$ we obtain the formula 
  \begin{displaymath} 
   \widetilde{\ev_m} ((\sigma \star \tau) \Bisf{m} (\cG))) = \widetilde{\ev_m} (\rho_{\Lambda_m} (\sigma \Bisf{m} (\cG) , \tau \Bisf{m} (\cG))) = \widetilde{\ev_m} (\sigma \Bisf{m} (\cG))) \cdot \tau(m)  
  \end{displaymath}
\end{lemma}

\begin{proof}
 Fix $\sigma \in \Bis {\cG}$ and $\tau \in \Loop{m} (\cG)$ and compute 
 \begin{displaymath}
   \widetilde{\ev_m} ((\sigma \star \tau) \Bisf{m} (\cG))) = (\sigma \star \tau )(m) = \sigma (\underbrace{\beta (\tau(m))}_{=m}) \cdot \tau (m) = \sigma (m) \cdot \tau (m) = \widetilde{\ev_m} (\sigma \Bisf{m} (\cG)) \tau (m).
 \end{displaymath}

\end{proof}

\section{Locally trivial Lie groupoids and transitive group actions}\label{sect: loc:triv}

Our aim is now to study the construction of groupoids from their groups of bisections for locally
trivial Lie groupoids. Again we consider in this section only Lie groupoids over a compact manifold
$M$ that admit an adapted local addition. Moreover, we choose and fix a point $m\in M$.

For locally trivial Lie groupoids the $\alpha$-fibre over any point already determines the manifold of arrows.
Hence, the groupoid quotient discussed in Theorem \ref{thm: gpd:quot} of $\cB (\cG)$ is determined by a quotient of the $\alpha$-fibre.
To construct the quotient, one needs to construct the fibre over a point and the vertex group from the group action of $\Bis (\cG)$ on $M$ and the subgroup $\Bisf{m}(\cG)$.
Following Proposition \ref{prop: quot:fibre} these objects can be obtained as certain quotients of the group of bisections.
The idea is now to study similar situation for abstract Lie groups and relate these Lie groups to groups of bisections.
To this end, we define the central notion of this section:

\begin{definition}\label{defn:
 tgpair} Let $\theta \colon K \times M \rightarrow M$ be a transitive
 (left-)Lie group action of a Lie group $K$ modelled on a metrisable space and $H$
 be a subgroup of $K$.
 
 Then we call $(\theta , H)$ a \emph{transitive pair} (over $M$ with base point
 $m$) if the following conditions are satisfied:
 \begin{enumerate}[label=({P\arabic*})]
  \item \label{defn: tgpair_1} the action is smoothly transitive, i.e., the
        orbit map $\theta_m \coloneq \theta (\cdot,m)$ is a surjective
        submersion,
  \item \label{defn: tgpair_2} $H$ is a normal Lie subgroup of the stabiliser
        $\Stab{m}$ of $m$ and this structure turns $H$ into a regular Lie group
        which is co-Banach as a submanifold in $\Stab{m}$.
 \end{enumerate}
 The largest subgroup of $H$ which is a normal subgroup of $K$ is called
 \emph{kernel} of the transitive pair.\footnote{We will see in Proposition
 \ref{prop: kernel} that there exists a kernel for each transitive pair. By
 standard arguments for topological groups, the kernel is a closed subgroup. In
 general this will not entail that the kernel is a closed Lie subgroup (of the
 infinite-dimensional Lie group $K$).} If the action of $K$ on $M$ is also
 $n$-fold transitive, then we call $(\theta,H)$ an \emph{$n$-fold transitive
 pair}.
\end{definition}

\begin{tabsection}
 Transitive pairs are closely related to Klein geometries \cite[Chapter
 3]{Sharpe97Differential-geometry}. Indeed, they can be understood as
 infinite-dimensional Klein geometries for principal bundles. 
 This view motivates the notion of the kernel of a transitive pair.
 We will come back
 to this perspective in Remark \ref{rem:klein_geometries_vs_transitive_pairs}.
 
 A transitive pair will allow us to construct a locally trivial Lie groupoid
 which is related to the group action on $M$. Before we begin with this
 construction, let us first exhibit two examples of transitive pairs.
\end{tabsection}

\begin{example}\label{ex: tgp:bis} 
  \begin{enumerate}
   \item Let $\cG = (G \toto M)$ be a locally trivial Banach-Lie groupoid over a compact manifold $M$.
   By \cite[Proposition 3.12]{SchmedingWockel14} $\cG$ admits an adapted local addition, whence $\Bis (\cG)$ becomes a Lie group.
   Then the Lie group action $\theta \colon \Bis (\cG) \times M \rightarrow M , \theta \coloneq \beta \circ \ev$ induces a submersion $\theta_m = \beta|_{\alpha^{-1} (m)} \circ \ev_m$ by Lemma \ref{lem: BG:loctriv}.
   Observe that $\Stab{m} = \Loop{m} (\cG)$ and $\Bisf{m}(\cG) \subseteq \Loop{m} (\cG)$ is a normal subgroup.
   Combining Proposition \ref{prop: subgp:vertex} and Proposition \ref{prop: stabsubgp:reg} we see that $(\theta , \Bisf{m}(\cG))$ is a transitive pair if the action $\theta$ is transitive.
 
   In general $\theta$ will not be transitive. However, under some mild assumptions, e.g., $M$ being connected or if $\cG$ admits bisections through each arrow (see Lemma \ref{lem: BG:loctriv}), the action will be transitive and we obtain a transitive pair.
 \end{enumerate}
The preceding example motivated the definition of a transitive pair. 
However, one has considerable freedom in choosing the ingredients for such a pair (see also Remark \ref{rem: spc:findim}):
 \begin{enumerate}
   \item[b)] Consider the diffeomorphism group $\Diff (M)$ of a compact and connected manifold $M$.
   Choose a Lie group $B$ modelled on a Banach space and define $K \coloneq \Diff (M) \times B$. 
   Then $K$ becomes a Lie group which acts transitively via $\theta \colon K \times M \rightarrow M , ((\varphi , b),m) \mapsto \varphi (m)$.
   Fix $m \in M$ and observe that $\theta_m$ is a submersion as $\theta_m = \ev_m \circ \pr_1$ and $\ev_m \colon \Diff (M) \rightarrow M$ is a submersion. 
   By construction $\Stab{m} = \Diff_m (M) \times B$ and $H \coloneq K_m$ is a regular (and normal) Lie subgroup of $K_m$ by Proposition \ref{prop: stabsubgp:reg}.\footnote{Here we have used that for the pair groupoid $\cP (M)$ the Lie group $\Diff (M) = \Bis (\cP (M))$ is regular and $\Bisf{m}(\cP(M)) = \Diff_m (M)$. Moreover, $B$ is regular as a Banach Lie group and $\Lf (\Diff (M) \times B) \cong \Lf (\Diff (M)) \times \Lf (B)$.}
   We conclude that $(\ev_m \circ \pr_1 , \Diff_m (M) \times B)$ is a transitive  pair.
  \end{enumerate}
\end{example}

\begin{remark}\label{rem: spc:findim}
 The conditions \ref{defn: tgpair_1} and \ref{defn: tgpair_2} in Definition \ref{defn: tgpair} are quite weak. 
 We illustrate this by rewriting the conditions for finite-dimensional Lie groups:
 
If $K$ is a finite-dimensional Lie group, the conditions \ref{defn: tgpair_1} and \ref{defn: tgpair_2} are equivalent to 
 \begin{enumerate}[label=(Pfin),align=left]
  \item \label{rem: spc:findim_1} $H$ is a normal closed subgroup of the stabiliser $\Stab{m}$ of $m$ under the action $\theta$.
 \end{enumerate}
 
 \begin{proof}
  To see this note that $\Stab{m} = \theta_m^{-1} (m)$ is a closed subgroup of the finite-dimensional Lie group $K$, whence it is a Lie subgroup of of $K$.
  Then $\theta_m$ factors through $K / \Stab{m} \cong M$ and thus \ref{defn: tgpair_1} holds as $K \rightarrow K/\Stab{m}$ is a submersion.
  Note that $H$ is a closed subgroup of $\Stab{m}$ and every Lie subgroup of a finite-dimensional Lie group is co-Banach as a submanifold and a regular Lie group.
  Hence \ref{rem: spc:findim_1} implies \ref{defn: tgpair_2}.
 \end{proof}
 \end{remark}

\begin{remark}\label{rem:base_point_free_version_of_transitive_pairs}
 From a transitive pair $(\theta\from K\times M\to M,H)$ we can construct the
 following normal subgroupoid $N( \theta, H)$ of the action groupoid
 $K\ltimes_{\theta} M$. For each $n\in M$, we choose some $k_{n}\in K$ with
 $\theta(k_{n},m)=n$ and set $H_{n}:=k_{n}\cdot H\cdot k_{n}^{-1}$. Then
 $H_{n}$ is a normal subgroup of $\Stab{m}$ that does not depend on the choice
 of $k_{n}$. Indeed, if $\theta(k'_{n},m)=n$, then we have
 $k_{n}^{-1}k_{n}'\in \Stab{m}$ and thus
 \begin{equation*}
  k_{n}H k_{n}^{-1}=  k_{n}k_{n}^{-1}k_{n}' H (k_{n}')^{-1}k_{n} k_{n}^{-1}= k_{n}' H (k_{n}')^{-1},
 \end{equation*}
 since $H$ is normal in $\Stab{m}$. Moreover,
 \begin{equation*}
  N(\theta, H):=\bigcup_{n\in M} H_{n}\times\{n\}
 \end{equation*}
 is a closed submanifold of $K\times M$, which can bee seen as follows: by \ref{defn: tgpair_1}
 there exist for each $n\in M$ an open neighbourhood $U\se M$ of $n$ such that
 we can choose $k_{p}$ to depend smoothly on $p$ for $p\in U$. Then
 \begin{equation*}
  G\times U\to G\times U,\quad (g,p)\mapsto (k_{p}^{-1}gk_{p},p)
 \end{equation*}
 is a diffeomorphism that maps $N(\theta,N)\cap (G\times U)$ to the submanifold
 $H\times U$ of $G\times M$. Thus $N(\theta,N)$ is a normal Lie subgroupoid of
 $K\ltimes_{\theta} M$.
 
 On the other hand, given a transitive action of $K$ on $M$, each normal Lie
 subgroupoid of $K\ltimes M$ gives rise to a normal subgroup $H$ of $\Stab{m}$,
 and one easily sees that these two constructions are inverse to each other.
 Thus normal Lie subgroupoids of action groupoids are the equivalent
 reformulation of transitive pairs that do not require to fix a point $m\in M$.
 However, it will be analytically much easier to work with transitive pairs
 (see Remark \ref{rem:quotient_of_action_groupoid}) and to have in mind that
 the choice of a base point does not matter.
\end{remark}

 Now we associate a locally trivial Lie groupoid to a transitive  pair. 
 To this end, we will first construct a principal bundle which will then give rise to the desired locally trivial Lie groupoid. 

\begin{proposition}\label{prop: const:pbund}
 Let $(\theta , H)$ be a transitive  pair.  
 Then the quotients $K/H$ and $\Lambda_m \coloneq \Stab{m}/H$ are Banach manifolds. 
 Moreover, the map $\theta_m$ induces a $\Lambda_m$-principal bundle $\pi \colon K/H \rightarrow M, kH \mapsto \theta (k,m)$.
\end{proposition}

 To prove Proposition \ref{prop: const:pbund}, the following Lemma deals with some needed technical details first (compare \cite[Example D.4]{Wockel13Infinite-dimensional-and-higher-structures-in-differential-geometry}).
 
 \begin{lemma}\label{lem: pBundle}
  Let $(\theta, H)$ be a transitive  pair. 
  The group action $\theta$ induces a $\Stab{m} = \theta_m^{-1} (m)$-principal bundle $\theta_m \colon K \rightarrow M$. 
  Canonical bundle trivialisations for this bundle are given by 
    \begin{equation}\label{eq: psub:triv}
     \theta_m^{-1} (U_i) \rightarrow U_i \times \Stab{m},\quad g \mapsto (\theta_m (g), \sigma_i (\theta_m (g))^{-1} \cdot g)
    \end{equation}
 where $(\sigma_i \colon U_i \rightarrow \theta_m^{-1} (U_i))_{i \in I}$ is a section atlas of $\theta_m$.
 \end{lemma}

\begin{proof}
 Note that $\theta_m$ is a surjective submersion. Thus \cite[Theorem D]{hg2015} implies that $\Stab{m}$ is a split Lie subgroup in $K$
 and $\theta_m$ descents to a homeomorphism $K/\Stab{m} \cong M$. Identify $K/\Stab{m}$ with the manifold $M$.
 Clearly by conjugation $\Stab{m} \cong \theta_m^{-1} (n)$ for all $n \in M$. 
 It is now trivial to check that \eqref{eq: psub:triv} yields bundle trivialisations whose trivialisation changes are $\Stab{m}$-torsor isomorphisms.
\end{proof}

\begin{proof}[of Proposition \ref{prop: const:pbund}]
 The manifold $M$ is finite-dimensional, whence $\Stab{m}$ is a Lie subgroup of finite codimension (as a submanifold).
 By \ref{defn: tgpair_2} $H$ is a co-Banach submanifold in $\Stab{m}$ and thus $H$ is also a co-Banach submanifold of $K$ by \cite[Lemma 1.4]{hg2015}.
 In particular, $H$ is a Lie subgroup of $K$.
 As $H$ is a regular Lie group by \ref{defn: tgpair_2}, we can apply \cite[Theorem G (a)]{hg2015} to obtain a manifold structure on the quotients
 \begin{displaymath}
   p_m \colon K \rightarrow K/H \text{ and } q_m \colon \Stab{m} \rightarrow \Stab{m} / H =: \Lambda_m
 \end{displaymath}
 turning the projections into submersions. 
 Moreover, we deduce from \cite[Theorem G]{hg2015} that $K/H$ is a Banach manifold, $\Lambda_m$ is a Banach-Lie group, $q_m$ is a morphism of Lie groups and $\Lambda_m$ acts on $K/H$ via 
  \begin{displaymath}
   \rho_{\Lambda_m} \colon K/H \times \Lambda_m \rightarrow K/H , (gH,\lambda H) \mapsto (g\cdot \lambda)H. 
  \end{displaymath}
 The subgroup $H$ is contained in $\Stab{m}$, whence $\theta_m$ induces a map $\pi \colon K/H \rightarrow M$ which satisfies $\pi \circ p_m = \theta_m$.
Now $p_m$ and $\theta_m$ are submersions, whence $\pi$ is smooth with surjective tangent map at every point (cf.\ \cite[p.2 and Lemma 1.8]{hg2015}). 
 Since $M$ is finite-dimensional, \cite[Theorem A]{hg2015} implies that $\pi$ is a submersion.
 The action $\theta$ is transitive and thus $\pi$ is surjective submersion.
 Note that the $\pi$-fibre over a point $n \in M$ is given by 
  \begin{displaymath}
   \pi^{-1} (n) = p_m (\theta_m^{-1} (n)) = \{gH \in K/H \mid \theta (g,m) = n\}.
  \end{displaymath}
 Now it is easy to see that the $\pi$-fibres coincide with the orbits of the action $\rho_{\Lambda_m}$ and the $\rho_{\Lambda_m}$-action on the fibres is free, i.e.\ $\pi^{-1} (n)$ is a $\Lambda_m$-torsor for each $n \in M$.
 
 To turn $\pi \colon K/H \rightarrow M$ into a $\Lambda_m$-principal bundle we will now prove that the change of trivialisations induce $\Lambda_m$-torsor isomorphisms.
 Recall that by Lemma \ref{lem: pBundle} the bundle $\theta_m \colon K \rightarrow M$ is a $\Stab{m}$-principal bundle.
 The trivialisations \eqref{eq: psub:triv} descent to $K/H$ via 
  \begin{equation}\label{eq: pbun:triv}
   \kappa_i \colon \pi^{-1} (U_i) \rightarrow U_i \times \Lambda_m ,\quad (gH) \mapsto (\pi (gH), ((\sigma_i (\pi (gH)))^{-1} \cdot g)H).
  \end{equation}
  For each $i  \in I$  we obtain a commutative diagram 
 \begin{displaymath}
  \begin{xy}
  \xymatrix{
      K \ar[dd]^{\theta_m} \ar[rd]^{p_m}  &    		& \theta_m^{-1} (U_i) \ar[ll]_{\supseteq} \ar[rr]^{\eqref{eq: psub:triv}} \ar[d] && U_i \times \Stab{m} \ar[ld]^{\id_{U_i} \times q_m} \ar[dd]^{\pr_1}\\
			     & K/H \ar[ld]^\pi 	& \pi^{-1} (U_i) \ar[l]_{\supseteq} \ar[r]^{\kappa_i} \ar[d] & U_i \times \Lambda_i \ar[rd]^{\pr_1} & \\
      M            & &   U_i \ar[ll]_{\supseteq} \ar@2{-}[rr] & & U_i   
  }
   \end{xy}
 \end{displaymath}
 and the sets $\pi^{-1} (U_i), i \in I$ cover $K/H$.  
 Note that the trivialisation changes descent to $\Lambda_m$-torsor isomorphisms, whence the $\kappa_i$ form an atlas of $\Lambda_m$-principal bundle trivialisations for $K/H$.
 Summing up, we have constructed a principal $\Lambda_m$-bundle $\pi \colon K/H \rightarrow M$.
\end{proof}

\begin{example}\label{ex: Bis:pbun}
 Let $\cG = (G\toto M)$ be a locally trivial Banach-Lie groupoid such that there is a bisection through each element in $G$.
 Consider the transitive pair $(\beta \circ \ev_m , \Bisf{m}(\cG))$ discussed in Example \ref{ex: tgp:bis} a).
 The $m$-stabiliser of $\beta \circ \ev$ coincides with $\Loop{m} (\cG)$ and Lemma \ref{lem: pBundle} yields the $\Loop{m} (\cG)$-bundle $\beta \circ \ev_m \colon \Bis (\cG) \rightarrow M$.
 Moreover, the $\Lambda_m$-principal bundle constructed in Proposition \ref{prop: const:pbund} is $\pi \colon \Bis (\cG)/\Bisf{m}(\cG) \rightarrow M, \sigma \Bisf{m}(\cG) \mapsto \beta (\sigma (m))$ with $\Lambda_m = \Loop{m} (\cG) /\Bisf{m}(\cG)$.
 Then the $\alpha$-fibre through $m$ yields a $\Vtx{m}(\cG)$-principal bundle $\beta|_{\alpha^{-1} (m)} \colon \alpha^{-1} (m) \rightarrow M$. 
 Now Proposition \ref{prop: quot:fibre} allows us to identify $\Lambda_m$ and $\Vtx{m}(\cG)$ and Lemma \ref{lem: ev:equiv} shows that $\ev_m \colon \Bis (\cG) \rightarrow \alpha^{-1} (m)$ descends to a $\Lambda_m$-principal bundle isomorphism 
  \begin{equation*}%
  \begin{xy}
  \xymatrix{
     \Bis (\cG) / \Bisf{m} (\cG) \ar[rr]^-{\widetilde{\ev_m}} \ar[rd]_\pi  &     &  \alpha^{-1} (m) \ar[dl]^{\beta|_{\alpha^{-1} (m)}}  \\
                             &  M  &
  }
\end{xy}.
  \end{equation*}
\end{example}

\begin{definition}\label{defn:
 Liegp:loctgpd} Let $(\theta, H)$ be a transitive pair with associated
 $\Lambda_m$-principal bundle $\pi \colon K/H \xrightarrow{\Lambda_m} M$. As
 principal bundles correspond to locally trivial groupoids, this allows us to
 construct a gauge groupoid
 \begin{displaymath}
  \cR(\theta, H) \coloneq \left(\vcenter{ \xymatrix{ \frac{K/H \times
  K/H}{\Lambda_m} \ar@<1.2ex>[d]^{\beta _{\cR}}\ar@<-1ex>[d]_{\alpha_{\cR}} \\
  M. }}\right)
 \end{displaymath}
 with $\alpha_{\cR} (\langle gH, kH\rangle) = \pi (kH)$ and
 $\beta_{\cR} (\langle gH, kH\rangle) = \pi (gH)$. 
\end{definition}

 Note that we will work with the gauge groupoid $\cR (\theta, H)$ associated to the transitive pair $(\theta,H)$ and not with the principal bundle (although many constructions will be carried out in the context of principal bundles).
 The reason for this is that many interesting maps considered later can not be described as morphisms of principal bundles (with fixed structure group).
 However, one can treat these maps as morphisms of (locally trivial) Lie groupoids over the fixed base $M$. 
 Hence we prefer the groupoid perspective.

 Before we continue, let us record some technical details on the construction of the Lie groupoid $\cR (\theta,H)$.
 
 \begin{remark}\label{rem: Liegp:loctgpd} Let $(\theta, H)$ be a transitive pair with $\theta \colon K \times M \rightarrow M$.
 \begin{enumerate}
  \item Observe that as $K/H$ is a Banach manifold, the gauge groupoid $\cR(\theta,H)$ is a Banach-Lie groupoid
 and thus $\cR(\theta, H)$ admits an adapted local addition by \cite[Proposition 3.12]{SchmedingWockel14}.
 Moreover, the gauge groupoid $\cR (\theta, H)$ is source connected if and only if $K/H$ is connected.
  \item Choose a section atlas $(\sigma_i,U_i)_{i \in I}$ of $\theta_m \colon K \xrightarrow{\Stab{m}} M$ as in \eqref{eq: psub:triv}.
  This atlas induces a section atlas $s_i \coloneq p_m \circ \sigma_i \colon U_i \rightarrow \pi^{-1} (U_i) \subseteq K/H$
 of the bundle $\pi \colon K/H \xrightarrow{\Lambda_m} M$ (cf.\ the proof of Proposition \ref{prop: const:pbund}). Using these section,
 we identify the bisections of $\cR (\theta,H)$ with bundle automorphisms
 via the Lie group isomorphism from \cite[Example 3.16]{SchmedingWockel14}
 \begin{equation}\label{eq: ident:AutBis}
  \Aut (\pi \colon K/H \rightarrow M) \rightarrow \Bis (\cR (\theta,H)), \quad f \mapsto (m \mapsto \langle f(s_i (m)) , s_i (m)\rangle , \text{ if } m\in U_i.
 \end{equation}
 For later use we recall that the bundle trivialisations \eqref{eq: pbun:triv}
 induce charts for the manifold $\frac{K/H \times K/H}{\Lambda_m}$ via
 \begin{equation}\label{eq: Gau:triv}
  \frac{\pi^{-1} (U_i) \times \pi^{-1} (U_j)}{\Lambda_m} \rightarrow U_i \times U_j \times \Lambda_m , \quad\langle p_1 , p_2\rangle \mapsto (\pi(p_1) , \pi (p_2), \delta (\sigma_i (\pi (p_1)), p_1) \delta (\sigma_j (\pi(p_2),p_2)^{-1}).
 \end{equation}
 where $\delta \colon K/H \times_\pi K/H \rightarrow \Lambda_m$ is the smooth
 map mapping a pair $(p,q)$ to the element $p^{-1} \cdot q \in \Lambda_m$ which
 maps $p$ to $q$ (via the $\Lambda_m$-right action). 
 \end{enumerate}
\end{remark}

\begin{remark}\label{rem:quotient_of_action_groupoid}
 In the base point free version of transitive pairs from Remark
 \ref{rem:base_point_free_version_of_transitive_pairs}, one constructs the
 groupoid $\cR(\theta,H)$ by taking the quotient
 \begin{equation*}
  (K\ltimes_{\theta} M)/N(\theta, H).
 \end{equation*}
 Indeed, we have the isomorphisms of Lie groupoids over $M$
 \begin{equation*}
  K\ltimes_{\theta}M:=
  \left(\vcenter{  \xymatrix{
  K\times M \ar@<1.2ex>[d]\ar@<-1ex>[d] \\ M
  }}\right)
  \cong 
  \left(\vcenter{  \xymatrix{
  \frac{(K\times K)}{\Stab{m}} \ar@<1.2ex>[d]\ar@<-1ex>[d] \\ M
  }}\right)
  \cong
  \left(\vcenter{  \xymatrix{
  \left.\frac{(K\times K)}{H}\right/\frac{H}{\Stab{m}} \ar@<1.2ex>[d]\ar@<-1ex>[d] \\ M
  }}\right).
 \end{equation*}
 Then $N(\theta,H)$ corresponds exactly to the normal subgroupoid
 $ \left.\frac{(H\times H)}{H}\right/\frac{H}{\Stab{m}}$ on the right hand
 side.
 
 However, it is much harder to construct the smooth structure on the quotient
 $(K\ltimes_{\theta} M)/N(\theta, H)$ directly. For instance, in the
 finite-dimensional case (or also in the case of $K$ being a Banach-Lie group)
 one can use Godement's criterion for the existence of the quotient
 $(K\ltimes_{\theta} M)/N(\theta, H)$ in the category of smooth manifolds (cf.\
 \cite[Theorem
 2.2.4]{Mackenzie05General-theory-of-Lie-groupoids-and-Lie-algebroids}). It is
 not known to the authors whether Godement's criterion extends beyond Banach
 manifolds, whereas the construction of $\cR(\theta,H)$ as in \ref{defn:
 Liegp:loctgpd} is possible in the full generality of our definition of a
 transitive pair.
\end{remark}

So far we have constructed a locally trivial groupoid $\cR(\theta,H)$ associated to a transitive  pair $(\theta,H)$.
Let us now analyse how the Lie group $K$ (i.e.\ the group acting via $\theta$ on $M$) is related to the Lie group $\Bis (\cR(\theta,H))$.  
To this end, we study a natural Lie group morphism $K \rightarrow \Bis (\cR (\theta , H))$ which is closely related to the action of the transitive pair.

\begin{lemma}\label{lem:canonical_morphism_into_bisections}
 Let $(\theta, H)$ be a transitive pair. Then the action of $K$ on $K/H$ by
 left multiplication gives rise to a group homomorphsim
 $K \rightarrow \Aut (\pi \colon K/H \xrightarrow{\Lambda_m} M)$. With respect
 to the canonical isomorphism
 $ \Aut (\pi \colon K/H \xrightarrow{\Lambda_m} M) \cong \Bis (\cR (\theta,H))$
 of Lie groups from \eqref{eq: ident:AutBis} this gives rise to the group
 homomorphism
 \begin{equation*} 
  a_{\theta,H} \colon K \rightarrow \Bis (\cR (\theta,H)) ,\quad k \mapsto (x \mapsto \langle k \cdot s_i (x) , s_i (x) \rangle , \text{ for } x \in U_i),
 \end{equation*}
 where $s_i = p_m \circ \sigma_i, i\in I$ are the sections from \ref{rem:
 Liegp:loctgpd} b). Moreover, $a_{\theta,H}$ is smooth and thus a morphism of Lie groups.
\end{lemma}

\begin{proof}
 Consider the smooth group action $\lambda_K \colon K \times K/H \rightarrow K/H , (k , gH) \mapsto \lambda_k (gH) \coloneq (kg)H$. 
 By Lemma \ref{lem: eff:act} this group action commutes with the right action $\rho_{\Lambda_m}$ on $K/H$. 
 Hence for each $k \in K$ the map $\lambda_K (k) \colon K/H \rightarrow K/H$ is a bundle automorphism of the $\Lambda_m$-principal bundle.
 Now $a_{\theta, H} (k)$ is the image of the bundle automorphism $\lambda_K (k)$ under the Lie group isomorphism \eqref{eq: ident:AutBis}.
 Since $\lambda_K (kk') = \lambda_K (k)\lambda_K (k')$, we derive that $a_{\theta , H}$ is a group homomorphism.
 
 Let us now prove that $a_{\theta,H}$ is smooth. 
 To this end recall that $K$ and $\frac{K/H \times K/H}{\Lambda_m}$ are modelled on metrisable spaces and $\Bis (\cR(\theta ,H)) \subseteq C^\infty (M, \frac{K/H \times K/H}{\Lambda_m})$.
 Since $M$ is compact, we apply the exponential law Theorem \ref{thm:  MFDMAP} \ref{thm:manifold_structure_on_smooth_mapping_d} to see that $a_{\theta,H}$ will be smooth if the map 
  \begin{displaymath}
   a_{\theta,H}^\vee \colon K \times M \rightarrow \frac{K/H \times K/H}{\Lambda_m} , (k,x) \mapsto  \langle k \cdot s_i (x) , s_i (x) \rangle  , \text{ for } x \in U_i 
  \end{displaymath}
 is smooth. 
 We work locally around $(k,x) \in K \times M$ . 
 Fix $i \in I$ such that $x \in U_i$ and recall that $s_i = p_m \circ \sigma_i$ and $\theta_m (\sigma_i) = \id_{U_i}$ hold.
 Then we have
  \begin{equation}\label{eq: app:beta}
   \beta_\cR (\langle k \cdot s_i (x) , s_i (x)\rangle) = \pi (k\cdot s_i(x)) = \theta_m (k\cdot \sigma_i (x)) = \theta (k , \theta_m (\sigma_i (x)) = \theta (k,x). 
  \end{equation}
 Now we choose $j \in I$ with $\theta (k , x) \in U_j$ and denote by $\kappa_{ji}$ the manifold charts \eqref{eq: Gau:triv} defined in Remark \ref{rem: Liegp:loctgpd}. 
 Then the composition $\kappa_{ji} \circ a_{\theta,H}$ which is defined at least on the pair $(k,x)$ and we compute: 
 \begin{align*}
  \kappa_{ji} \circ a_{\theta, H}^\vee (k,x) &\stackrel{\eqref{eq: Gau:triv}}{=} (\pi (k \cdot s_i (x)) , \pi (s_i (x)), \delta (s_j (\pi (k \cdot s_i (x))), k\cdot s_i (x))\delta^{-1} (s_i \circ \underbrace{\pi \circ  s_i (x)}_{=x} , s_i (x)))\\
					&\stackrel{\eqref{eq: app:beta}}{=} (\theta (k,x) , x, \delta (s_j (\theta (k,x)), k\cdot s_i (x))) = (\theta (k,x) , x, \delta (s_j (\theta (k,x)), \lambda_K (k, s_i (x))))
 \end{align*}
 Note that the above formula did not depend on $(k,x)$, whence it is valid for all $(g,y)$ with $a_{\theta , H}^\vee (g,y) \in \tfrac{\pi^{-1} (U_j) \times \pi^{-1} (U_i)}{\Lambda_m}$. 
 In particular, we see that $\kappa_{ji} \circ a_{\theta, H}^\vee$ is smooth as a composition of the smooth maps $\theta, \delta$ and $\lambda_K$. 
 Since $\frac{K/H \times K/H}{\Lambda_m}$ carries the identification topology with respect to the atlas $(K_{ji})_{i,j \in I}$, we deduce that $a_{\theta,H}^\vee$ is smooth.
 Summing up, this proves that $a_{\theta ,H}$ is smooth and thus a Lie group morphism.
\end{proof}

Before we clarify the relation of $a_{\theta, H}$ and $\theta$ let us return briefly to the problem of (re-)constructing a Lie groupoid from its group of bisections (see Theorem \ref{thm: gpd:quot} and Remark \ref{rem: Re:inconstruction}).
To obtain a construction principle for Lie groupoids, we would like $a_{\theta,H}$ to be an isomorphism of Lie groups.
Then $a_{\theta,H}$ would identify the Lie group with the group of bisections of $\cR(\theta,H)$ and thus transitive pairs would induce (up to isomorphism) unique locally trivial Lie groupoids.
However, in general for an arbitrary transitive pair $(\theta,H)$ the Lie group morphism $a_{\theta , H}$ will neither be injective nor surjective.
We illustrate this with two examples:

\begin{example} \label{ex: Lhom:bad}
  \begin{enumerate}
   \item Let $K$ be a compact finite-dimensional Lie group. Then $K$ acts on itself transitively via left multiplication $\lambda \colon K \times K \rightarrow K$. 
   Take $m=1_K$ and $H = \{1_K\}$ to obtain the principal bundle $\id_K\colon K \xrightarrow{H} K$.
   The associated gauge groupoid is the pair groupoid $K\times K \toto K$ whose bisections are given by $\Diff (K)$.
   Taking this identification, $a_{\theta,H}$ becomes the map $K \rightarrow \Diff (K), k \mapsto \lambda (k,\cdot)$ which will only be surjective in trivial cases. 
   \item We return to Example \ref{ex: tgp:bis} b): Let $B$ be a Banach Lie group, $M$ a compact connected manifold and $m \in M$. 
   Then $(\ev_m \circ \pr_1 , \Diff_m (M) \times B)$ is a transitive pair.
   Set $H \coloneq \Diff_m (M) \times B$ and observe
    \begin{displaymath}
     (\Diff (M) \times B) / H \cong \Diff (M) / \Diff_m (M) \cong M.
    \end{displaymath}
 Moreover, since $H$ is the $m$-stabiliser of the action $\ev_m \circ \pr_1$, we deduce that $\cR(\ev_m \circ \pr_1,H)$ is isomorphic to the pair groupoid $\cP (M)$. 
 With respect to these identifications, the map $a_{\ev_m \circ \pr_1 ,H}$ becomes 
  \begin{displaymath}
   \Diff (M) \times B \rightarrow \Bis (\cP (M)) \cong \Diff (M) , (\varphi , b) \mapsto \varphi
  \end{displaymath}
 which is surjective but can not be injective for non-trivial $B$.
 Note that this example arose from enlarging $\Diff (M) \cong \Bis (\cP(M))$. 
 Moreover, we record that the action of $\Diff (M) \times B$ by left multiplication on $(\Diff (M) \times B) / H$ is not effective and this causes $a_{\ev_m \circ \pr_1 , H}$ to be not injective (cf.\ Lemma \ref{lem: tp:eff} below). 
  \end{enumerate}
\end{example}

As we have already pointed out, transitive pairs are quite flexible and more general than groups of bisections (of locally trivial Lie groupoids)
However, transitive pairs are a source of Lie group morphisms from Lie groups with transitive actions on $M$ into the bisections of suitable locally trivial Lie groupoids over $M$.
In particular, the morphism $a_{\theta,H}$ is closely related to the action $\theta$ as the following Lemma shows.

\begin{lemma} \label{lem:canonical_morphisms_commute}
 For a transitive pair $(\theta, H)$ the Lie group morphism $a_{\theta,H}$ makes the diagram 
 \begin{displaymath}
      \begin{xy}
      \xymatrix{
	K \ar[rr]^{a_{\theta, H}} \ar[rrd]^{\theta^\wedge} & & \Bis (\cR (\theta , H)) \ar[d]^{(\beta_\cR)_*} \\
	& & \Diff (M)
  }
\end{xy}
    \end{displaymath}
 commutative. If $a_{\theta ,H}$ is an isomorphism of Lie groups, then the map $\theta^\wedge$ is a submersion.
\end{lemma}

\begin{proof} 
Observe first that for $x \in M$ we have (after choosing an appropriate section $s_i$) the formula \eqref{eq: app:beta} 
  \begin{displaymath}
  \beta_\cR \circ a_{\theta , H} (k)(x) = \beta_\cR (\langle k \cdot s_i (x) , s_i (x)\rangle) = \theta (k,x) = \theta^\wedge (k)(x). 
  \end{displaymath}
 Hence $(\beta_\cR)_* \circ a_{\theta , H} = \theta^\wedge$ and the diagram commutes.
 This also entails $a_{\theta , H} (H) \subseteq \Loop{m} (\cR (\theta,H))$. 
 
 If $a_{\theta,H}$ is a Lie group isomorphism, $\theta^\wedge$ is a submersion if $(\beta_\cR)_* \colon \Bis (\cR (\theta,H)) \rightarrow \Diff (M)$ is a submersion.
 However, since $\cR(\theta,H)$ is locally trivial, the map $(\beta_\cR)_*$ is a submersion by \cite[Example 3.16]{SchmedingWockel14}. 
\end{proof}

\begin{remark}
 That $\theta^\wedge$ must be a submersion if $a_{\theta,H}$ is an isomorphism of Lie groups can be understood as the statement that the Lie group $K$ needs to be large enough to be eligible to be the bisection group.
 In particular, if $M$ is not a zero-dimensional manifold, this condition rules out every transitive pair which arises by a group action of a finite-dimensional Lie group.
\end{remark}

Building on the observation in Example \ref{ex: Lhom:bad} b) we will now develop a simple criterion which ensures that for a given transitive pair $(\theta, H)$ the morphism $a_{\theta , H}$ is injective.

\begin{lemma}\label{lem: tp:eff}
 Let $(\theta, H)$ be a transitive pair. 
 The Lie group morphism  
    \begin{displaymath}
     a_{\theta, H} \colon K \rightarrow  \Bis (\cR (\theta,H)), g \mapsto (kH \mapsto (g\cdot k)H) 
    \end{displaymath}
 is injective if and only if the action $\lambda_H\colon H \times K/H \rightarrow K/H, (h,gH) \mapsto (hg)H$ is effective. 
 \end{lemma}

 \begin{proof}
 The isomorphism $ \Aut (\pi \colon K/H \xrightarrow{\Lambda_m} M) \cong \Bis (\cR (\theta,H))$ allows us to rewrite $a_{\theta,H}$ as the left multiplication $\lambda_K \colon K \rightarrow  \Aut (\pi \colon K/H \xrightarrow{\Lambda_m} M)$, where $\lambda_K (k) \colon K/H \rightarrow K/H , gH \mapsto (kg)H$. 
  Assume first that $a_{\theta,H}$ (and thus also $\lambda_K$) is injective. 
  Now consider $k \in H$ such that $(hg)H = gH$ for all $g \in G$. 
  This implies $\lambda_K (h) (gH) = (hg)H = gH = \id_{K/H} (gH)$, whence $h = 1_K$ as the group homomorphism $\lambda_K$ is injective.
  
  Conversely assume that  the action $H \times K/H \rightarrow K/H, (h,gH) \mapsto (hg)H$ is effective and consider $g \in \ker \lambda_K$, i.e.\ $\lambda_K (g) = \id_{G/H}$.
  Then $(g\cdot k)H = kH$ holds for all $kH \in G/H$.
  As this entails $gH = 1_G H$, we deduce $g \in H$. 
  Now the left action of $H$ on $G/H$ by multiplication is effective, forcing $g \in H$ to be the identity $1_K$.
 \end{proof}

 \begin{proposition}\label{prop: kernel}
  Consider a transitive pair $(\theta,H)$ and denote by $\lambda_H$ the left action on the quotient $K/H$ as in Lemma \ref{lem: tp:eff}.
  Then the kernel of $a_{\theta,H}$ is given by 
    \begin{displaymath}
    \Ker a_{\theta,H} = \{h \in H \mid \lambda_H (h,\cdot) \equiv \id_{K/H}\}
    \end{displaymath}
  and coincides with the kernel of the transitive pair $(\theta,H)$.
  In particular every transitive pair admits a unique kernel.
 \end{proposition}

 \begin{proof}
 Recall from the proof of Lemma
 \ref{lem: tp:eff} that $\Ker a_{\theta,H}$ is contained in $H$ and consists of
 all elements of $H$ which act trivially by left multiplication on $K/H$.
 Thus we obtain the first description of $\Ker a_{\theta,H}$.
 
 As a kernel of a Lie group morphism, $\Ker a_{\theta, H}$ is a closed and normal subgroup of $K$.
 Let us now prove that every subgroup $G$ of $H$ which is normal in $K$ is contained in $\Ker a_{\theta,H}$.
 Then for $k\in K$ and
 $g \in G \subseteq H$ we derive from $G$ being normal in $K$ that $g k = k g'$
 for $g' \in H$, i.e.\ $gkH=kH$ for all $k \in K$. Thus elements in $G$ act
 trivially on $K/H$, whence $G \subseteq \Ker a_{\theta,H}$.
 We conclude that $\ker a_{\theta,H}$ is the kernel of the transitive pair $(\theta, H)$.
 \end{proof}

 \begin{definition}
  A transitive pair $(\theta , H)$ is called \emph{effective} if $H$ acts effectively on the quotient $K/H$ by left multiplication, i.e.\ the kernel of the transitive pair is trivial. 
 \end{definition}
 
 The characterisation of the kernel of a transitive pair in Proposition \ref{prop: kernel} can be used to compute it. 
 For the examples considered in this section we obtain:
 
 \begin{example}
  The transitive pairs in Example \ref{ex: tgp:bis} a) and \ref{ex: Lhom:bad} a) are effective, whence the kernel is trivial.
  For the transitive pair $(\ev_m \circ \pr_1 , \Diff(M)_m \times B)$ from Example \ref{ex: tgp:bis} b) the kernel is $\{\id_M\} \times B$ (by Example \ref{ex: Lhom:bad} b)).
  In Theorem \ref{thm:integrating_extensions_to_transitive_pairs} below the kernel of a class of transitive pairs arising from extensions of diffeomorphism groups is computed.
 \end{example}

 Although the criterion for
 the injectivity of $a_{\theta,H}$ gives rise to a very natural condition on
 the transitive pair, the question of surjectivity is much more subtle.
 
\begin{remark}\label{rem:klein_geometries_vs_transitive_pairs}
 We now describe the relation between Klein geometries \cite[Chapter
 3]{Sharpe97Differential-geometry} and transitive pairs. First note that our
 setting is infinite-dimensional, and the notion of a transitive pair takes the
 additional analytical issues caused by this into account.
 
 Recall that a \emph{Klein geometry} is a pair $(K,H)$, where $K$ is a
 finite-dimensional Lie group (called the \emph{principal group}) and $H$ is a
 closed subgroup such that the manifold $K/H$ is connected. The \emph{kernel}
 of a Klein geometry is the largest subgroup $L$ of $H$ which is normal in $K$.
 A Klein geometry is called \emph{effective} if $L$ is trivial. Klein geometries are constructed to model geometry via the principal $H$-bundle
 $K \rightarrow K/H$.
 
 Note that the principal group of a Klein geometry is finite-dimensional. Hence
 the quotient $K/H$ inherits a canonical manifold structure turning
 $K \rightarrow K/H$ into a submersion. 
 In our infinite-dimensional setting the quotient does not automatically inherit a manifold structure,
 whence a transitive pair has to guarantee this behaviour 
 via extra assumptions (cf.\ Proposition \ref{prop: const:pbund}).
 
 In studying a transitive pair $(\theta,H)$, we are interested in the principal
 $\Stab{m}/H$-bundle $K/H \rightarrow K/\Stab{m} = M$. Thus a transitive pair
 encodes more information than a Klein geometry, as the principal $H$-bundle
 $K \rightarrow K/H$ is obtained as additional information.
 To some extend, one can interpret a transitive pair as a ``Klein geometry for principal bundles''.
 In particular, the notion of a transitive pair also covers the concept of an
 (infinite-dimensional) Klein geometry in the case that $K_{m}=H$ and the
 quotient $K/H$ is connected.
 
 Finally, there is a close connection between effective
 transitive pairs and effective Klein geometries. Namely, the kernel of the transitive pair (i.e.\ the kernel of the
 Lie group morphism $a_{\theta,H} \colon K \rightarrow \cR (\theta,H)$) is by definition the largest closed subgroup of $H$ which is normal in $K$.
 Hence the kernel of the transitive pair $(\theta,H)$
 corresponds to the kernel of a Klein geometry. Let us stress again that contrary to the finite dimensional case, the
 kernel of a transitive pair will only be a closed subgroup and not automatically a closed Lie subgroup. 
 Summing up, if the pair $(\theta,H)$ is
 effective and $H = K_m$ and $K/H$ is connected, then the transitive pair
 corresponds to an (infinite-dimensional) effective Klein geometry.
\end{remark}

\begin{tabsection}
 Let us now return to Example \ref{ex: tgp:bis} a) and consider
 $\cR(\theta ,H)$ and $a_{\theta, H}$ for the action of a bisection group on
 $M$. We will see that the constructions given in this section are in a certain
 sense inverse to computing the bisections of a locally trivial groupoid.
\end{tabsection}

\begin{example}\label{ex:
 BisR:inv} Let $\cG = (G \toto M)$ be a locally trivial Banach-Lie groupoid
 such that there is a bisection through every $g \in G$. Consider the group
 action
 $\beta \circ \ev \colon \Bis (\cG) \times M \rightarrow M , (\tau,m) \mapsto \beta \circ \tau (m)$.
 We have seen in Example \ref{ex: tgp:bis} a) that
 $(\beta \circ \ev , \Bisf{m} (\cG))$ is a transitive pair. Moreover, since
 every locally trivial Lie groupoid is transitive, Lemma \ref{lem: eff:act2}
 shows that $(\beta \circ \ev , \Bisf{m} (\cG))$ is an effective transitive
 pair. For this special effective transitive pair we note the following
 consequences (a detailed verification for the claims made in this example can
 be found in Lemma \ref{lem: relchi:a} below):\medskip
 
 \textbf{The groupoid $\cG$ is isomorphic to
 $\cR(\beta \circ \ev , \Bisf{m} (\cG))$.} We have already seen in Example
 \ref{ex: Bis:pbun} that the vertex bundles associated to the locally trivial
 Lie groupoids $\cG$ and $\cR(\beta \circ \ev , \Bisf{m} (\cG))$ are
 isomorphic. Moreover, from Proposition \ref{prop: quot:fibre} we recall that
 $\Bis (\cG) / \Bisf{m} (\cG) \cong \alpha^{-1} (m)$ and
 $\Loop{m} (\cG) / \Bisf{m} (\cG) \cong \Vtx{m}(\cG)$. Hence the gauge groupoid
 $\cR(\beta \circ \ev , \Bisf{m} (\cG))$ is canonically isomorphic to the gauge
 groupoid of a vertex bundle of $\cG$. Consider the canonical map over $M$ from
 the gauge groupoid $\cR (\beta \circ \ev , \Bisf{m} (\cG))$ to the locally
 trivial Lie groupoid $\cG$
 \begin{equation}\label{eq: chi} 
  \chi_\cG \colon \frac{\Bis (\cG) / \Bisf{m}(\cG) \times \Bis (\cG) / \Bisf{m}(\cG)}{\Loop{m} (\cG)/\Bisf{m}(\cG)} \rightarrow G , \langle \sigma \Bisf{m}(\cG) , \tau \Bisf{m}(\cG)\rangle \mapsto \sigma (m)\cdot(\tau(m))^{-1}.
 \end{equation}
 Since $\cG$ admits bisections through each arrow and $\cG$ is locally trivial,
 Proposition \ref{prop: quot:fibre} shows that the image of $\chi_\cG$
 coincides with $G$. One then proves that $\chi_\cG$ is an isomorphism of Lie
 groupoids over $M$. We can thus recover the locally trivial Lie groupoid $\cG$
 (up to isomorphism depending on $m$) as the groupoid
 $\cR (\beta \circ \ev, \Bisf{m}(\cG))$. \medskip
 
 \textbf{The map
 $a_{\beta \circ \ev , \Bisf{m} (\cG)} \colon \Bis (\cG) \rightarrow \Bis (\cR(\beta \circ \ev , \Bisf{m} (\cG)))$
 is a Lie group isomorphism.} For this special transitive pair, an inverse of
 $a_{\beta \circ \ev , \Bisf{m} (\cG)}$ is given by
 $\Bis (\chi_\cG) \colon \Bis (\cR(\theta,H)) \rightarrow \Bis (\cG) , \sigma \mapsto \chi_\cG \circ \sigma$.
\end{example}

\begin{lemma}\label{lem: relchi:a}
 Let $\cG = (G \toto M)$ be a locally trivial Banach-Lie groupoid such that the action $\beta \circ \ev$ of the bisections on $M$ is transitive, e.g.\ $\cG$ is source-connected. 
 Then
 \begin{enumerate} 
  \item the Lie groupoid morphism $\chi_\cG$ from \eqref{eq: chi} induces an isomorphism onto the open subgroupoid $\ev (\Bis (\cG) \times M)$ of $\cG$.
  Hence $\ev$ is surjective if and only if $\chi_\cG \colon \cR (\beta \circ \ev , \Bisf{m} (\cG)) \rightarrow \cG$ is an isomorphism.  
  \item the Lie group morphism $\Bis (\chi_\cG) \colon \Bis (\cR ( \beta \circ \ev, \Bisf{m} (\cG))) \rightarrow \Bis (\cG)$ is an isomorphism with inverse $a_{\beta \circ \ev, \Bisf{m} (\cG)}$. 
 \end{enumerate}
\end{lemma}

\begin{proof}\begin{enumerate}
              \item Clearly $\chi_\cG$ is injective and after restricting it to its image, it becomes a bijection. 
              Now consider the commutative diagram 
		\begin{equation}\label{diag: chi} \begin{aligned}
		 \begin{xy}
			  \xymatrix{
 	\Bis (\cG) / \Bisf{m} (\cG) \times \Bis (\cG) / \Bisf{m} (\cG) \ar[d]^{q} \ar[rrr]^-{\widetilde{\ev_m} \times \widetilde{\ev_m}}& & &\alpha^{-1} (m) \times \alpha^{-1} (m) \ar[d]^{(\xi,\eta) \mapsto \xi \cdot \eta^{-1}} \\
 	 \frac{\Bis (\cG) / \Bisf{m} (\cG) \times \Bis (\cG) / \Bisf{m} (\cG)}{\Loop{m} (\cG) / \Bisf{m} (\cG)} \ar[rrr]^-{\chi_\cG}& & & G
   }
  \end{xy}        \end{aligned} .
		\end{equation}
   Here $q$ is the canonical quotient map which is a submersion. 
   The map $\widetilde{\ev_m} \times \widetilde{\ev_m}$ is the map induced on the quotient via $\widetilde{\ev_m} \times \widetilde{\ev_m} \circ (p_m \times p_m) = \ev_m \times \ev_m$ where $p_m \colon \Bis (\cG) \rightarrow \Bis (\cG) / \Bisf{m} (\cG)$ is the canonical quotient map.
   Now $\ev_m \times \ev_m$ is a submersion by \cite[Lemma 1.6]{hg2015} and Corollary \ref{cor: evx:subm}). 
   Since $\cG$ is a Banach-Lie groupoid and $q_m$ is a submersion, we deduce with \cite[Lemma 1.10]{hg2015} that $\widetilde{\ev_m} \times \widetilde{\ev_m}$ is a submersion.
   Finally, \cite[Proposition 1.3.3]{Mackenzie05General-theory-of-Lie-groupoids-and-Lie-algebroids} shows that the division map $\alpha^{-1} (m) \times \alpha^{-1} (m) \rightarrow G, (\xi, \eta) \mapsto \xi \cdot \eta^{-1}$ is a surjective submersion as $\cG$ is a locally trivial Lie groupoid.
   
   Note that the division map restricts to a map $\im (\ev_m \times \ev_m) \rightarrow \ev (\Bis (\cG) \times M)$ because for $\sigma , \tau \in \Bis (\cG) $
   we have $\sigma (m) \cdot (\tau (m))^{-1} = (\sigma \star \tau^{-1}) (\beta (\tau (m))) \in \ev (\Bis (\cG) \times M)$.
   In particular $\im (\chi_\cG) \subseteq \ev(\Bis (\cG) \times M)$
   We claim that the image of $\chi_\cG$ coincides with $\ev (\Bis (\cG) \times M)$.
   If the claim is true, then the proof can be finished as follows: 
   The set $\ev (\Bis (\cG) \times M)$ is an open and wide subgroupoid of $\cG$ by Theorem \ref{thm: bis:points}.
   Hence \eqref{diag: chi} proves that $\chi_\cG$ induces a Lie groupoid isomorphism from $\cR (\beta \circ \ev, \Bisf{m} (\cG))$ onto the subgroupoid $\ev (\Bis (\cG) \times M)$ of $\cG$.
   \medskip
   
   \textbf{Proof of the claim:} Choose a bisection $\gamma \in \Bis (\cG)$ and $x \in M$ and let us show that $\gamma (x) \in \im \chi_\cG$. 
   The action $\beta \circ \ev \colon \Bis (\cG) \times M \rightarrow M , (\sigma , y ) \mapsto \beta (\sigma (y))$ is transitive.
   Hence there are $\sigma, \tau \in \Bis (\cG)$ with $\beta (\tau (m)) = x$ and $\beta (\sigma (m)) = \beta (\gamma (x))$. 
   We compute 
    \begin{displaymath} 
     (\gamma (x))^{-1} \cdot \sigma (m) \cdot (\tau(m))^{-1} = (\gamma (x))^{-1} \cdot (\sigma \star \tau^{-1}) (x) = \underbrace{\gamma^{-1} \star \sigma \star \tau^{-1}}_{\equalscolon l_x \in \Bis (\cG)} (x).
    \end{displaymath}
   By construction $\gamma \star l_x (x) = \sigma (m) \cdot (\tau (m))^{-1}$ and thus $\gamma (x) = \sigma (m) \cdot ((l_x \star \tau) (m))^{-1} \in \im \chi_\cG$.
   As $\gamma$ and $x$ were arbitrary this establishes $\ev (\Bis (\cG) \times M) = \im \chi_\cG$.          
   \item Set $\Lambda_m \coloneq \Loop{m} (\cG) / \Bisf{m} (\cG)$.
   To compute on the $\Lambda_m$-principal bundle $\beta \circ \ev_m \colon \Bis (\cG) / \Bisf{m}(\cG) \rightarrow M$ we choose a section atlas $(s_i,U_i)_{i \in I}$ and obtain   
  \begin{align*}
   \Bis (\chi_\cG) \circ a_{\beta \circ \ev, \Bisf{m} (\cG)} (\sigma) &= \chi_\cG \circ a_{\beta \circ \ev, \Bisf{m} (\cG)} (\sigma) = \chi_\cG \circ (x \mapsto \langle \sigma \star s_i (x), s_i (x)\rangle , x \in U_i)\\
					   &= (x \mapsto (\sigma \star s_i (x)) (m) \cdot (s_i (x) (m))^{-1},x \in U_i)\\
					   &= (x \mapsto \sigma (\underbrace{\beta (s_i (x)(m))}_{=x}) \cdot (s_i (x)(m)) \cdot (s_i (x)(m))^{-1} ,  x \in U_i ) = \sigma.
  \end{align*}
 Hence $\Bis (\chi_\cG) \circ a_{\beta \circ \ev, \Bisf{m} (\cG)} = \id_{\Bis (\cG)}$.
 Let us show  $a_{\beta \circ \ev, \Bisf{m} (\cG)} \circ \Bis (\chi_\cG) = \id_{\Bis (\cR (\theta,H))}$. Denote by $\psi \colon \Aut (\beta \circ \ev_m \colon \Bis (\cG) / \Bisf{m}(\cG) \rightarrow M) \rightarrow \Bis (\cR (\theta ,H))$ the isomorphism from \eqref{eq: ident:AutBis}.
 We fix a $\Lambda_m$-principal bundle automorphism $f$ and compute the image of $\psi (f)$ under $a_{\beta \circ \ev, \Bisf{m} (\cG)} \circ \Bis (\chi_\cG)$ as  
  \begin{align}
   a_{\beta \circ \ev, \Bisf{m} (\cG)} \circ \Bis (\chi_\cG) (\psi(f)) &= a_{\beta \circ \ev, \Bisf{m} (\cG)} (\chi_\cG (x \mapsto \langle f(s_i (x)), s_i (x) \rangle , x \in U_i ) \notag \\ 
					    &= \psi \circ \psi^{-1} \circ a_{\beta \circ \ev, \Bisf{m} (\cG)} (x \mapsto (f(s_i (x))(m)) \cdot (s_i(x)(m))^{-1}, x \in U_i) \notag \\
				            &= \psi (\tau \Bisf{m}(\cG) \mapsto \left(((f(s_i (\bullet))(m)) \cdot (s_i(\bullet)(m))^{-1}) \star \tau \right) \Bisf{m}(\cG)) .\label{eq: lastline}
  \end{align}
 In passing from the second to the third line we have used that $\psi^{-1}$ takes $a_{\theta, H}$ to the left action $\lambda_{\Bis (\cG)}$ discussed in Lemma \ref{lem: eff:act}.
 We will now prove that the argument of $\psi$ in \eqref{eq: lastline} coincides with the principal bundle automorphism $f$.
 If this is true, then $\Bis (\chi_\cG)$ is an isomorphism with inverse $a_{\beta \circ \ev, \Bisf{m} (\cG)}$.
 
 The equivariant map $\widetilde{\ev_m} \colon \Bis (\cG) / \Bisf{m} (\cG) \rightarrow \alpha^{-1} (m)$ from Lemma \ref{lem: ev:equiv} identifies $\Bis (\cG) / \Bisf{m}(\cG)$ with the open subset $\ev_m (\Bis(\cG))$ of $\alpha^{-1} (m)$.
 Recall from  Proposition \ref{prop: quot:fibre} that $\Lambda_m$ is isomorphic to an open subgroup $e_m (\Lambda_m) \opn \Vtx{m}(\cG)$ and $e_m (\Lambda_m) = \{\tau (m) \mid \tau \in \Loop{m} (\cG)\}$.
 Since $\beta \circ \ev$ is transitive, the $\Vtx{m}(\cG)$-principal bundle $\beta|_{\alpha^{-1} (m)} \colon \alpha^{-1} (m) \rightarrow M$ restricts on $\ev_m (\Bis (\cG))$ to a $e_m (\Lambda_m)$-principal bundle. 
 Identifying the groups, we obtain a $\Lambda_m$-principal bundle $\beta|_{\ev_m (\Bis (\cG))} \colon \ev_m (\Bis (\cG)) \rightarrow M$ and $\widetilde{\ev_m}$ induces an isomorphism of $\Lambda_m$-principal bundles (cf.\ Example \ref{ex: Bis:pbun}).
 
 The principal bundle isomorphism $\widetilde{\ev_m}$ allows us to associate to the automorphism $f$ a $e_m(\Lambda_m)-$bundle automorphism $\tilde{f} \colon \ev_m (\Bis (\cG)) \rightarrow \ev_m (\Bis (\cG))$ via $\tilde{f} \circ \widetilde{\ev_m} = \widetilde{\ev_m} \circ f $.
 Evaluating the argument of $\psi$ from \eqref{eq: lastline} at $\tau \Bisf{m}(\cG)$, our preparations allow us to compute as follows. 
  \begin{align*}
   \widetilde{\ev_m}^{-1} (((f(s_i (\bullet))(m)) \cdot (s_i(\bullet)(m))^{-1}) \star \tau (m)) &=  \widetilde{\ev_m}^{-1} ((f(s_i (\beta (\tau (m))(m)) \cdot(s_i(\beta (\tau (m))(m))^{-1} \cdot\tau (m)) \\
											      &=  \widetilde{\ev_m}^{-1} (\tilde{f}(s_i (\beta (\tau (m))(m))) \cdot \underbrace{(s_i(\beta (\tau (m))(m))^{-1} \cdot\tau (m)}_{\in e_m (\Lambda_m) \opn \Vtx{m}(\cG)}))\\
											      &=  \widetilde{\ev_m}^{-1} (\tilde{f}(\underbrace{s_i (\beta (\tau (m))(m)) \cdot (s_i(\beta (\tau (m))(m))^{-1}}_{=1_{\beta (\tau (m))}} \cdot\tau (m)))\\
											      &=  \widetilde{\ev_m}^{-1} (\tilde{f}(\tau(m)) = f(\tau \Bisf{m}(\cG)) 
  \end{align*} 
             \end{enumerate}
 \end{proof}
  
\begin{problem}\label{prob1}
 \begin{enumerate}
  \item It would be interesting to develop a notion/theory of infinitesimal
        transitive pairs, i.e., an infinitesimal action $\fk\to \cV(M)$
        together with an ideal $\fh$ of $\fk_{m}$. In particular, a derivation
        of the Lie algebra morphism $\Lf(a_{\theta,H})$ directly from these
        data in case of an effective transitive pair would be a valuable tool.
  \item It would also be interesting to develop the theory of this section for
        not necessarily locally trivial Lie groupoids (analogously perhaps to
        the notion of Lie-Rinehart algebras in the infinitesimal setting, see
        \cite{Huebschmann90Poisson-cohomology-and-quantization}).
 \end{enumerate}
\end{problem}

\section{Integrating extensions of Lie groups to transitive pairs} %
\label{sec:application_to_integration_of_extensions}

\begin{tabsection}
 We now study the application of the previously developed theory to the
 integration theory of extensions of Lie algebroids and Lie groupoids.
 Throughout this section, $M$ denotes a compact and \emph{1-connected} manifold
 for which we choose some fixed base-point $m\in M$. We will heavily use the
 integration theory of abelian extensions of infinite-dimensional Lie groups,
 for which we refer to
 \cite{Neeb04Abelian-extensions-of-infinite-dimensional-Lie-groups} (see also
 the Appendix in
 \cite{Neeb04Abelian-extensions-of-infinite-dimensional-Lie-groups} for some of
 the terminology that we are using).
\end{tabsection}

\begin{tabsection}
 Throughout the rest of this section, $K$ will be a
 connected and $C^{\infty}$-regular Lie subgroup of $\Diff(M)$ such that
 $(\theta,K_{m})$ is a 2-fold transitive pair, where
 $\theta \from K\times M\to M$, is the natural action of $K\leq\Diff(M)$ on $M$
 and $K_{m}$ is the stabiliser of the base-point $m$ in $K$. In particular,
 $\ev_{m}\from K\to M$ then factors through a diffeomorphism $K/K_{m}\to M$ and
 $K_{m}$ acts transitively on $M\setminus \{m\}$.
 
 The Lie group $\Diff(M)$ acts naturally and smoothly on
 $C^{\infty}(M):=C^{\infty}(M,\R)$ via $\varphi.f:=f \circ \varphi^{-1}$. This
 turns $\fa:=C^{\infty}(M)$ into a $K$-module containing the constant functions
 $\fa_{0}=\fa^{K}\cong\R$. If we set $\fk=\Lf(K)$ and $\fk_{m} = \Lf (K_m)$, then $\fa$ is also a
 $\fk$-module for the derived action
 \begin{equation*}
  (X.f)(n):=df_{n}(X(n)).
 \end{equation*}
 Moreover, $\fa_{m}\coloneq C^{\infty}_{m}(M)$ is a $\fk_{m}$-submodule for
 \begin{equation*}
  C^{\infty}_{m}(M) :=\{f\in C^{\infty}(M)\mid f(m)=0\}.
 \end{equation*}
 Clearly, $\fa_{0}$ is a module complement for this $\fk_{m}$-submodule, so
 that we have $\fa\cong \fa_{0}\oplus \fa_{m}$ as $\fk_{m}$-modules.
 
 The last piece of information that we choose is a closed 2-form
 $\omega\in \Omega^{2}(M):= \Omega^{2}(M,\R)$. This gives rise to the abelian cocycle
 \begin{equation*}
  \ol{\omega}\from \fk\times \fk\to \fa,\quad(X,Y)\mapsto (n\mapsto \omega_{n}(X(n),Y(n))),
 \end{equation*}
 i.e.,
 \begin{equation*}
  [(f,X),(g,Y)]:=(X.g-Y.f+\ol{\omega}(X,Y),[X,Y])
 \end{equation*}
 defines on $\fa\oplus \fk$ the structure of a Lie algebra. We denote this Lie
 algebra by $\fa\oplus_{\ol{\omega}}\fk$, and the canonical maps give rise to
 an abelian extension
 \begin{equation*}
  \fa\to \fa\oplus_{\ol{\omega}}\fk\to \fk
 \end{equation*}
 of topological Lie algebras. Moreover, as elements of $\fk_m$ vanish when they are evaluated in $m$, we have
 $\ol{\omega}(\fk_{m}\times\fk_{m})\se\fa_{m}$. Hence there is a subalgebra
 \begin{equation*}
  \fa_{m}\oplus_{\ol{\omega}_{m}}\fk_{m}\leq \fa\oplus_{\ol{\omega}}\fk,
 \end{equation*}
 where $\ol{\omega}_{m}\from \fk_{m}\times\fk_{m}\to\fa_{m}$ denotes the
 restriction (and corestriction) of $\ol{\omega}$ to $\fk_{m}\times\fk_{m}$
 (and $\fa_{m}$).
 
 Before we go on we give the two examples of the above situations that we have
 in mind.
\end{tabsection}

\begin{example}
 \begin{enumerate}
  \item The easiest example is that where $K=\Diff(M)_{0}$ and
        $\fa=C^{\infty}(M)$. Then we have seen in Example
        \ref{exemp:Diff_m_is_Lie_subgroup} that $\Diff_{m}(M)$ is a regular Lie
        subgroup of $\Diff(M)$. Moreover, $\Diff(M)$ acts smoothly transitively
        on $M$ by Corollary \ref{cor: diff0:trans} and Corollary \ref{cor:
        evx:subm}. Finally, $\Diff(M)$ acts 2-fold transitively on $M$
        \cite{MichorVizman94n-transitivity-of-certain-diffeomorphism-groups}.
        In this case $\omega\in \Omega^{2}(M)$ can be an arbitrary closed
        2-form on $M$ (cf.\ \cite[Section
        9]{Neeb04Abelian-extensions-of-infinite-dimensional-Lie-groups}).
  \item For the next example, let $\omega\in \Omega^{2}(M)$ be symplectic and
        $\dim(M)=2d$. Then
        \begin{equation*}
         \Symp(M,\omega)\coloneq\{\varphi\in \Diff(M)\mid \varphi^{*}\omega-\omega=0 \}
        \end{equation*}
        is a closed Lie subgroup of $\Diff(M)$ \cite[Theorem 43.12]{conv1997}
        with Lie algebra
        \begin{equation*}
         \symp(M,\omega)\coloneq\{X\in \cV(M)\mid \cL_{X}\omega=d(i_{X}\omega)=0\}
        \end{equation*}
        the symplectic vector fields on $M$. This Lie algebra has the
        Hamiltonian vector fields
        \begin{equation*}
         \ham(M,\omega)\coloneq\{X\in \symp(M)\mid i_{X}\omega\text{ is exact}\}
        \end{equation*}
        as closed subalgebra. Moreover, the chart constructed in the proof of
        \cite[Theorem 43.12]{conv1997} maps
        \begin{equation*}
         \fk_{m}:=\symp_{m}(M,\omega)\coloneq\{X\in\symp(M,\omega)\mid X(m)=0\}
        \end{equation*}
        to
        \begin{equation*}
         \Symp_{m}(M,\omega)\coloneq\{\varphi\in \Symp(M,\omega)\mid \varphi(m)=m\}.
        \end{equation*}
        Thus $K\coloneq\Symp(M,\omega)_{0}$ has $K_{m}=\Symp_{m}(M,\omega)_{0}$
        as Lie subgroup with Lie algebra $\fk_{m}$. That $(\theta,K_{m})$ is
        indeed a transitive pair will follow from the following two
        propositions. Finally, $K$ acts 2-fold transitively by
        \cite{MichorVizman94n-transitivity-of-certain-diffeomorphism-groups}.
 \end{enumerate}
\end{example}

\begin{proposition}
 Let $(M,\omega)$ be a compact symplectic manifold and $x\in M$. Then the
 evaluation map
 \begin{equation*}
  \ev_{x}\from  \Symp(M,\omega)_{0}\to M,\quad \varphi\mapsto \varphi(x)
 \end{equation*}
 is a submersion. If, moreover, $M$ is connected, then $\ev_{x}$ is also
 surjective.
\end{proposition}

\begin{proof}
 We show that $\ev_{x}$ is a submersion by showing that
 $T_{\id}\ev_{x}\from T_{\id}\Symp(M,\omega)\to T_{x}M$ is surjective. This
 suffices by \cite[Theorem A]{hg2015}. To this end, note that $T_{\id}\ev_{x}$
 is given with respect to the identification
 $\symp(M,\omega)\cong T_{\id}\Symp(M,\omega)$ by
 \begin{equation*}
  \ev_{x}\from \symp(M,\omega)\to T_{x}M,\quad X\mapsto X(x)
 \end{equation*}
 (cf.\ \cite[Corollary 42.18]{conv1997} or \cite[Theorem
 7.9]{SchmedingWockel14}). Thus the claim follows from observing that for each
 $v\in T_{x}M$ there exists a function $H\in C^{\infty}(M)$ with
 $dH_{x}=\omega_{x}(v,\argument)$, and thus a Hamiltonian vector field $X$ (which
 is then in particular symplectic) with $X(x)=v$.
 
 Since the point $x\in M$ in the above argument was arbitrary, this shows in
 particular that each orbit of the natural action of $\Symp(M,\omega)$ on $M$ is
 open and thus consist of unions of path components. In particular, $\ev_{x}$
 is surjective if $M$ is connected.
\end{proof}

\begin{proposition}
 Let $(M,\omega)$ be a compact symplectic manifold and
 $k\in \N_{0}\cup\{\infty\}$. Then $\Symp(M,\omega)$ and
 \begin{equation*}
  \Symp_{m}(M,\omega)\coloneq \{\varphi\in \Symp(M,\omega)\mid \varphi(m)=m\}
 \end{equation*}
 are $C^{k}$-regular Lie subgroups of $\Diff(M)$.
\end{proposition}

\begin{proof}
 By \cite[Theorem 43.12]{conv1997} the Lie group $\Symp(M,\omega)$ is a closed
 Lie subgroup of $\Diff (M)$. Recall from \cite[Theorem 5.5]{SchmedingWockel14}
 that $\Diff (M)$ is $C^k$-regular. Arguing as in the proof of \cite[Theorem
 43.12]{conv1997} (where $C^\infty$-(semi)regularity for $\Symp(M,\omega)$ was
 established), one proves that $\Symp(M,\omega)$ is $C^k$-semiregular. Hence
 Lemma \ref{lem: semisub:reg} implies that $\Symp(M,\omega)$ is
 $C^{k}$-regular.
 
 Fix some $\eta\in C^{k}([0,1],\symp_m(M,\omega))$ and let $\gamma_{\eta}$ be a
 solution in $\Symp(M,\omega)$ of the corresponding initial value problem from
 Definition \ref{def:regularity}. Note that $\gamma_\eta$ is also a solution of
 the corresponding initial value problem in $\Diff(M)$. By Proposition
 \ref{prop: stabsubgp:reg}, $\eta_{\gamma}$ takes its image in
 $\Symp_{m}(M,\omega)$, whence $\Symp_{m}(M,\omega)$ is $C^{k}$-semiregular.
 Again by Lemma \ref{lem: semisub:reg}, $\Symp_{m}(M,\omega)$ is also
 $C^{k}$-regular.
\end{proof}

For later reference we also record the following fact (see also \cite[Corollary 3.5]{Janssens15Universal-Central-Extension-of-the-Lie-Algebra-of-Hamiltonian-Vector-Fields}).

\begin{proposition}\label{prop:restriction_of_canonical_cocycle_to_ham_is_coboundary}
 If $(M,\omega)$ is a compact and connected symplectic manifold, then the
 restriction of the abelian cocycle
 \begin{equation*}
  \ol{\omega}\from \cV(M)\times \cV(M)\to C^{\infty}(M),\quad (X,Y)\mapsto \omega(X,Y)
 \end{equation*}
 to $\ham(M,\omega)$ is a coboundary. This also applies to the
 restriction of $\ol{\omega}$ to $\symp(M,\omega)$ if $M$ is 1-connected.
\end{proposition}

\begin{proof}
 We consider the continuous linear surjection
 $\xi\from C^{\infty}(M)\to \ham(M,\omega)$, $f\mapsto X_{f}$, where $X_{f}$ is
 the unique vector field such that $i_{X_{f}}\omega=df$. Since
 $\ker(\xi)\cong\R$ is finite-dimensinoal and thus complemented, $\xi$ has a
 continuous linear section $\sigma\from \ham(M,\omega)\to C^{\infty}(M)$. We
 claim that
 \begin{equation*}
  b\from \ham(M,\omega)\to C^{\infty}(M),\quad X\mapsto  \int_{M} \sigma(X)\, \omega^{d}-\sigma(X)
 \end{equation*}
 satisfies $d_{\op{CE}}b=\omega$, where $d=\frac{\dim(M)}{2}$ and we consider
 $\R$ as constant functions on $M$. To this end we note that we have the
 equality
 \begin{equation}\label{eqn14}
  \omega(X,Y)-\int_{M}\omega(X,Y)\, \omega^{d}=\sigma([X,Y])-\int_{M}\sigma([X,Y]) \omega^{d}.
 \end{equation}
 Indeed, both sides of \eqref{eqn14} are uniquely determined by the property of
 being mapped to $[X,Y]$ under $\xi$ and having vanishing integral over $M$. If
 $X=X_{f}$ and $Y=X_{g}$ for some $f,g\in C^{\infty}(M)$, then we also have
 \begin{equation}\label{eqn13}
  \int_{M}\omega(X,Y)\; \omega^{d}=
  \int_{M}\{f,g\} \;\omega^{d}=
  \int _{M}(\cL_{X_{f}}g)\;\omega^{d}=
  \int_{M}\cL_{X_{f}}(g \;\omega^{d})=
  \int_{M}d(i_{X_{f}}g \;\omega^{n})=0.
 \end{equation}
 Thus we conclude from \eqref{eqn13} and \eqref{eqn14} that
 \begin{equation*}
  d_{\op{CE}}b(X,Y)=-X.\sigma(Y)+Y.\sigma(X)- \big(\int_{M}\sigma([X,Y])\omega^{d}-\sigma([X,Y])\big)=\omega(X,Y).
 \end{equation*}
 If $M$ is also simply connected, then $\ham(M,\omega)=\symp(M,\omega)$ and the
 assertion follows.
\end{proof}

\begin{remark}\label{rem:prequant}
 We shortly fix our conventions about the periods and prequantisation of
 presymplectic manifolds and integration of abelian extensions of Lie algebras
 to Lie group extensions. Let $N$ be an arbitrary manifold, $V$ be a locally
 convex space and $\omega\in \Omega^{2}(M,V)$ be a $V$-valued closed 2-form.
 Then the \emph{period homomorphism} associated to the cohomology class
 $[\omega]\in H^{2}_{\op{dR}}(N,V)$ is the homomorphism
 \begin{equation*}
  \per_{[\omega]}\from \pi_{2}(N)\to V,\quad [\sigma]\mapsto \int_{S^{2}}\sigma^{*}\omega,
 \end{equation*}
 where $\sigma\from S^{2}\to N$ is a smooth representative of $[\sigma]$ (cf.\
 \cite[Appendix
 A.3]{Neeb02Central-extensions-of-infinite-dimensional-Lie-groups} or
 \cite[Corollary
 14]{Wockel06A-Generalisation-of-Steenrods-Approximation-Theorem}). Moreover,
 let $\Gamma\se\R$ an arbitrary but fixed discrete subgroup and set
 $A_{\Gamma}:=\fa/\Gamma$ denote by $q_{\Gamma}\from \fa\to A_{\Gamma}$ the
 canonical quotient homomorphism. In addition, we assume that $V$ comes along
 with a distinguished embedding $\R\hookrightarrow V$.
 
 We now consider two special cases of this. At first, assume that $V=\R$ and
 that $N$ is finite-dimensional (or more generally smoothly paracompact) and
 1-connected. Then the subgroup $\per_{[\omega]}(\pi_{2}(M))$ is called the
 group of \emph{periods} (or shortly just the periods) of $(M,[\omega])$.
 Moreover, we say that $(M,\omega)$ is $\Gamma$-\emph{prequantisable} if there
 a principal $\TGamma{\Gamma}:=\R/\Gamma$-bundle $P\to N$ that admits a
 connection with curvature $\omega$. By the general theory (see, e.g.,
 \cite[Chapter 8]{Woodhouse92Geometric-quantization} or
 \cite{Kostant70Quantization-and-unitary-representations.-I.-Prequantization})
 this is the case if and only if the periods $\per_{[\omega]}(\pi_{2}(M))$ are
 contained in $\Gamma$ (or more generally if $[\omega]$ is contained in the
 image of $H^{2}(N,\Gamma)\to H^{2}(M,V)$ if $N$ is not simply connected).
 
 The other case is that $N=K$ and $V=\fa$ for $K$ and $\fa$. Then we consider
 the equivariant extension $\ol{\omega}^{\op{eq}}\in \Omega^{2}(K,\fa)$ of
 $\ol{\omega}$, which is given by
 \begin{equation*}
  (\ol{\omega}^{\op{eq}})_{k}\from T_{k}K\times T_{k}K\to \fa,\quad(X,Y)\mapsto k.\ol{\omega}(X\cdot k^{-1},Y\cdot k^{-1}).
 \end{equation*}
 Note that we have here used the right trivialisation of the tangent bundle of
 a Lie group, which is more adapted to our setting than the left trivialisation
 (cf.\ \cite[Remarks 2.4 and 2.6]{SchmedingWockel14}). Then we call
 $\per_{[\ol{\omega}^{\op{eq}}]}(\pi_{2}(K))$ the group of \emph{primary
 periods} (or shortly just the primary periods) of $(K,[{{\omega}]})$. Note
 that they really only depend on the cohomology class of $[{\omega}]$ in
 $H^{2}_{\op{dR}}(M,\R)$ by the following lemma. Moreover, the primary periods
 $\per_{[\ol{\omega}^{\op{eq}}]}(\pi_{2}(K))$ are contained in the fixed points
 $\fa^{K}$ \cite[Lemma
 4.2]{Neeb04Abelian-extensions-of-infinite-dimensional-Lie-groups}. If they are
 also contained in $\Gamma$, then there exists an extension of Lie groups
 \begin{equation}\label{eqn17}
  A_{\Gamma}\to K^{\sharp}\xrightarrow{q^{\sharp}} \wt{K},
 \end{equation}
 whose underlying extension of Lie algebras is equivalent to
 $\fa\to \fa\oplus_{\omega}\fg\to\fg$ \cite[Theorem
 6.7]{Neeb04Abelian-extensions-of-infinite-dimensional-Lie-groups}. Here and in
 the sequel, $q_{\pi_{1}(K)}\from \wt{K}\to K$ denotes the universal covering
 morphism. The extension \eqref{eqn17} is uniquely determined (up to
 equivalence) by the associated Lie algebra extension \cite[Theorem
 7.2]{Neeb04Abelian-extensions-of-infinite-dimensional-Lie-groups}. Moreover,
 $K^{\sharp}\to \wt{K}$ is a principal $A_{\Gamma}$-bundle that admits a
 connection with curvature $\omega^{\op{eq}}$, so $\omega^{\op{eq}}$ is
 $\Gamma$-prequantisable on $\wt{K}$ in this case (cf.\ the proof of Lemma
 \ref{lem:secondary_discreteness}). The question, whether
 $A_{\Gamma}\to K^{\sharp}\to \wt{K}$ factors to an extension of $K$ by
 $A_{\Gamma}$ is then controlled by the flux homomorphism
 \begin{equation*}
  F_{[\ol{\omega}]}\from \pi_{1}(K)\to H^{1}_{\op{dR}}(M,\R)\se H^{1}_{c}(\fk,\fa).
 \end{equation*}
 (cf.\ \cite[Lemma 6.2, Proposition 6.3 and Proposition
 9.13]{Neeb04Abelian-extensions-of-infinite-dimensional-Lie-groups}). Since $M$
 is assumed to be 1-connected, $F_{[\ol{\omega}]}$ vanishes automatically.
 Almost all flux phenomena will be irrelevant for this paper since we only work
 with 1-connected $M$. The flux will only occur shortly in Remark
 \ref{rem:relation_of_extension_groups}. One main observation of this section
 is that the integration of the Lie algebra extension
 $\fa\to \fa\oplus_{\ol{\omega}}\fk\to\fk$ to a transitive pair is \emph{not}
 governed by the flux, but rather by the \emph{secondary} periods that we will
 introduce in Remark \ref{rem:secondary_periods}.
\end{remark}

\begin{lemma}\label{lem:injective}
 The map
 \begin{equation}\label{eqn9}
  H^{2}_{\op{dR}}(M,\R)\to H^{2}_{c}(\cV(M),C^{\infty}(M)),\quad [\omega]\mapsto [\ol{\omega}].
 \end{equation}
 is injective.
\end{lemma}

\begin{proof}
 We may assume without loss of generality that $\dim(M)\geq 2$. Let
 $H^{2}_{c,\Delta}(\cV(M),C^{\infty}(M))$ be the diagonal cohomology of $\cV(M)$
 with coefficients in $C^{\infty}(M)$, i.e., the cohomology of the subcomplex
 of $\R$-linear alternating cochains
 \begin{equation*}
  \xi\from\cV(M)\times\cV(M)\to C^{\infty}(M)
 \end{equation*}
 for which $\xi(X,Y)(m)$ only depends on the germs of $X$ and $Y$ at $m$. Since
 $\omega\in \Omega^{2}(M)$ is $C^{\infty}$-linear in both arguments,
 $\ol{\omega}$ is a cocycle of this kind, so that the map \eqref{eqn9} factors
 as
 \begin{equation}\label{eqn10}
  H^{2}_{\op{dR}}(M,\R)\to H^{2}_{c,\Delta}(\cV(M),C^{\infty}(M)) \to H^{2}_{c}(\cV(M),C^{\infty}(M)).
 \end{equation}
 Now the first map in \eqref{eqn10} is injective by \cite[Corollary
 2]{Losik70The-cohomology-of-infinite-dimensional-Lie-algebras-of-vector-fields.}
 and the second map in \eqref{eqn10} is injective by \cite[Theorem
 2.4.10]{Fuks86Cohomology-of-infinite-dimensional-Lie-algebras}, and the
 assertion follows.
\end{proof}

The following is the only result on the flux homomorphism that we shall need in the sequel.

\begin{lemma}\label{lem:pointed_flux_vanishes}
 The flux
 \begin{equation*}
  F_{[\ol{\omega}_{m}]}\from \pi_{1}(K_{m})\to H^{1}_{c}(\fk_{m},\fa_{m})
 \end{equation*}
 also vanishes. More generally, if $M$ is only assumed to be compact, then
 $F_{[\ol{\omega}_{m}]}$ vanishes if $F_{[\ol{\omega}]}$ does so.
\end{lemma}

\begin{proof}
 This follows immediately from the commuting diagram
\begin{equation*}
  \xymatrix{
  \wt{K_{m}}\ar[r] \ar@{=}[d] & 
  \wt{K}\ar[r]^(.35){F_{[\ol{\omega}]}} & H^{1}_{\op{dR}}(M,\R)\xyhookrightarrow{1.8em} & H^{1}_{c}(\fk,\fa)\ar[r]^{\res_{\fk_{m}}} &  H^{1}_{c}(\fk_{m},\fa)\ar[d]_(.45){\pr_{m}} \\
   \wt{K_{m}} \ar[rrrr]^-{F_{[\ol{\omega}_{m}]}}  & & & &  H^{1}_{c}(\fk_{m},\fa_{m}),
  }
 \end{equation*}
 where $\wt{K_{m}}\to \wt{K}$ is induced by the inclusion $K_{m}\hookrightarrow K$,
 $\res_{\fk_{m}}$ is induced by the inclusion $\fk_{m}\hookrightarrow \fk$ and
 $\pr_{m}$ is induced by the morphism $\fa\to\fa_{m}$, $f\mapsto f-f(m)$ of
 $\fk_{m}$-modules.
\end{proof}

\begin{tabsection}
 The previous remark introduces the most important concepts from the
 integration theory of abelian extensions of Lie groups that occur in the
 sequel, for the rest see
 \cite{Neeb04Abelian-extensions-of-infinite-dimensional-Lie-groups}. From
 Proposition \ref{prop:restriction_of_canonical_cocycle_to_ham_is_coboundary}
 we immediately obtain the following
\end{tabsection}

\begin{corollary}\label{cor:primary_periods_vasish_for_symp}
 If $K=\Symp(M,\omega)_{0}$ for $\omega\in \Omega^{2}(M)$ symplectic, then
 $\per_{[\ol{\omega}]}(\pi_{2}(K))=0$.
\end{corollary}

\begin{remark}\label{rem:extension_of_K_m}
 Keeping the notation as in Remark \ref{rem:prequant}, we can also integrate
 the restricted extension
 \begin{equation*}
  \fa_{m}\to \fa_{m}\oplus_{\ol{\omega}_{m}}\fk_{m}\to\fk_{m}
 \end{equation*}
 to a unique extension
 \begin{equation*}
  \fa_{m}\to K^{\sharp}_{m}\to \wt{K_{m}}.
 \end{equation*}
 Indeed, $\per_{\ol{\omega}_{m}}(\pi_{2}(K_{m}))$ is contained in the
 fixed-points of $K_{m}$. Since $K_{m}$ is assumed to act transitively on
 $M\setminus \{m\}$, if follows that
 \begin{equation*}
  (\fa_{m})^{K_{m}}=\{f\in C^{\infty}(M)\mid f\text{ is constant on }M\setminus\{m\}\text{ and }f(m)=0\}=\{0\}
 \end{equation*}
 by the continuity of $f$. Consequently,
 $\per_{\ol{\omega}_{m}}(\pi_{2}(K_{m}))=\{0\}$ and \cite[Theorem
 6.7]{Neeb04Abelian-extensions-of-infinite-dimensional-Lie-groups} implies that
 there exists an extension $\fa_{m}\to {K}_{m}^{\sharp}\to \wt{K_{m}}$ whose
 derived extension is equivalent to
 $\fa_{m}\to \fa_{m}\oplus_{\ol{\omega}_{m}}\fk_{m}\to \fk_{m}$.
\end{remark}

\begin{tabsection}
 Now we have the extensions ${K}^{\sharp}\to \wt{K}$ from Remark
 \ref{rem:extension_of_K_m} and $K_{m}^{\sharp}\to \wt{K_{m}}$ from Remark
 \ref{rem:extension_of_K_m} at hand, and we want to build our transitive pair
 from them. The question is now how these two extensions are relate to one
 another. The crucial information about this in contained in the modification
 of the period homomorphism $\per_{[\omega]}^{\flat}$ and the associated
 secondary periods that we introduce and analyse now.
\end{tabsection}

\begin{remark}\label{rem:secondary_periods}
 From the transitive action of $K$ on $M$ we obtain the principal
 $K_{m}$-bundle $\ev_{m}\from K\to M$. This gives in particular rise to the
 long exact sequence
 \begin{equation*}
  \cdots\to \pi_{2}(K_{m})\to \pi_{2}(K)\xrightarrow{\ev_{m}} \pi_{2}(M)\xrightarrow{\delta} \pi_{1}(K_{m})\to \pi_{1}(K)\to\cdots
 \end{equation*}
 At $\pi_{2}(M)$, this induces a short exact sequence
 \begin{equation*}
  0 \to\Lambda\xrightarrow{\ev_{m}}  \pi_{2}(M)  \xrightarrow{\delta} \Delta\to 0
 \end{equation*}
 with $\Delta:=\ker(\pi_{1}(K_{m})\to \pi_{1}(K))$,
 $\Lambda:=\im(\pi_{2}(K)\to \pi_{2}(M))$. From \cite[Theorem
 3.18]{Neeb11Lie-groups-of-bundle-automorphisms-and-their-extensions} it
 follows that
 $\per_{[\ol{\omega}]}([\gamma])=\per_{[\omega]}([\ev_{m}\circ \gamma])$ holds
 for $\gamma\in C_{*}^{\infty}(\bS^{2},K)$ and thus $\per_{[\ol{\omega}]}$
 factors through $\left.\per_{[\omega]}\right|_{\Lambda}$.
 
 If now $\Gamma\leq\R$ is a (not necessarily discrete) subgroup that contains
 $\per_{[\ol{\omega}]}(\pi_{2}(K))$ and $q_{\Gamma}\from \R\to \R/\Gamma$ is
 the quotient map, then $q_{\Gamma} \circ \per_{[\omega]}$ factors through
 $\delta$ and a homomorphism
 $\per^{\flat}_{[\omega]}\from \Delta\to \R/\Gamma$. This gives rise to a
 morphism
 \begin{equation*}
  \xymatrix{
  0\ar[r] & \Lambda \ar[r] \ar[d]^{\per_{[{\ol{\omega}}]}}& \pi_{2}(M) \ar[r]\ar[d]^{\per_{[\omega]}} & \Delta \ar[r]\ar[d]^{\per^{\flat}_{[\omega]}} & 0\\
  0\ar[r] & \Gamma \ar[r] & \R \ar[r] & \R/\Gamma \ar[r] & 0
  }
 \end{equation*}
 of short exact sequences.
\end{remark}

\begin{definition}
 If, in the setting of the previous remark,
 $\Gamma:=\per_{[\ol{\omega}]}(\pi_{2}(K))$ is the group of primary
 periods (cf.\ Remark \ref{rem:prequant}), then we call
 \begin{equation*}
  \per^{\flat}_{[\omega]}(\pi_{1}(K_{m}))\se\R/\Gamma
 \end{equation*}
 the group of \emph{secondary periods} (or shortly the secondary periods) of
 $(K,[\omega])$.
\end{definition}

\begin{remark}\label{rem:relation_of_extension_groups}
 We keep the notation from Remark \ref{rem:prequant} and Remark
 \ref{rem:extension_of_K_m}. The inclusion
 $\fa_{m}\oplus _{\ol{\omega}_{m}}\fk_{m}\hookrightarrow \fa\oplus_{\omega} \fk$
 induces a unique homomorphism
 \begin{equation*}
  \varphi_{m}\from K_{m}^{\sharp}\to K
 \end{equation*}
 (cf.\ \cite[Theorem 8.1]{Milnor84Remarks-on-infinite-dimensional-Lie-groups})
 that makes the diagram
 \begin{equation}
  \label{eqn15} \vcenter{  \xymatrix{ 
  \fa_{m}\ar[r] \ar[d]& \ar[d]^{\varphi_{m}} K_{m}^{\sharp}\ar[r]^{q^{\sharp}_{m}}&\wt{K_{m}}\ar[d]^{q_{\Delta}}\\
  A_{\Gamma} \ar[r]& K^{\sharp}\ar[r] & \wt{K}
  }}
 \end{equation}
 commute. Here, $\fa_{m}\hookrightarrow A_{\Gamma}=\fa/{\Gamma}$ is the
 morphism induced from the embedding $\fa_{m}\hookrightarrow \fa$ and the
 quotient $\fa\to\fa/\Gamma$ (note that $\fa_{m}\cap \Gamma=\{0\}$).
 The morphism $q_{\Delta}\from \wt{K_{m}}\to\wt{K}$ in \eqref{eqn15} is induced
 from the inclusion $K_{m}\hookrightarrow K$. The kernel $\ker(q_{\Delta})$ can
 be identified with $\Delta=\ker(\pi_{1}(K_{m})\to \pi_{1}(K))$ and the image
 with $\wt{K_{m}}/\Delta$. Then the latter is a closed Lie subgroup of
 $\wt{K}$. This gives rise to an extension
 \begin{equation*}
  A_{m}^{\sharp}\longrightarrow K^{\sharp}_{m}\xrightarrow{q_{\Delta}\circ q^{\sharp}_{m}} \wt{K_{m}}/\Delta
 \end{equation*}
 of $\wt{K_{m}}/\Delta$ by $A^{\sharp}_{m}:=(q_{m}^{\sharp})^{-1}(\Delta)$.
 Now $F_{[\omega_{m}]}$ vanishes by Lemma \ref{lem:pointed_flux_vanishes}, whence
 there exists a homomorphism $\sigma\from \Delta\to Z(K_{m}^{\sharp})$ with
 $q_{m}\circ \sigma =\id_{\Delta}$ \cite[Corollary
 6.6]{Neeb04Abelian-extensions-of-infinite-dimensional-Lie-groups}.
\end{remark}

\begin{lemma}\label{lem:secondary_discreteness}
 In the situation of Remark \ref{rem:relation_of_extension_groups} and Remark
 \ref{rem:secondary_periods} we have
 \begin{equation*}
  \varphi_{m}( \sigma( x))+\fa_{m}=\per_{[{\omega}]}^{\flat}(x)+\fa_{m}
 \end{equation*}
 for each $x\in \Delta$.
\end{lemma}

\begin{proof}
 We first recall the following fact about abelian extensions of Lie groups. If
 $A\to \wh{G}\xrightarrow{q} G$ is such an extension and if
 $\tau\from \fg:=\Lf(G)\to \wh{\fg}:=\Lf(\wh{G})$ is a continuous linear
 splitting of the induced extension $\Lf(A)\to\Lf(\wh{G})\to\Lf(G)$, then we
 obtain a connection on the principal right $A$-bundle $q\from \wh{G}\to G$,
 induced for each $\wh{g}\in \wh{G}$ by the horizontal lift of tangent vectors
 \begin{equation*}
  \tau_{\wh{g}}\from  T_{g}G\to T_{\wh{g}}\wh{G},\quad v\mapsto \tau(v\cdot g^{-1})\cdot \wh{g}
 \end{equation*}
 where $g:=q(\wh{g})$. The curvature of this connection is given by
 $\omega^{\op{eq}}_{\tau}$, where $\omega_{\tau}\from \fg\times\fg\to\Lf(A)$ is
 the abelian Lie algebra cocycle $(x,y)\mapsto \tau([x,y])-[\tau(x),\tau(y)]$.
 Note again that we have here used the right trivialisation of the tangent
 bundle.\medskip
 
 For each $x\in \Delta$ we now choose a representative
 $\gamma_{x}\in C^{\infty}_{*}(\bS^{1},K_{m})$ in $\pi_{1}(K_{m})$ such that
 there exists a smooth map $F_{x}\from [0,1]\times\bS^{1}\to K$ with
 $F_{x}(0,\cdot)\equiv e_{K}$, $F_{x}(1,\cdot)=\gamma_{x}$ and
 $F_{x}(t,*)=\id_{M}$ for all $t\in [0,1]$. If $\gamma_{x}^{\sharp}$ denotes
 the horizontal lift of $\gamma_{x}$, induced by the canonical linear splitting
 $\fk_{m}\to \fa_{m}\oplus_{\ol{\omega}_{m}}\fk_{m}$ and the Lie algebra
 isomorphism $\Lf(K_{m}^{\sharp})\cong \fa_{m}\oplus_{\ol{\omega}_{m}}\fk_{m}$,
 then there exists some $b_{\sigma}\from \Delta\to \fa_{m}$ such that
 $\sigma(x)=\gamma_{x}^{\sharp}(1)+b_{\sigma}(x)$ (cf.\ \cite[Corollary
 6.6]{Neeb04Abelian-extensions-of-infinite-dimensional-Lie-groups}).
 
 By the construction of the morphism $\varphi_{m}$ (cf.\ \cite[Theorem
 8.1]{Milnor84Remarks-on-infinite-dimensional-Lie-groups}),
 \begin{equation*}
  \varphi_{m}(\sigma(x))=\varphi_{m}(\gamma_{x}^{\sharp}(1))+\varphi_{m}(b_{\sigma}(x))= \Gamma_{x}(1)+\varphi_{m}(b_{\sigma}(x)),
 \end{equation*}
 where $\Gamma_{x}\in C^{\infty}([0,1],K^{\sharp})$ is the solution of the
 initial value problem in $K^{\sharp}$ with $\Gamma_{x}(0)=e_{K^{\sharp}}$ and
 $\delta^{l}(\Gamma_{x})=\iota \circ \delta^{l}(\gamma^{\sharp}_{x})$ for
 $\iota\from \fa_{m}\oplus_{\ol{\omega}_{m}}\fk_{m}\to \fa\oplus_{\ol{\omega}}\fk$
 the canonical embedding. Note that $\Gamma_{x}$ exists as $K^{\sharp}$ is
 $C^\infty$-regular as an extension of $C^\infty$-regular Lie groups (cf.\
 \cite[Appendix
 B]{NeebSalmasian12Differentiable-vectors-and-unitary-representations-of-Frechet-Lie-supergroups}).
 This implies in particular that $\Gamma_{x}$ is a horizontal lift of the
 smooth loop $\gamma'_{x}:= q_{\Delta} \circ \gamma_{x}$ for the connection on
 the principal $A_{\Gamma}$-bundle $K^{\sharp}\to \wt{K}$ that is induced by
 the canonical linear splitting $\fk\to \fa\oplus \fk$ and the Lie algebra
 isomorphism $\Lf(K^{\sharp})\cong \fa\oplus_{\omega}\fk$. Consequently,
 $\varphi_{m}(\gamma_{x}^{\sharp}(1))=\Gamma_{x}(1)$ equals the holonomy of the
 loop $\gamma_{x}'$ for this connection.
 
 If we take the restriction $\left.K^{\sharp}\right|_{\wt{K_{m}}/\Delta}$ of
 the $A_{\Gamma}$-bundle $K^{\sharp}\to K$ to the Lie subgroup
 $\wt{K_{m}}/\Delta$, then the curvature of the above connection takes values
 in the subspace $\fa_{m}$ of $A_{\Gamma}\cong \fa_{m}\times \TGamma{\Gamma}$.
 Consequently, the connection is flat modulo $\fa_{m}$ and thus the
 $\TGamma{\Gamma}$-component of the holonomy can be computed as the integral of
 the curvature over any filler of the loop $\gamma_{x}'$. This implies
 \begin{equation*}
  \Gamma_{x}(1)+\fa_{m}=\op{hol}(\gamma'_{x})+\fa_{m}=\int_{F} \omega^{\op{eq}}+\fa_{m} =\per_{[{\omega}]}^{\flat}(x)+\fa_{m}.
 \end{equation*}
 Since $\varphi_{m}(b_{\sigma}(x))\se\varphi_{m}(\fa_{m})\se\fa_{m}$, this
 establishes the claim.
\end{proof}

\begin{tabsection}
 The following example illustrates the r\^ole of $\per^{\flat}_{[\omega]}$
 pretty well.
\end{tabsection}

\begin{example}\label{exmp:hoppf_fibration_1}
 Let $\omega$ be the standard volume form on $M:=\bS^{2}$ with total volume
 $1$. Then we consider the subgroup $K:=\Diff(\bS^{2})_{0}$. The action of
 $\SO_{3}(\R)$ on $\bS^{2}$ by rotations induces a map
 $\SO_{3}(\R)\to \Diff(\bS^{2})_{0}$, which is a homotopy equivalence
 \cite{Smale59Diffeomorphisms-of-the-2-sphere}. Consequently, $\pi_{2}(K)=0$
 and the primary periods vanish. Thus we may take $\Gamma=0$ to integrate the
 extension
 \begin{equation*}
  C^{\infty}(\bS^{2})\to C^{\infty}(\bS^{2})\oplus_{\ol{\omega}}\cV(\bS^{2})\to\cV(\bS^{2})
 \end{equation*}
 to an extension
 \begin{equation*}
  C^{\infty}(\bS^{2})\to K^{\sharp}\to \wt{K}
 \end{equation*}
 of Lie groups (cf.\ Remark \ref{rem:prequant}). From the five lemma and the
 long exact sequence in homotopy groups of the fibrations
 $K_{m}\to K\to \bS^{2}$ and $\SO_{2}(\R)\to \SO_{3}(\R)\to \bS^{2}$ it follows
 that the induced map $\SO_{2}(\R)\to K_{m}$ is also a homotopy equivalence.
 From this it follows that the exact sequence
 \begin{equation*}
  \pi_{2}(K) \to \pi_{2}(\bS^{2})\to \pi_{1}(K_{m})\to \pi_{1}(K)\to \pi_{1}(\bS^{2})
 \end{equation*}
 identifies with
 \begin{equation*}
  0\to \Z\xrightarrow{\cdot 2}\Z\to \Z/2\Z\to 0
 \end{equation*}
 and with respect to this identification we have $\Delta=2\Z$. Since
 $\per_{[\omega]}$ is given by the natural embedding $\Z \hookrightarrow \R$,
 it follows that $  \per^{\flat}_{\ol{\omega}}$ is given by
 \begin{equation*}
  \per^{\flat}_{[{\omega}]}\from \Delta\cong 2\Z\to \R,\quad 2x\mapsto x.
 \end{equation*}
 Thus the secondary periods coincide with $\Z$ and
 $\varphi_{m}(A^{\sharp}_{m})=C_{m}^{\infty}(\bS^{2})\times\Z$.
\end{example}

We now put all the bits and pieces that we have collected so far together.

\begin{theorem}\label{thm:integrating_extensions_to_transitive_pairs}
 Let $M$ be a compact and 1-connected manifold, $\omega\in \Omega^{2}(M)$ be
 closed and let $K\leq\Diff(M)$ be a connected and $C^{\infty}$-regular Lie
 subgroup such that $\ev_{m}\from K\to M$ is a subjective submersion and $K$
 acts 2-fold transitively on $M$. Let $\Delta$ be the kernel of the map
 $\wt{K_{m}}\to \wt{K}$ induced on the universal cover by the inclusion
 $K_{m}\hookrightarrow K$. Suppose $\Gamma,\Pi\se\R$ are discrete subgroups
 with $\per_{[\ol{\omega}]}\se \Gamma\se \Pi$, set
 $A_{\Gamma}:=C^{\infty}(M,\R)/\Gamma$ and identify $A_{\Gamma}$ with
 $C^{\infty}_{m}(M)\times \TGamma{\Gamma}$ (as abelian Lie groups or as
 $K_{m}$-modules). Then the following assertions are equivalent:
 \begin{enumerate}
  \item \label{item:integrating_extensions_to_transitive_pairs_1}
        $\per^{\flat}_{[\omega]}(\Delta)\se\Pi/\Gamma$.
  \item \label{item:integrating_extensions_to_transitive_pairs_2} Let
        \begin{equation*}
         A_{\Gamma} \to {K}^{\sharp}\xrightarrow{q^{\sharp}} \wt{K}
        \end{equation*}
        be the unique extension of $\wt{K}$ by $A_{\Gamma}$ whose Lie algebra
        extension is equivalent to $\fa\to\fa\oplus_{\ol{\omega}}\fk\to\fk$
        with $\fa:=C^{\infty}(M)$. If $\fa_{m}:=C^{\infty}_{m}(M)$, then the
        closed Lie subalgebra $\fa_{m}\oplus_{\ol{\omega}_{m}}\fk_{m}$ of
        $\fa\oplus\fk$ integrates to a closed Lie subgroup $I_{m}$ of
        $K^{\sharp}$ such that
        $I_{m}\cap A_{\Gamma}\se\fa_{m}\times \Pi/\Gamma$.
  \item \label{item:integrating_extensions_to_transitive_pairs_3} Let
        \begin{equation*}
         \fa_{m} \to {K}^{\sharp}_{m}\xrightarrow{q_{m}^{\sharp}} \wt{K_{m}}
        \end{equation*}
        be the unique extension of $\wt{K_{m}}$ by $\fa_{m}$ whose Lie algebra
        extension is equivalent to $\fa_{m}\oplus_{\ol{\omega}_{m}}\fk_{m}$ and
        set $A^{\sharp}_{m}:=(q_{m}^{\sharp})^{-1}(\Delta)$. Then the image of
        the Lie group morphism $\varphi_{m}\from K^{\sharp}_{m}\to K^{\sharp}$
        induced by the canonical embedding
        $\fa_{m}\oplus_{\ol{\omega}_{m}}\fk_{m}\to \fa\oplus_{\ol{\omega}}\fk$
        is a closed Lie subgroup and
        $\varphi_{m}(A^{\sharp}_{m})\se C_{m}^{\infty}(M,\R)\times \Pi/\Gamma$.
 \end{enumerate}
 If one (and thus all) of these conditions is satisfied, then the composition
 of the maps
 \begin{equation*}
  \theta^{\wedge}\from K^{\sharp}\xrightarrow{q^{\sharp}}\wt{K}\xrightarrow{q_{\pi_{1}(K)}} K\hookrightarrow\Diff(M)
 \end{equation*}
 gives rise to a transitive pair
 $(\theta,(\varphi_{m}(K^{\sharp}_{m})\cdot\Pi/\Gamma))$ with kernel
 $\Pi/\Gamma$. Moreover, the associated principal $\T_{\Pi}$-bundle
 $K^{\sharp}/(\varphi_{m}(K^{\sharp}_{m})\cdot\Pi/\Gamma)\to M$ admits a
 connection whose curvature equals $\omega$ and thus is (together with the
 choice of such a connection) a $\Pi$-prequantisation of $\omega$.
\end{theorem}

\begin{proof}
 We will use throughout that $C^{\infty}$-regularity is an extension property
 \cite[Appendix
 B]{NeebSalmasian12Differentiable-vectors-and-unitary-representations-of-Frechet-Lie-supergroups},
 so that all extensions that appear will automatically be $C^{\infty}$-regular
 (which suffices to integrate Lie algebra morphisms to 1-connected Lie groups).
 \begin{description}
  \item[\ref{item:integrating_extensions_to_transitive_pairs_1}$\Rightarrow$\ref{item:integrating_extensions_to_transitive_pairs_3}:]
        Consider $\wt{K_{m}}/\Delta$ as a closed Lie subgroup of $\wt{K}$ and
        obtain from this the extension
        \begin{equation*}
         A_{\Gamma}\to \left.K^{\sharp}\right|_{\wt{K_{m}}/\Delta}\to \wt{K_{m}}/\Delta.
        \end{equation*}
        Since $A_{\Gamma}\cong \fa_{m}\times \TGamma{\Gamma}$ is a
        decomposition of $A_{\Gamma}$ as a $\wt{K_{m}}/\Delta$-module and
        $\sigma\from \Delta\to K^{\sharp}$ induces the isomorphism
        $A^{\sharp}_{m}\cong \fa_{m}\times \Delta$, we obtain from Remark
        \ref{rem:relation_of_extension_groups} and Lemma
        \ref{lem:secondary_discreteness} the morphism
        \begin{equation*}
         \xymatrix{	 
         \fa_{m}\times \Delta\ar[rr]\ar[d]_{\id_{\fa_{m}}\times \per_{[\omega]}^{\flat}} & & K^{\sharp}_{m}\ar[r]\ar[d]^{\varphi_{m}} & \wt{K_{m}}/\Delta\ar@{}[d]|{\ensuremath{\rotatebox[origin=c]{270}{$\xhookrightarrow{~~~~~}$}}}\\
         A_{\Gamma}  \ar[rr] & & K^{\sharp} \ar[r] & \wt{K}
         }
        \end{equation*}
        of extensions of Lie groups. Since
        $\wt{K_{m}}/\Delta\hookrightarrow \wt{K}$ is injective and
        $\im(\per_{[\omega]}^{\flat})\se  \Pi /\Gamma$ is discrete in
        $\TGamma{\Gamma}$, this implies that
        $\left.K^{\sharp}\right|_{\wt{K_{m}}/\Delta}$ reduces to an extension
        of $\wt{K_{m}}/\Delta$ by the closed Lie subgroup
        $\fa_{m}\times \im(\per_{[\omega]}^{\flat})$ of
        $\fa_{m}\times \TGamma{\Gamma}$. This reduction is itself a closed Lie
        subgroup which equals $\varphi_{m}(K^{\sharp}_{m})$ by construction.
        Moreover,
        $\varphi_{m}(A^{\sharp}_{m})\se C_{m}^{\infty}(M)\times \Pi/\Gamma$
        follows from
        $\fa_{m}\times\im(\per_{[\omega]}^{\flat})\se \fa_{m}\times \Pi/\Gamma$.
  \item
        [\ref{item:integrating_extensions_to_transitive_pairs_3}$\Rightarrow$\ref{item:integrating_extensions_to_transitive_pairs_2}:]
        As above we see that $\varphi_{m}(K^{\sharp}_{m})\se K^{\sharp}$ is a
        reduction of $\left.K^{\sharp}\right|_{\wt{K_{m}}/\Delta}$ to an
        extension of $\wt{K_{m}}/\Delta$ by
        $\fa_{m}\times \im(\per_{[\omega]}^{\flat})$, and thus in particular a
        closed Lie subgroup. Thus we may take
        $I_{m}:=\varphi_{m}(K^{\sharp}_{m})$.
  \item
        [\ref{item:integrating_extensions_to_transitive_pairs_2}$\Rightarrow$\ref{item:integrating_extensions_to_transitive_pairs_1}:]
        By construction we have that $\wt{I_{m}}\cong K^{\sharp}_{m}$, so that
        $\varphi_{m}$ factors through the inclusion
        $I_{m}\hookrightarrow K^{\sharp}$ and the universal covering map
        $K^{\sharp}_{m}\to I_{m}$. Thus
        $\per^{\flat}_{[\omega]}(\Delta)\se\Pi/\Gamma $ follows from
        $I_{m}\cap A_{\Gamma}\se \fa_{m}\times \Pi/\Gamma$.
 \end{description}
 It remains to show the assertion that if
 \ref{item:integrating_extensions_to_transitive_pairs_3} is satisfied, then
 $(\theta, \varphi_{m}(K^{\sharp}_{m})\cdot \Pi/\Gamma)$ is a transitive pair
 with kernel $\Pi/\Gamma$ and that
 $K^{\sharp}/(\varphi_{m}(K^{\sharp}_{m})\cdot \Pi/\Gamma)\to M$ is a
 $\Pi$-prequantisation of $\omega$.
 
 Set $H_{m}:=\varphi_{m}(K^{\sharp}_{m})\cdot \Pi/\Gamma$. To show that
 $(\theta,H_{m})$ is a transitive pair, we first show that $H_{m} $ is in fact
 a normal subgroup of $(\theta^{\wedge})^{-1}(K_{m})$. In fact, each
 $g\in (\theta^{\wedge})^{-1}(K_{m})$ may be written as a product
 $a_{0}\cdot g_{m}$ for $g_{m}\in \varphi_{m}(K^{\sharp}_{m})$ and
 $a_0\in \TGamma{\Gamma}$. This is due to the fact that
 $\varphi_{m}(K^{\sharp}_{m})$ is a reduction of
 $ \left.K^{\sharp}\right|_{\wt{K_{m}}/\Delta} $ to an extension by
 $C^{\infty}_{m}(M,\R)\times \im(\per^{\flat}_{[\omega]})$ and that
 $C_{m}^{\infty}(M,\R)\cdot \TGamma{\Gamma}\cong C^{\infty}(M,\R)/\Gamma=A_{\Gamma}$.
 Since the action of $k\in \wt{K}$ on $C^{\infty}(M,\TGamma{\Gamma})$ coincides
 with conjugation action of an arbitrary lift of $k$ to $K^{\sharp}$ it follows
 that
 $a_{0}\in \TGamma{\Gamma}= C^{\infty}(M,T_{\Gamma})^{\wt{K}}\se Z(K^{\sharp})$
 and thus $\Ad(a_{0})=\id_{K^{\sharp}}$.
 Consequently, $\Ad(g)=\Ad(g_{m})$ and thus $\Ad(g)$ preserves the subalgebra
 $\Lf(\varphi_{m}(K^{\sharp}_{m}))$. Furthermore, $\Ad (g) = \Lf (c_g)$ (where
 $c_g$ is conjugation by $g$ and the groups $K^{\sharp}$ and
 $\varphi_{m}(K^{\sharp}_{m})$ are connected and regular Lie groups. Thus Lemma
 \ref{lem: Lmorph:subgp} implies that conjugation by $g$ preserves
 $\varphi_{m}(K^{\sharp}_{m})$. Since
 $H_{m}=\varphi_{m}(K^{\sharp}_{m})\cdot  \Pi/\Gamma$ and
 $\Pi/\Gamma\se \TGamma{\Gamma}\se Z(K^{\sharp})$ follows as above, it also
 follows that conjugation by $g$ preserves $H_{m}$.

 Since $g\in (\theta^{\wedge})^{-1}(K_{m})$ was arbitrary, this shows that
 $H_{m}$ is normal in $(\theta^{\wedge})^{-1}(K_{m})$. To conclude that
 $(\theta,H_{m})$ is a transitive pair it thus suffices to observe that
 $\theta (\cdot,m)= \op{ev}_{m}\op{\circ} q_{\pi_{1}(K)} \op{\circ} q^{\sharp}$
 clearly is a submersion, and $\Lf(\wh{G})/\Lf(H_{m})$ is finite-dimensional,
 so $H_{m}$ in particular co-Banach.
 
 It remains to show that $P:=K^{\sharp}/(H_{m})\to M$ is a
 $\Pi$-prequantisation. First note that $H_{m}$ is an extension of
 $\wt{K_{m}}/\Delta$ by $\fa_{m}\times \Pi/\Gamma$ and
 $(\theta^{\wedge})^{-1}(K_{m})$ is an extension of $\wt{K_{m}}/\Delta$ by
 $\fa_{m}\times \TGamma{\Gamma}$, so that the morphism of Lie groups
 \begin{equation}\label{eqn18}
  \TGamma{\Pi}\to ((\theta^{\wedge})^{-1}(K_{m}))/H_{m},
 \end{equation}
 induced by mapping an element of $\TGamma{\Pi}$ to the respective constant
 function, is an isomorphism. With respect to this isomorphism we endow
 $P\to M$ with the structure of a $\TGamma{\Pi}$-bundle over $M$.
 
 We now construct a connection on $P\to M$ with curvature $\omega$ as follows.
 Let $H\leq T\wt{K}$ be a horizontal distribution on the bundle
 $\ev_{m}\from \wt{K}\to M$, i.e., we have each $k\in \wt{K}$ a subspace
 $\wt{H}_{k}\leq T_{k}K$ such that $\wt{H}_{k\cdot k'}=\wt{H}_{k}\cdot k'$ for
 $k'\in \wt{K_{m}}/\Delta$ and that $T \ev_{m}\from \wt{H}_{k}\to T_{k(m)}M$ is
 a linear isomorphism. Denote by
 $\wt{\sigma}\from \cV(M)\to \cV(\wt{K})^{\wt{K_{m}}/\Delta}$ the corresponding
 horizontal lift of vector fields. On the bundle
 $q^{\sharp}\from K^{\sharp}\to \wt{K}$ we have the connection which is induced
 by the isomorphisms $TR_{k^{-1}}\from T_{k}\wt{K}\to T_{e}\wt{K} \cong\fk$,
 $TR_{\ol{k}^{-1}}\from T_{\ol{k}}K^{\sharp}\to T_{e}K^{\sharp}\cong \fa\oplus_{\ol{\omega}}\fk$
 and the canonical linear splitting
 $\sigma\from \fk\to\fa\oplus_{\ol{\omega}}\fk$. Denote the corresponding
 horizontal lift by
 $\sigma_{\ol{\omega}}\from \cV(\wt{K})\to \cV(K^{\sharp})^{A_{\Gamma}}$. If we
 now set
 \begin{equation*}
  H^{\sharp}_{\ol{k}}:= \sigma(\wt{H}_{k}\cdot k^{-1})\cdot \ol{k}
 \end{equation*}
 for $\ol{k}\in \wt{K}$ and $k:=q^{\sharp}(\ol{k})$ (where we suppressed the
 isomorphisms $T_{k}\wt{K}\cong \fk$ and
 $T_{\ol{k}}K^{\sharp}\cong \fa\oplus_{\ol{\omega}}\fk$), then
 $T_{\ol{k}} (\ev_{m}\circ q^{\sharp})$ also restricts to a linear isomorphism
 on $H^{\sharp}_{\ol{k}}$ and we have
 \begin{equation*}
  H^{\sharp}_{\ol{k} \ol{k}'}=\sigma(\wt{H}_{kk'}(kk')^{-1})\cdot \ol{k}\ol{k}'=H^{\sharp}_{\ol{k}}\cdot \ol{k}'
 \end{equation*}
 for each $\ol{k}'\in (q^{\sharp})^{-1}(\wt{K}_{m}/\Delta)$ and
 $k':=q^{\sharp}(\ol{k}')$. Consequently, $H^{\sharp}$ defines a horizontal
 distribution on the bundle $ \ev_{m}\circ q^{\sharp}\from K^{\sharp}\to M$. If
 $\sigma^{\sharp}\from \cV(M)\to \cV(K^{\sharp})^{(\theta^{\wedge})^{-1}(K_{m})}$
 denotes the corresponding horizontal lift of vector fields, then we clearly
 have $\sigma^{\sharp}= \sigma_{\ol{\omega}} \circ \wt{\sigma}$.
 
 From this we obtain a connection
 \begin{equation*}
  \sigma_{P}\from \cV(M)\to \cV(P)^{\TGamma{\Pi}},\quad X\mapsto TQ_{*}( \sigma^{\sharp}(X))
 \end{equation*}
 on $P\to M$, where $Q\from K^{\sharp}\to P=K^{\sharp}/(H_{m})$ is the
 canonical quotient morphism. For the curvature of the connection $\sigma_{P}$
 we then have
 \begin{equation*}
  F_{\sigma_{P}}(X,Y):=\sigma_{P}([X,Y])-[\sigma_{P}(X),\sigma_{P}(Y)]=TQ_{*} (\sigma ^{\sharp}([X,Y])-[\sigma_{P}(X),\sigma_{P}(Y)])=TQ_{*}( F_{\sigma^{\sharp}}(X,Y))
 \end{equation*}
 for $X,Y\in \cV(M)$, and, furthermore
 \begin{multline*}
  F_{\sigma^{\sharp}}(X,Y)=\sigma_{\ol{\omega}}(\wt{\sigma}([X,Y]))-[\sigma_{\ol{\omega}}(\wt{\sigma}(X)),\sigma_{\ol{\omega}}(\wt{\sigma}(Y))]=\sigma_{\ol{\omega}}(\wt{\sigma}([X,Y]))-\sigma_{\ol{\omega}}([\wt{\sigma}(X),\wt{\sigma}(Y)])+\\F_{\sigma_{\ol{\omega}}}(\wt{\sigma}(X),\wt{\sigma}(Y))=
  \sigma_{\ol{\omega}}(F_{\wt{\sigma}}(X,Y))+F_{\sigma_{\ol{\omega}}}(\wt{\sigma}(X),\wt{\sigma}(Y)).
 \end{multline*}
 Since $F_{\sigma_{\ol{\omega}}}=\ol{\omega}^{\op{eq}}$ (cf.\ the proof of
 Lemma \ref{lem:secondary_discreteness}) and since
 $\sigma_{\ol{\omega}}(F_{\wt{\sigma}}(X,Y))$ is at each point tangential to
 the fibre $H_{m}$ of $Q$, it follows that
 \begin{multline*}
  TQ( F_{\sigma^{\sharp}}(X,Y)(k(m)))=TQ(F_{\sigma_{\ol{\omega}}}(\wt{\sigma}(X)(k),\wt{\sigma}(Y)(k)))
  =TQ(\ol{\omega}^{\op{eq}}(\wt{\sigma}(X)(k),\wt{\sigma}(Y)(k)))
  \\
  =\ev_{m}(k.\ol{\omega}(\wt{\sigma}(X)(k)\cdot k^{-1}, \wt{\sigma}(Y)(k)\cdot k^{-1}))
  =\omega_{k(m)}(X(k(m)),X(k(m)))
 \end{multline*}
 for each $\ol{k}\in K^{\sharp}$ and $k:=q^{\sharp}(\ol{k})\in \wt{K}$. Thus
 $F_{\sigma^{P}}=\omega$.
 
 It remains to check that the kernel actually coincides with $\Pi/\Gamma$. To
 this end we define $q \coloneq q_{\pi_1 (K)} \circ q^{\sharp}$ and consider
 the diagram
 \begin{equation*}
  \xymatrix{
  \ker(q^{\sharp})\ar@{}[d]|{\ensuremath{\rotatebox[origin=c]{270}{$\xhookrightarrow{~~~~~}$}}}
\ar[rr]^{\left.a_{\theta,H_{m}}\right|_{\ker(q^{\sharp})}} && \ker((\beta_\cR)_{*})\ar@{}[d]|{\ensuremath{\rotatebox[origin=c]{270}{$\xhookrightarrow{~~~~~}$}}}
\\ 
  K^\sharp\ar[rr]^(.35){a_{\theta,H_{m}}}\ar[d]^{q^{\sharp}} \ar[rrd]^{\theta^\wedge}&& \Bis(\cR(\theta,H_{m}))\ar[d]^{(\beta_\cR)_{*}}\\
  K\xyhookrrightarrow{4em} && \Diff(M),
  }
 \end{equation*}
 which commutes by Lemma \ref{lem:canonical_morphisms_commute} and the
 construction of $\theta^\wedge$. From this it follows that the kernel of
 $a_{\theta,H_{m}}$ is contained in $\ker(q^{\sharp})$. Moreover, we have
 $\ker(q^{\sharp})=A_{\Gamma}= C^{\infty}(M,\TGamma{\Gamma})$ by definition. To
 determine $a_{\theta,H_{m}}(\gamma)$ for $\gamma\in \ker(q)$, we first note
 that the element in $\Aut(K^\sharp/H_{m}\to M)$ corresponding to
 $a_{\theta,H_{m}}(\gamma)$ is given by
 $\ol{k} H_{m}\mapsto (\gamma\cdot \ol{k}) H_{m}  $ (cf.\ Lemma
 \ref{lem:canonical_morphism_into_bisections}). On the other hand, an element
 $\eta\in \ker((\beta_\cR)_{*})\cong C^{\infty}(M,\TGamma{\Pi})$ acts on
 $K^\sharp/ H_{m} $ by
 \begin{equation*}
  \ol{k} H_{m} \mapsto \ol{k}H_{m}\cdot \eta(\theta(\ol{k},m))= (\ol{k}\cdot  \eta(\theta (\ol{k},m))) H_{m} ,
 \end{equation*}
 since the bundle projection $K^\sharp/ H_{m} \to M$ is given by
 $\ol{k}  H_{m} \mapsto \theta(\ol{k},m)$ and $\TGamma{\Gamma}$ is contained in
 $Z(K^{\sharp})$ (cf.\ \eqref{eqn18}). From this it follows that the value of
 $a_{\theta, H_{m} }(\gamma)$ in $\theta(\ol{k},m)$ has to satisfy
 \begin{equation*}
  (\ol{k}^{-1}\cdot\gamma \cdot \ol{k}) H_{m} =a_{\theta, H_{m} }(\gamma)(\theta(\ol{k},m)) H_{m}.
 \end{equation*}
 Since
 $\ol{k}^{-1}\cdot\gamma \cdot \ol{k}=\gamma \circ \theta^{\wedge}(\ol{k})$
 follows from the fact that $K^\sharp\to K$ is an abelian extension for the
 natural action of $K$ on $C^{\infty}(M,\TGamma{\Gamma})$, we conclude that
 $\left.a_{\theta, H_{m} }\right|_{\ker(q)}$ coincides with the map that is
 induced by the projection
 $\TGamma{\Gamma}\to \TGamma{\Pi}=\TGamma{\Gamma}/(\Pi/\Gamma)$ and the
 isomorphisms $C^{\infty}(M,\TGamma{\Gamma})\cong \ker(q^{\sharp})$ and
 $C^{\infty}(M,\TGamma{\Pi})\cong \ker((\beta_\cR)_{*})$. Consequently,
 $\ker(a_{\theta,H_{m}})=\ker(a_{\theta,H_{m}})\cap \ker(q^{\sharp})\cong \Pi/\Gamma$ (cf.\ Proposition \ref{prop: kernel}).
\end{proof}

\begin{corollary}
 With the notation and under the assumptions of Theorem
 \ref{thm:integrating_extensions_to_transitive_pairs} the following assertions
 are equivalent:
 \begin{enumerate}
  \item The primary periods $\per_{\ol{\omega}}(\pi_{2}(M))\se \R$ are
        discrete.
  \item The extension $\fa\to\fa\oplus_{\ol{\omega}}\fk\to\fk$ integrates to an
        extension $A_{\Gamma}\to K^{\sharp}\to\wt{K}$ of Lie groups.
 \end{enumerate}
 If one (and thus both) of these conditions is satisfied, then the following
 assertions are equivalent:
 \begin{enumerate}
  \item [i)] The secondary periods $\per_{[\omega]}(\pi_{1}(K_{m}))\se \R$ are
        discrete.
  \item [ii)] The closed subalgebra $\fa_{m}\oplus_{\ol{\omega}_{m}}\fk$ of
        $\fa\oplus_{\ol{\omega}}\fk$ integrates to a closed Lie subgroup of
        $K^{\sharp}$.
 \end{enumerate}
\end{corollary}

\begin{corollary}
 If $M$ is a compact and 1-connected manifold and $\omega\in \Omega^{2}(M)$ is
 closed, then the extension $\R\to \R\oplus_{\omega}TM \to TM$ of Lie
 algebroids integrates to an extension of Lie groupoids if and only if the
 extension of Lie algebras
 $C^{\infty}(M)\to C^{\infty}(M)\oplus_{\ol{\omega}}\cV(M)\to\cV(M)$ integrates
 to an extension $A\to \wh{K}\to \Diff(M)_{0}$ of Lie groups and
 $C^{\infty}_{m}(M)\oplus_{\ol{\omega}_{m}}\cV_{m}(M)$ integrates to a closed
 Lie subgroup in $\wh{K}$.
\end{corollary}

\begin{example}\label{exmp:hoppf_fibration_2}
 This is a continuation of Example \ref{exmp:hoppf_fibration_1}. Of course,
 $(\bS^{2},\omega)$ is $\Z$-prequantisable, a (suitably normalised)
 prequantisation is the Hopf fibration $\bS^{3}\to\bS^{2}$, viewed as a
 $U(1):=T_{\Z}$-principal bundle (together with the standard contact form on
 $\bS^{3}$). From this we obtain the extension$ $
 \begin{equation*}
  C^{\infty}(\bS^{2},U(1))\to \Aut(\bS^{3}\to \bS^{2})\to \Diff(\bS^{2})_{0}
 \end{equation*}
 (cf.\ Example \ref{exmp:evaluation_for_Aut(P)}) whose Lie algebra extension is
 equivalent to
 \begin{equation}\label{eqn16}
  C^{\infty}(\bS^{2},\R)\to C^{\infty}(\bS^{2},\R)\oplus_{\ol{\omega}}\cV(\bS^{2})\to\cV(\bS^{2}).
 \end{equation}
 Since the primary periods vanish, the extension \eqref{eqn16} integrates to an
 extension (unique up to equivalence)
 \begin{equation*}
  C^{\infty}(\bS^{2},\R)\to K^{\sharp}\to \wt{\Diff(\bS^{2})_{0}}
 \end{equation*}
 and the identity on $C^{\infty}(\bS^{2},\R)\oplus_{\ol{\omega}}\cV(\bS^{2})$
 integrates to a Lie group morphism
 $\psi\from K^{\sharp}\to \Aut(\bS^{3}\to \bS^{2})$. This morphism makes
 \begin{equation*}
  \xymatrix{
  C^{\infty}(\bS^{2},\R)\ar[r]\ar[d]^{q_{\Z}} & K^{\sharp} \ar[r]\ar[d]^{\psi} & \wt{\Diff(\bS^{2})_{0}}\ar[d]^{q_{\pi_{1}}}\\
  C^{\infty}(\bS^{2},U(1)) \ar[r] & \Aut(\bS^{3}\to \bS^{2}) \ar[r] & \Diff(\bS^{2})_{0}
  }
 \end{equation*}
 commute, where $q_{\Z}$ is induced by the quotient map $\R\to U(1)=\R/\Z$ and
 $q_{\pi_{1}}$ is the universal covering morphism of $\Diff(\bS^{2})_{0}$. If
 now $o\in \bS^{3}$ is mapped to the base-point $m\in \bS^{2}$, then
 $\Aut_{o}(\bS^{3}\to \bS^{2})$ is a closed Lie subgroup of
 $\Aut(\bS^{3}\to\bS^{2})$ (cf.\ Example \ref{exmp:evaluation_for_Aut(P)}) and
 since $q_{\pi_{1}}$ is a covering morphism,
 $\psi^{-1}(\Aut_{o}(\bS^{3}\to \bS^{2}))$ is also a closed Lie subgroup of
 $K^{\sharp}$. Since $\left.\psi\right|_{C^{\infty}(\bS^{2},\R)}=q_{\Z}$, we
 have
 \begin{equation*}
  \psi^{-1}(C^{\infty}(\bS^{2},U(1))\cap \Aut_{o}(\bS^{3}\to \bS^{2}))= C_{m}^{\infty}(\bS^{2},\R)\times \Z,
 \end{equation*}
 and thus $\psi^{-1}(\Aut_{o}(\bS^{3}\to \bS^{2}))$ gives rise to an extension
 \begin{equation*}
  C^{\infty}_{m}(\bS^{2},\R)\times\Z  \to\psi^{-1}(\Aut_{o}(\bS^{3}\to \bS^{2})) \to  q_{\pi_{1}}^{-1}(\Diff_{m}(\bS^{2})_{0}).
 \end{equation*}
 If we identify $q_{\pi_{1}}^{-1}((\Diff(\bS^{2})_{0})_{m})$ with
 $\wt{\Diff_{m}(\bS^{2})_{0}}/\Delta$ (for $\Delta$ as in Remark
 \ref{rem:secondary_periods}), then we deduce from Example
 \ref{exmp:hoppf_fibration_1} and Theorem
 \ref{thm:integrating_extensions_to_transitive_pairs} that
 $\psi^{-1}(\Aut_{o}(\bS^{3}\to \bS^{2}))$ is precisely the Lie subgroup
 $\varphi_{m}((\Diff(\bS^{2})_{0})_{m}^{\sharp})$.
\end{example}

\begin{example}
 An example where the conditions of Theorem
 \ref{thm:integrating_extensions_to_transitive_pairs} are not fulfilled is the
 following (cf.\ \cite[Example
 1]{TsengZhu06Integrating-Lie-algebroids-via-stacks}). Let $\eta$ be the
 standard volume form on $\bS^{2}$ with total volume $1$. On
 $M:=\bS^{2}\times \bS^{2}$, consider the form
 $\omega\in \Omega^{2}(\bS^{2}\times \bS^{2})$
 \begin{equation*}
  \omega_{(p,q)}\from T_{(p,q)}\bS^{2}\times \bS^{2}\cong T_{p}\bS^{2}\times T_{q}\bS^{2}\to \R,\quad
  (x,y)\mapsto \eta(x)+\lambda \eta(y).
 \end{equation*}
 Then $\pi_{2}(\bS^{2}\times \bS^{2})= \Z\times \Z$ and we have
 $\per_{[\omega]}((1,0))=1$ and $\per_{[\omega]}((0,1))=\lambda$. Thus
 $\per_{[\omega]}(\pi_{2}(S^{2}\times S^{2}))$ is the subgroup of $\R$ which is
 generated by $1$ and $\lambda$. If $\lambda\notin \Q$, then this is not
 contained in any discrete subgroup.
 
 If we take $K=\Symp(M,\omega)$, then the primary periods vanish by Corollary
 \ref{cor:primary_periods_vasish_for_symp} and we may take $\Gamma=\{0\}$ in Theorem
 \ref{thm:integrating_extensions_to_transitive_pairs}. Consequently, the
 secondary periods $\per_{[\omega]}^{\flat}(\pi_{1}(K_{m}))$ are not contained
 in any discrete subgroup of $\R$ and the subalgebra
 $C_{m}^{\infty}(M)\oplus _{\ol{\omega}_{m}}\fk_{m}$ does not integrate to a
 closed Lie subgroup in the extension
 \begin{equation*}
  C^{\infty}(M,\R)\to {K}^{\sharp}\to \wt{K}.
 \end{equation*}
\end{example}
 
\begin{problem}
 If one takes the results of this section, then the following questions seem to
 be natural and interesting.
 \begin{enumerate}
  \item How do the primary and secondary periods
        $\per_{[\ol{\omega}]}(\pi_{2}(K))$ and
        $\per^{\flat}_{[\omega]}(\pi_{1}(K_{m}))$ vary if one varies the
        subgroup $K$? It is clear that for $K,K'$ with $K\leq K'$ we have
        $\per_{[\ol{\omega}]}(\pi_{2}(K))\leq \per_{[\ol{\omega}]}(\pi_{2}(K'))$
        and
        $\per^{\flat}_{[\omega]}(\pi_{1}(K_{m}))\leq \per^{\flat}_{[\omega]}(\pi_{1}(K'_{m}))$,
        but under which assumptions does one have equality here? In particular,
        it would be interesting to have a symplectic manifold with
        non-vanishing primary periods for $K=\Diff(M)_{0}$ (since for
        $K=\Symp(M,\omega)$ the primary periods always vanish by Proposition
        \ref{cor:primary_periods_vasish_for_symp}).
  \item It would be interesting to to develop an integration theory for
        infinitesimal transitive pairs (cf.\ Problem \ref{prob1}). In
        particular, this should shed some further light on the precise relation
        between the integration theory of Lie algebroids, Lie algebras (of
        sections) and the associated obstructions.
  \item What is the interplay between the primary and secondary periods and the
        flux group
        \begin{equation*}
         F_{[\omega]}(\pi_{1}(\Symp(M,\omega)))\se H^{1}(M,\R)
        \end{equation*}
        in the case that $M$ is only assumed to be connected? Conjecturally,
        there might be a relation of the long exact homotopy sequence of the
        evaluation fibration and the one induced by $\Gamma\to\R\to\R/\Gamma$
        \begin{equation*}
         \xymatrix{\pi_{2}(\Symp(M,\omega))\ar[r]\ar[d]^{\per_{[\ol{\omega}]}} & \pi_{2}(M) \ar[r]\ar[d]^{\per_{[\omega]}} & \pi_{1}(\Symp_{m}(M,\omega))\ar[r]\ar@{-->}[d]^{?} & \pi_{1}(\Symp(M,\omega))\ar[d]^{F_{[\omega]}}\\
         H^{0}(M, \Gamma)\ar[r] & H^{0}(M,\R) \ar[r] & H^{0}(M,\R/\Gamma)\ar[r] & H^{1}(M,\Gamma)
         }
        \end{equation*}
        in case that the primary periods
        $\per_{[\ol{\omega}]}(\pi_{2}(\Symp(M,\omega)))\leq \R$ are contained
        in the discrete subgroup $\Gamma\leq \R$ and the flux group is contained in
        $H^{1}(M,\Gamma)$. Note that the flux group is known to be discrete by the
        proof of the flux conjecture
        \cite{Ono06Floer-Novikov-cohomology-and-the-flux-conjecture} and that
        both, the secondary periods and the flux subgroup are related to the
        integrability of Lie subalgebras to closed Lie subgroups (cf.\
        \cite[Proposition
        10.20]{McDuffSalamon98Introduction-to-symplectic-topology}). The
        conjectural homomorphism
        $\xymatrix{\pi_{1}(\Symp_{m}(M,\omega))\ar@{-->}[r]& H^{0}(M,\R/\Gamma)}$
        should be related to the fluxes $F_{[\ol{\omega}]}$,
        $F_{[\ol{\omega}_{m}]}$ and the homomorphism
        $\varphi_{m}\from K^{\sharp}_{m}\to K^{\sharp}$, restricted to the
        pre-image $(q_{m}^{\sharp})^{-1}(\pi_{1}(\Symp_{m}(M,\omega)))$.
        However, this involves the (continuous) Lie algebra cohomology of
        $\fk_{m}$ with coefficients in $C_{m}^{\infty}(M)$, a topic
        that goes beyond the scope of the present paper.
 \end{enumerate}
\end{problem}

\newpage
  \appendix
  
  \section{Local bisections for infinite-dimensional Lie groupoids}
  
  In this appendix we prove that (infinite-dimensional) Lie groupoids over a finite-dimensional manifold admit local bisections through each point.
  Consequently, we are able to derive that their vertex groups are in a natural way Lie groups. 
  These results are standard for finite-dimensional Lie groupoids. 
  We repeat them here for the readers convenience since some details of proofs need to be adapted for our infinite-dimensional setting. 
  
  \begin{definition}
 Let $\cG = (G \toto M)$ be a locally convex Lie groupoid. For $U \opn M$, a \emph{local bisection of $\cG$ on $U$} is a smooth map $\sigma \colon U \rightarrow G$ such that $\alpha \circ \sigma = \id_U$ and $\beta \circ \sigma \colon U \rightarrow (\beta \circ \sigma) (U) \opn M$ is a diffeomorphism. 
  \end{definition}
  
  \begin{lemma}[{{\cite[Proposition 1.4.9]{Mackenzie05General-theory-of-Lie-groupoids-and-Lie-algebroids}}}]\label{lem: enough:locbis}
         Let $\cG = (G \toto M)$ be a locally convex Lie groupoid such that $M$ is a finite dimensional manifold. 
         For each $g \in G$ there exists an $\alpha (g)$-neighborhood $U \opn M$ and a local bisection $\sigma$ of $\cG$ on $U$ such that $\sigma (\alpha (g)) = g$.
        \end{lemma}

 \begin{proof}
  Note that in \cite{Mackenzie05General-theory-of-Lie-groupoids-and-Lie-algebroids} only finite-dimensional Lie groupoids are discussed. 
 However, to prove the assertion, one can copy the proof of \cite[Proposition 1.4.9]{Mackenzie05General-theory-of-Lie-groupoids-and-Lie-algebroids} verbatim to our setting, since the Lie groupoid is assumed to have a finite-dimensional base.
 The crucial point here is that the algebraic argument used in the proof of \cite[Proposition 1.4.9]{Mackenzie05General-theory-of-Lie-groupoids-and-Lie-algebroids} carries over to subspaces of finite codimension of arbitrary locally convex spaces.
 Assuming that the base $M$ is finite-dimensional ensures exactly this property.
 \end{proof}
 
 For mappings into finite-dimensional manifolds (whose domain is an infinite-dimensional manifold), one can define maps of constant rank analogously to the finite-dimensional case (cf.\ \cite[Theorem F]{hg2015}). 
 \begin{corollary}[{{\cite[Corollary 1.4.10]{Mackenzie05General-theory-of-Lie-groupoids-and-Lie-algebroids}}}]\label{cor: const:rk}
  Let $\cG = (G \toto M)$ be a locally convex Lie groupoid over a finite dimensional manifold $M$.
  Then for each $m\in M$ the maps
    \begin{displaymath}
     \beta|_{\alpha^{-1} (m)} \colon \alpha^{-1} (m) \rightarrow M, g \mapsto \beta (g) \text{ and } \alpha|_{\beta^{-1} (m)} \colon \beta^{-1} (m) \rightarrow M, g \mapsto \alpha (g)
    \end{displaymath}
  are maps of constant rank.
 \end{corollary}

 The next proof follows \cite[Corollary 1.4.11]{Mackenzie05General-theory-of-Lie-groupoids-and-Lie-algebroids} but we need to adapt the arguments. 
 \begin{lemma}\label{lem: vertex:Lie}
  Let $\cG = (G \toto M)$ be a locally convex Lie groupoid over a finite dimensional manifold $M$.
  Then for all $m,n \in M$, $\alpha^{-1} (m) \cap \beta^{-1} (n)$ is a split submanifold (of finite codimension) of $\alpha^{-1} (m)$, of $\beta^{-1} (n)$ and of $G$. 
  In particular, each vertex group $\Vtx{m}(\cG) \coloneq \alpha^{-1} (m) \cap \beta^{-1} (m)$ is a closed submanifold of $G$ and this structure turns it into a Lie group.
 \end{lemma}

 \begin{proof}
  The set $\alpha^{-1} (m) \cap \beta^{-1} (n)$ is the preimage of $n$ under the constant rank map $\beta|_{\alpha^{-1} (m)}$.
  As $M$ is a finite dimensional manifold we can apply Gl{\"o}ckner's constant rank theorem \cite[Theorem F]{hg2015}.
  Thus $\alpha^{-1} (m) \cap \beta^{-1} (n)$ is a split submanifold of finite codimension in $\alpha^{-1} (m)$.
  Moreover, $\alpha^{-1} (m)$ is a split submanifold of $G$ (by the regular value theorem \cite[Theorem D]{hg2015}) and since $M$ is finite-dimensional, $\alpha^{-1} (m)$ is even of finite-codimension in $G$.
  Thus \cite[Lemma 1.4]{hg2015} implies that also $\alpha^{-1} (m) \cap \beta^{-1} (n)$ is a split submanifold (of finite-codimension) of $G$.
    Analogously one shows that $\alpha^{-1} (m) \cap \beta^{-1} (n)$ is a split submanifold (of finite-codimension) of $\beta^{-1} (n)$.
  
  Groupoid multiplication and inversion induce a group structure on $\Vtx{m}(\cG)$.
  By the universal property of the pullback, the inclusion $\Vtx{m}(\cG) \times \Vtx{m}(\cG) \rightarrow G \times_{\alpha,\beta} G$ is smooth, whence this group structure turns $\Vtx{m}(\cG)$ into a Lie group.
 \end{proof}  
 
 \newpage
 
  \section{Locally convex manifolds, Lie groups and spaces of smooth maps}\label{Appendix:
   MFD}
   In this appendix we collect the necessary background on the theory of manifolds and Lie groups
   that are modelled on locally convex spaces and how spaces of smooth maps can be
   equipped with such a structure. Let us first recall some basic facts concerning
   differential calculus in locally convex spaces. We follow
   \cite{hg2002a,BertramGlocknerNeeb04Differential-calculus-over-general-base-fields-and-rings}.
   \begin{definition}\label{defn:
   deriv} Let $E, F$ be locally convex spaces, $U \subseteq E$ be an open subset,
   $f \colon U \rightarrow F$ a map and $r \in \N_{0} \cup \{\infty\}$. If it
   exists, we define for $(x,h) \in U \times E$ the directional derivative
   $$df(x,h) \coloneq D_h f(x) \coloneq \lim_{t\rightarrow 0} t^{-1} (f(x+th) -f(x)).$$
   We say that $f$ is $C^r$ if the iterated directional derivatives
   \begin{displaymath}
   d^{(k)}f (x,y_1,\ldots , y_k) \coloneq (D_{y_k} D_{y_{k-1}} \cdots D_{y_1}
   f) (x)
   \end{displaymath}
   exist for all $k \in \N_0$ such that $k \leq r$, $x \in U$ and
   $y_1,\ldots , y_k \in E$ and define continuous maps
   $d^{(k)} f \colon U \times E^k \rightarrow F$. If $f$ is $C^\infty$ it is also
   called smooth. We abbreviate $df \coloneq d^{(1)} f$.
   From this definition of smooth map there is an associated concept of locally
   convex manifold, i.e., a Hausdorff space that is locally homeomorphic to open
   subsets of locally convex spaces with smooth chart changes. Accordingly, a locally convex Lie group is a manifold, equipped with a group structure such that all group operations are smooth. See
   \cite{Wockel13Infinite-dimensional-and-higher-structures-in-differential-geometry,neeb2006,hg2002a}
   for more details.
   \end{definition}

   \begin{definition}\label{defn:
   conno} Let $M$ be a smooth manifold. Then $M$ is called \emph{Banach} (or
   \emph{Fr\'echet}) manifold if all its modelling spaces are Banach (or
   Fr\'echet) spaces. The manifold $M$ is called \emph{locally metrisable} if the
   underlying topological space is locally metrisable (equivalently if all
   modelling spaces of $M$ are metrisable). It is called \emph{metrisable} if it
   is metrisable as a topological space (equivalently locally metrisable and
   paracompact).
   \end{definition}

 \begin{definition}\label{def:regularity}
  Let $H$ be a Lie group modelled on a locally convex space, with identity
  element $\one$, and $r\in \N_0\cup\{\infty\}$. We use the tangent map of the
  right translation $\rho_h\colon H\to H$, $x\mapsto xh$ by $h\in H$ to define
  $v.h\coloneq T_{\one} \rho_h(v) \in T_h H$ for $v\in T_{\one} (H) \equalscolon \Lf (H)$.
  Following \cite{HGRegLie15}, $H$ is called
  \emph{$C^r$-semiregular} if for each $C^r$-curve
  $\gamma\colon [0,1]\rightarrow \Lf(H)$ the initial value problem
  \begin{displaymath}
   \begin{cases}
   \eta'(t)&= \gamma(t).\eta(t)\\ \eta(0) &= \one
   \end{cases}
  \end{displaymath}
  has a (necessarily unique) $C^{r+1}$-solution
  $\Evol (\gamma)\coloneq\eta\colon [0,1]\rightarrow H$. 
  If in addition the map
  \begin{displaymath}
   \evol \colon C^r([0,1],\Lf(H))\rightarrow H,\quad \gamma\mapsto \Evol
   (\gamma)(1)
  \end{displaymath}
  is smooth, we call $H$ \emph{$C^k$-regular}. 
 \end{definition}

 \begin{remark}
  If $H$ is $C^r$-regular and $r\leq s$, then $H$ is also $C^s$-regular. A
  $C^\infty$-regular Lie group $H$ is called \emph{regular} \emph{(in the sense
  of Milnor}) -- a property first defined in
  \cite{Milnor84Remarks-on-infinite-dimensional-Lie-groups}. Every finite
  dimensional Lie group is $C^0$-regular (cf. \cite{neeb2006}). Several
  important results in infinite-dimensional Lie theory are only available for
  regular Lie groups (see
  \cite{Milnor84Remarks-on-infinite-dimensional-Lie-groups}, \cite{neeb2006},
  \cite{hg2015}, cf.\ also \cite{conv1997} and the references therein).
 \end{remark}

 \begin{lemma}\label{lem: semisub:reg}
  Let $G$ be a $C^k$-regular Lie group for $k \in \N_0 \cup \{\infty\}$ and $H$ be a closed Lie subgroup of $G$.
  If $H$ is $C^k$-semiregular, then $H$ is $C^k$-regular.
 \end{lemma}

 \begin{proof}
   Denote by $i_H \colon H \rightarrow G$ and $I_H \colon C^{k+1}([0,1] ,H) \rightarrow C^{k+1} ([0,1], G)$ the canonical inclusions. 
   Then $\Lf (i_H) \colon \Lf (H) \rightarrow \Lf (G)$ allows us to identify curves $\eta \in C^k([0,1],\Lf (H))$ with $C^k$-curves $\Lf (i_H) \circ \eta$ in $\Lf (G)$.
   As $G$ and $H$ are $C^k$-semiregular, we obtain maps $\Evol_J \colon C^k ([0,1],\Lf (J)) \rightarrow C^{k+1} ([0,1], J)$ for $J \in \{G,H\}$ which map a curve to the solution of the initial value problem 
     \begin{displaymath}
      \begin{cases}
       \gamma' (t) &= \gamma(t).\eta (t) \quad \forall t \in [0,1] \\
       \gamma (0) &= \one 
       \end{cases}
     \end{displaymath}
   in the respective group. 
   Consider the map $\Lf (i_H)_* \colon C^k ([0,1] , \Lf (H)) \rightarrow C^k ([0,1], \Lf (\Bis (\cG))), \eta \mapsto \Lf (i_H) \circ \eta$, which is smooth by \cite[Lemma 1.2]{GN2012}.
   By \cite[1.16]{HGRegLie15} we have 
     \begin{displaymath}
      \Evol_G \circ \Lf (i_H) = I_H \circ \Evol_H 
     \end{displaymath}
   As $H$ is a closed subgroup of $G$, the same holds for $C^{k+1} ([0,1],H) \subseteq C^{k+1} ([0,1] , G)$ (cf.\ \cite[1.8 and 1.10]{HGRegLie15})
   Hence $\Evol_H$ is smooth if $I_H \circ \Evol_H$ is smooth.
   Observe that since $G$ is $C^k$-regular, $\Evol_G$ and thus $\Evol_G \circ \Lf (i_H)_* = I_H \circ \Evol_H$ is smooth.
   We deduce that $\Evol_H$ is smooth and \cite[Lemma 3.1]{HGRegLie15} shows that $H$ is $C^k$-regular.
 \end{proof}

  \begin{lemma}\label{lem: Lmorph:subgp}
  Suppose $H,K$ are Lie groups with Lie subgroups $H_{*}\leq H$, $K_{*}\leq K$ such that $K$ and $K_{*}$ are regular and $H_{*}$ is connected.
  Let $\varphi\from H\to K$ be a
  morphism of Lie groups. If $\Lf(\varphi)(\Lf(H_{*}))\se \Lf(K_{*})$, then
  $\varphi(H_{*})\se K_{*}$.\footnote{The authors believe that this result is well-known to experts in the field but were unable to locate a reference.}
 \end{lemma}

 \begin{proof}
  Since $H_{*}$ is connected, it is contained in the identity component of $H$.
  Restricting to that component we may assume that also $H$ is connected. Then
  \cite[Proposition II.4.1]{neeb2006} implies that $\Lf(\varphi)$ has at most
  one integration to a morphism of Lie groups, which thus has to coincide with
  $\varphi$. 
   Recall that the left logarithmic derivative $\delta^l (\gamma) \colon [0,1] \rightarrow \Lf (H)$ of a $C^1$-curve $\gamma \colon [0,1] \rightarrow H$ in a Lie group is defined via $\delta^l (\gamma) (t) \coloneq \gamma(t)^{-1}.\gamma'(t)$.
  Thus $\Evol (\eta) = \gamma$ if and only if $\delta^l (\gamma) = \eta$ and $\gamma (0) = \one_H$.
  The integration of $\Lf(\varphi)$ is constructed by taking a smooth
  path $\gamma\from [0,1]\to H$ with $\gamma(0)=\one_{H}$, applying $\Lf(\varphi)$
  to $\delta^{l}(\gamma)$, solving the initial
  value problem
  \begin{equation}\label{eq: reg:int}
   \begin{cases}
   \delta^{l}(\eta) &=\Lf(\varphi)\circ \delta^{l}(\gamma) \\
   \eta(0) &=\one_{K}
   \end{cases}
  \end{equation}
  in $K$ and setting $\varphi(\gamma(1))=\eta(1)$ (which is possible since $K$
  is regular, cf.\ \cite[Theorem
  8.1]{Milnor84Remarks-on-infinite-dimensional-Lie-groups}). Under the
  assumptions made, this only depends on $\gamma(1)$.
 
  Since $H_{*}$ is assumed to be connected, we can choose for each $h\in H_{*}$
  a path $\gamma$ with $\gamma(0)=e_{H}$ and $\gamma(1)=h$ such that
  $\gamma(t)\in H_{*}$ for all $t\in [0,1]$. Consequently,
  $\delta^{l}(\gamma)(t)\in \Lf(H_{*})$, and thus
  $\Lf(\varphi) (\delta^{l}(\gamma)(t))\in \Lf(K_{*})$ for all $t\in[0,1]$. Now
  $K_{*}$ is regular and $\Lf(\varphi) \circ\delta^{l}(\gamma)$ takes its image
  in $\Lf(K_{*})$ by assumption. Thus we can solve \eqref{eq: reg:int} in $K_{*}$. From
  \cite[Lemma 10.1]{HGRegLie15}, we deduce that the solution to the initial
  value problem \eqref{eq: reg:int} for $\eta$ in $K$ coincides with the
  solution in $K_{*}$, and thus takes its values in $K_{*}$. Summing up,
  $\varphi(h)=\eta(1)\in K_{*}$ for each $h\in H_{*}$.
 \end{proof}

   \begin{definition}\label{def:local_addition}
   Suppose $M$ is a smooth manifold. Then a \emph{local addition} on $M$ is a
   smooth map $\A\from U\opn TM\to M$, defined on an open neighbourhood $U$ of
   the submanifold $M\se TM$ such that
   \begin{enumerate}
   \item \label{def:local_addition_a} $\pi\times \A\from U\to M\times M$,
   $v\mapsto (\pi(v),\A(v))$ is a diffeomorphism onto an open
   neighbourhood of the diagonal $\Delta M\se M\times M$ and
   \item \label{def:local_addition_b} $\A(0_{m})=m$ for all $m\in M$.
   \end{enumerate}
   We say that $M$ \emph{admits a local addition} if there exist a local addition
   on $M$.
   \end{definition}
  
   \begin{definition}(cf.\
  \cite[10.6]{michor1980}) Let $s\from Q\to N$ be a surjective submersion. Then
  a \emph{local addition adapted to $s$} is a local addition
  $\A \from U\opn TQ\to Q$ such that the fibres of $s$ are additively closed
  with respect to $\A$, i.e.\ $\A (v_{q})\in s^{-1}(s(q))$ for all $q\in Q$ and
  $v_{q}\in T_{q}s^{-1}(s(q))$ (note that $s^{-1}(s(q))$ is a submanifold of
  $Q$).
 \end{definition}
    \begin{tabsection}
   An important tool will be the following excerpt from \cite[Theorem
   7.6]{Wockel13Infinite-dimensional-and-higher-structures-in-differential-geometry}.
   \end{tabsection}
   \begin{theorem}\label{thm:
   MFDMAP} Let $M$ be a compact manifold and $N$ be a locally convex and locally
   metrisable manifold that admits a local addition $\A\from U\opn TN\to N$. Set
   $V:=(\pi\times \A)(U)$, which is an open neighbourhood of the diagonal
   $\Delta N$ in $N\times N$. For each $f\in C^{\infty}(M,N)$ we set
   \begin{equation*}
   O_{f}\coloneq\{g\in C^{\infty}(M,N)\mid (f(x),g(x))\in V \}.
   \end{equation*}
   Then the following assertions hold.
   \begin{enumerate}
   \item \label{thm:manifold_structure_on_smooth_mapping_a} The set $O_{f}$
   contains $f$, is open in $C^{\infty}(M,N)$ and the formula
   $(f(x),g(x))=(f(x),\A(\varphi_{f}(g)(m)))$ determines a homeomorphism
   \begin{equation*}
   \varphi_{f}\from  O_{f}\to \{h\in C^{\infty}(M,TN)\mid \pi(h(x))=f(x)\}\cong \Gamma(f^{*}(TN))
   \end{equation*}
   from $O_{f}$ onto the open subset
   $\{h\in C^{\infty}(M,TN)\mid \pi(h(x))=f(x)\}\cap C^{\infty}(M,U)$ of
   $\Gamma(f^{*}(TN))$.
   \item \label{thm:manifold_structure_on_smooth_mapping_b} The family
   $(\varphi_{f}\from O_{f}\to \varphi_{f}(O_{f}))_{f\in C^{\infty}(M,N)}$
   is an atlas, turning $C^{\infty}(M,N)$ into a smooth locally convex and
   locally metrisable manifold.
   The manifold structure is independent of the
   choice of the local addition.
   \item \label{thm:manifold_structure_on_smooth_mapping_d} If $L$ is another
   locally convex and locally metrisable manifold, then a map
   $f\from L\times M\to N$ is smooth if and only if
   $f^{\wedge}\from L\to C^{\infty}(M,N)$ is smooth. In other words,
   \begin{equation*}
   C^{\infty}(L\times M,N)\to C^{\infty}(L,C^{\infty}(M,N)), \quad f\mapsto f^{\wedge}
   \end{equation*}
   is a bijection (which is even natural).
   \end{enumerate}
   \end{theorem}

  \section*{Acknowledgements}

\begin{tabsection}
 The authors would like to thank Helge Gl\"{o}ckner for various stimulating
 discussions about submersion properties for maps between infinite-dimensional
  manifolds and Friedrich Wagemann for help with Lemma \ref{lem:injective}.
 The research on this paper was partially supported by the
 DFG Research Training group 1670 \emph{Mathematics inspired by String Theory
 and Quantum Field Theory}, the Scientific Network \emph{String Geometry} (DFG
 project code NI 1458/1-1) and the project \emph{Topology in Norway} (Norwegian
 Research Council project 213458).
\end{tabsection}

 \end{document}